\documentclass[12pt]{amsart}

\usepackage{a4wide}
\usepackage{amsmath, amssymb}
\usepackage{amsthm}
\usepackage[utf8]{inputenc}
\usepackage{subfigure}
\usepackage{enumerate}
\usepackage{cancel}
\usepackage{xspace}
\usepackage{color}
\usepackage{float}

\usepackage{wrapfig}

\usepackage{epsfig}

\usepackage{tikz}
\usetikzlibrary{patterns}

\usepackage{algorithm}
\usepackage{algorithmic}

\newcommand{\demiinf}[1]{\lfloor {#1}/2 \rfloor}

\def\T{{\mathcal{T}}}
\def\F{{\mathcal{F}}}

\newcommand{\B}{\ensuremath{\mathcal{B}}}
\renewcommand{\P}{\ensuremath{\mathcal{P}}}

\def\TT{\hat{{\mathcal{T}}}}
\def\FF{\hat{{\mathcal{F}}}}
\newcommand{\BB}{\ensuremath{\hat{\mathcal{B}}}}
\newcommand{\PP}{\ensuremath{\hat{\mathcal{P}}}}

\newcommand{\code}{\ensuremath{\phi}}

\newcommand{\phiPP}{\ensuremath{\hat{\Phi}_{\P}}}
\newcommand{\phiPbar}{\ensuremath{\bar{\Phi}_{\P}}}
\newcommand{\phiBB}{\ensuremath{\hat{\Phi}_{\B}}}
\newcommand{\phiFF}{\ensuremath{\hat{\Phi}_{\F}}}

\def\ie{{\em i.e., }}

\newcommand{\triples}{triples of non-intersecting paths\xspace}

\def\SYoung#1{\vbox{\smallskip\offinterlineskip
    \halign{&\vbox{##}\kern-\SThickness\cr #1}}}

\newdimen\SSquaresize \SSquaresize=4.5pt
\newdimen\SThickness \SThickness=.15pt
\newdimen\SCorrection \SCorrection=7pt

\def\SCarre#1{\hbox{\vrule width \SThickness
   \vbox to \SSquaresize{\hrule height \SThickness\vss
      \hbox to \SSquaresize{\hss$\scriptstyle#1$\hss}
   \vss\hrule height\SThickness}
   \unskip\vrule width \SThickness}
   \kern-\SThickness}

\newcommand{\Th}
{
\begin{array}{c}
\SYoung{
\SCarre{\cdot}&\SCarre{\bullet}&\SCarre{}\cr
\SCarre{\bullet}&\SCarre{}&\SCarre{\bullet}\cr}
\end{array}
}

\newcommand{\Tv}
{
\begin{array}{c}
\SYoung{
\SCarre{\cdot}&\SCarre{\bullet}\cr
\SCarre{\bullet}&\SCarre{}\cr
\SCarre{}&\SCarre{\bullet}\cr}
\end{array}
}

\newcommand{\Tc}
{
\begin{array}{c}
\SYoung{
\SCarre{\cdot}&\SCarre{\bullet}\cr
\SCarre{\bullet}&\SCarre{\bullet}\cr}
\end{array}
}

\newcommand{\Fp}{\begin{array}{c}\includegraphics[scale=0.7]{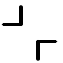}\end{array}}

\newcommand{\THaut}{\begin{array}{c}\includegraphics[scale=1.0]{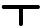}\end{array}}
\newcommand{\TBas}{\begin{array}{c}\includegraphics[scale=1.0]{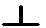}\end{array}}
\newcommand{\TGauche}{\begin{array}{c}\includegraphics[scale=1.0]{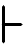}\end{array}}
\newcommand{\TDroit}{\begin{array}{c}\includegraphics[scale=1.0]{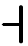}\end{array}}

\DeclareMathOperator\iso{iso}
\DeclareMathOperator\ins{rib} %Former notation was \DeclareMathOperator\ins{ins}

\newcommand{\scaleBTLT}{0.75}
\newcommand{\scaleExplPFP}{0.9}
\newcommand{\scalePFP}{0.75}
\newcommand{\PFPsegment}[4]{$[(#1,#2),(#3,#4)]$}
\newcommand{\scaleBaxterTLTAscentDescent}{0.65}
\newcommand{\scaleBaxterTLTAscentDescentFormsSitua}{0.75}
\newcommand{\scaleNILP}{0.8}

\newtheorem{theorem}{Theorem}
\newtheorem{proposition}[theorem]{Proposition}
\newtheorem{lemma}[theorem]{Lemma}
\newtheorem{corollary}[theorem]{Corollary}
\newtheorem{observation}[theorem]{Observation}
\newtheorem{definition}[theorem]{Definition}

\title{Baxter Tree-like Tableaux}
\author[J.-C. Aval]{Jean-Christophe Aval}
       \address[JCA]{LaBRI - CNRS, Universit\'e de Bordeaux, 351 Cours de la Lib\'eration, 33405 Talence, France.}
       \email{jean-christophe.aval.1@u-bordeaux.fr}

\author[A. Boussicault]{Adrien Boussicault}
       \address[AB]{LaBRI - CNRS, Universit\'e de Bordeaux, 351 Cours de la Lib\'eration, 33405 Talence, France.}
       \email{boussica@labri.fr}
       
 \author[M. Bouvel]{Mathilde Bouvel}
       \address[MB]{Institut für Mathematik, Universität Zürich, Winterthurerstr. 190, CH-8057 Zürich, Switzerland, and \\
       Université de Lorraine, CNRS, Inria, LORIA, F-54000 Nancy, France}
       \email{mathilde.bouvel@loria.fr}

\author[O. Guibert]{Olivier Guibert}
       \address[OG]{LaBRI - CNRS, Universit\'e de Bordeaux, 351 Cours de la Lib\'eration, 33405 Talence, France.}
       \email{olivier.guibert@labri.fr}
       
\author[M. Silimbani]{Matteo Silimbani}
       \address[MS]{LaBRI - CNRS, Universit\'e de Bordeaux, 351 Cours de la Lib\'eration, 33405 Talence, France.}  


\begin{document}

\begin{abstract}
Tree-like tableaux are objects in bijection with alternative or permutation tableaux.
They have been the subject of a fruitful combinatorial study for the past few years.
In the present work, we define and study a new subclass of tree-like tableaux enumerated by Baxter numbers.
We exhibit simple bijective links between these objects and three other combinatorial classes: 
(packed or mosaic) floorplans, twisted Baxter permutations and triples of non-intersecting lattice paths. 
From several (and unrelated) works, these last objects are already known to be enumerated by Baxter numbers, 
and our main contribution is to provide a unifying approach to bijections between Baxter objects,
where Baxter tree-like tableaux play the key role.
We moreover get new enumerative results about alternating twisted Baxter permutations. 
Finally, we define a new subfamily of floorplans, which we call alternating floorplans, 
and we enumerate these combinatorial objects.
\end{abstract}

\maketitle

%%%%%%%%%%%%%%%%%%%%
\section{Introduction}

Baxter permutations are named after the mathematician Glen\footnote{Not to be confused with the physicist Rodney Baxter.} 
Baxter \cite{GBaxter}, who introduced them in 1964 (in an analysis context).
They are enumerated by {\em Baxter numbers} \cite[sequence \textsc{a001181}]{oeis}, 
whose formula was obtained by  Chung {\it et al.} \cite{CGHK78} 
(see also \cite{Viennot_SLC81} for a combinatorial proof):
\begin{equation*}
Bax_n = \frac{2}{n(n+1)^2} \sum_{k=1}^n {{n+1} \choose {k-1}} {{n+1} \choose {k}} {{n+1} \choose {k+1}}.
\end{equation*}
Since then, Baxter numbers have appeared to enumerate various classes of combinatorial objects:
pairs of twin binary trees \cite{DG96}, 
several kinds of standard Young tableaux with three rows \cite{DG,CEF}, 
plane bipolar orientations \cite{RBaxter,BBF}, 
and three other classes that we shall present in more details, as they play important parts in our work. 

The first one is the class of {\em twisted Baxter permutations}.
These permutations were defined a few years ago by Reading in an algebraic context \cite{reading05}: 
they naturally index bases for subalgebras of the Malvenuto-Reutenauer Hopf algebra of permutations.
Like Baxter permutations, twisted Baxter permutations may be characterized by pattern avoidance,
and West proved \cite{West} (by a recursive bijection) that they are enumerated by Baxter numbers.
These objects are also endowed with a nice Hopf structure, as revealed by the recent works \cite{LR,Gir}.

The second class we are interested in is a class of {\em triples of non-intersecting lattice paths}.
The Lindstr\"om-Gessel-Viennot lemma \cite{lind,GV} relates the enumeration of non-intersecting lattice
paths (NILP) to the computation of determinants of integer matrices. 
Because of that, NILPs are ubiquitous objects that appear in many contexts in combinatorics.
The class we shall consider here is known \cite{DG} to be in bijection with pairs of twin binary trees,
see their precise definition in Section \ref{sec:nilp}.

The third and last class we shall present here is the class of {\em mosaic floorplans}.
The notion of floorplans finds its origin in integrated circuits:
a floorplan encodes the relative positions of modules in a circuit.
A mosaic floorplan may be defined as an equivalence class of some 
rectangular partitions of a rectangle (called floorplans).
They were proved to be enumerated by Baxter numbers \cite{saka},
and a bijection was found with pairs of twin binary trees \cite{yao}.
We introduce here combinatorial objects that we call {\em packed} floorplans: 
they are canonical representatives of mosaic floorplans, 
in the sense that every mosaic floorplan contains exactly one packed floorplan. 

The goal of the present work is to link together these three combinatorial classes 
through the use of new objects that we call {\em Baxter tree-like tableaux}. 
Tree-like tableaux (TLTs) are combinatorial objects introduced in \cite{TLT}
as a new presentation of alternative or permutation tableaux \cite{pos,AT},
and have revealed interesting combinatorial properties \cite{TLT,ana}.
Baxter TLTs are defined in a very simple way 
by avoidance of {\em patterns} (a notion to be defined in Section~\ref{sec:tlt}) in TLTs. 
We mention here that Felsner {\it et.al.} also provide bijective links between combinatorial structures
enumerated by Baxter numbers in their paper \cite{ffno}.
But whereas their work is focused on twin binary trees and leads to Baxter permutations,
the central objects of this present article are Baxter TLTs and our bijections lead to twisted Baxter permutations.

The outline of the paper is as follows.
Section~\ref{sec:tlt} introduces our new class of Baxter tree-like tableaux.
Moreover, we recall in this section the recursive structure of tree-like tableaux,
which has already proved to be the key tool dealing with these objects, 
and will be essential for the work reported here. 
Next, Sections~\ref{sec:floorplans}, \ref{sec:bax} and \ref{sec:nilp}
are respectively devoted to packed floorplans, twisted Baxter permutations and triples of non-intersecting lattice paths:
we define these three combinatorial classes and in each case, we build a simple bijection with Baxter TLTs.
In Section~\ref{sec:special}, we consider the restriction of our construction to {\em alternating} objects. 
This allows us to obtain new combinatorial results, such as the enumeration of alternating twisted Baxter permutations 
(see Corollary~\ref{cor:product_of_Catalan}), and to identify several enumerative questions which remain open.

%%%%%%%%%%%%%%%%%%%%
\section{Baxter tree-like tableaux}
\label{sec:tlt}

\subsection{Tree-like tableaux: definitions and useful tools}

We refer to \cite{TLT} for a detailed study of tree-like tableaux. Here, we shall only recall the main definition and a few important properties.

\begin{definition}[Tree-like tableau]
\label{defi:tlt}
A \emph{tree-like tableau} (TLT) is a Ferrers diagram (drawn in the English notation) 
where each cell is either empty or pointed (\ie occupied by a point), with the following conditions:
\begin{enumerate}
\item the top leftmost cell of the diagram is occupied by a point, called the \emph{root point};
\item for every non-root pointed cell $c$, there exists a pointed cell $p$ either above $c$ in the same column, 
or to its left in the same row, {but not both}; $p$ is called the \emph{parent} of $c$ in the TLT;
\item every column and every row contains at least one pointed cell.
\end{enumerate}
The \emph{size} of a TLT is the number of pointed cells it contains. 
\end{definition}

These objects were named tree-like tableaux because of the underlying tree structure they contain: 
recording the parent relations between the points of a TLT indeed produces a tree, 
whose root is the root point of the TLT. 
In this tree, every internal (\emph{i.e.}, non-leaf) vertex may have either a right child (shown by a horizontal edge), or a left child (shown by a vertical edge), or both. 
We refer to such trees as \emph{binary trees} (although they would more appropriately be called \emph{incomplete binary trees}).
Figure~\ref{fig:tlt} (left) shows an example of TLT, with its underlying binary tree. 
The reader interested in more details about the underlying trees of TLTs may find them in~\cite{ana,TLT}. 

\begin{definition}
A \emph{ribbon} in a TLT $T$ is a set $R$ of cells along the Southeast border of $T$,
that is connected (with respect to edge-adjacency), 
does not contain any $2\times 2$ square, 
and consists only of non-pointed cells.
Moreover it is required that the bottom leftmost cell of $R$ is to the right
of a pointed cell with no cell of $T$ below it,
and that the top rightmost cell of $R$ is below a pointed cell. 
\end{definition}

Figure~\ref{fig:tlt} (right) shows an example of a TLT of size $20$ with a ribbon indicated by shaded cells (of magenta color).
\begin{figure}[ht]
$$
\begin{array}{ccc}
\includegraphics[scale=0.329]{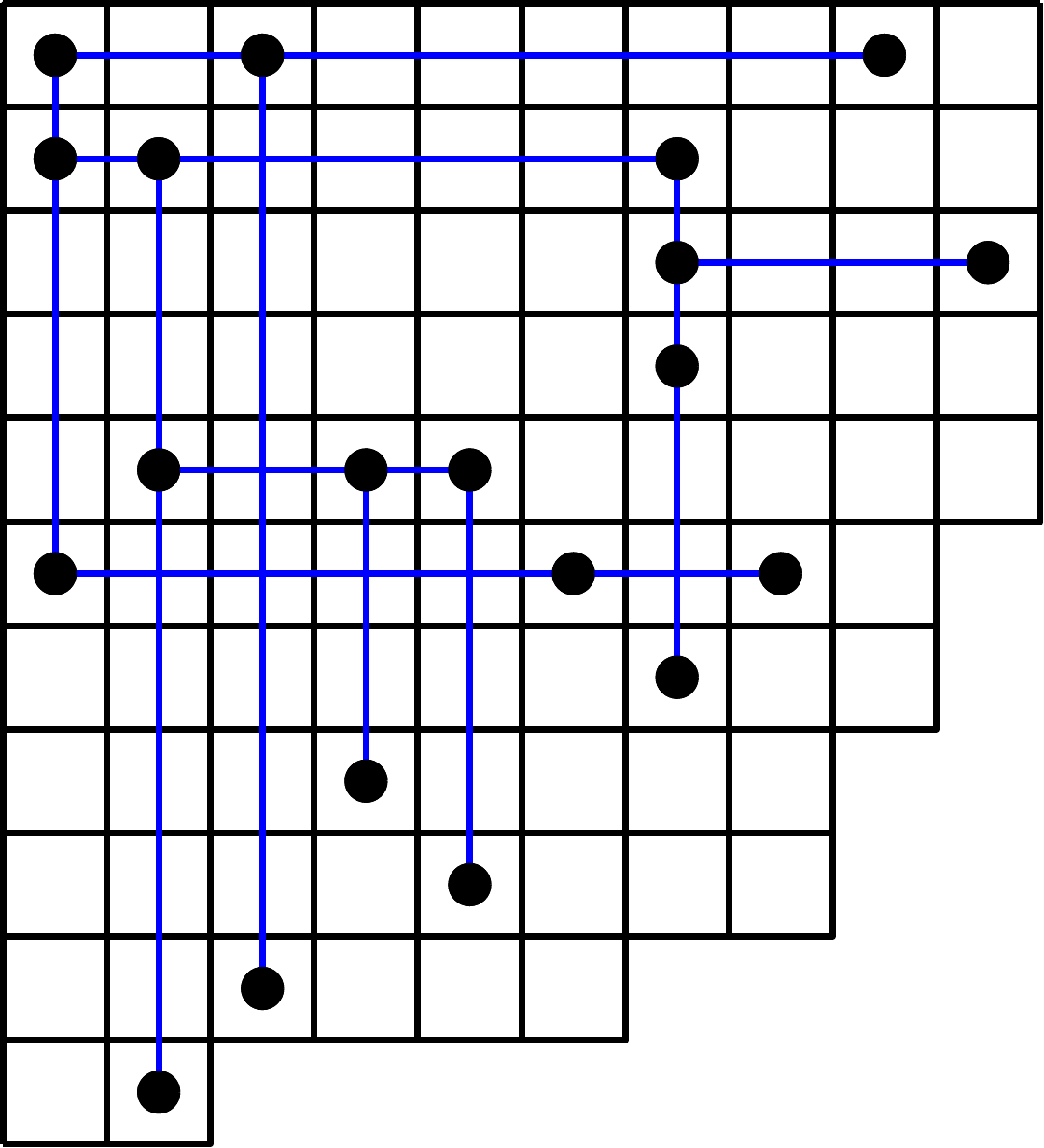} 
& \includegraphics[scale=0.7]{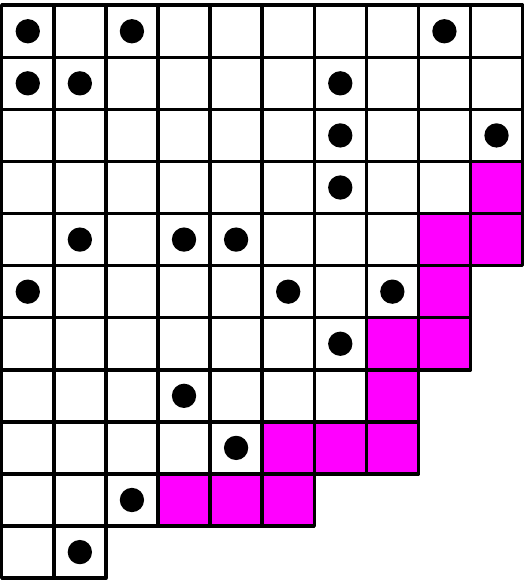} 
\end{array}
$$
\caption{Left: A tree-like tableau $T$ of size $20$, with its underlying binary tree.
Right: The same tree-like tableau $T$ with a ribbon in $T$.}
\label{fig:tlt}
\end{figure}

\medskip

In the article~\cite{TLT} that defines TLTs, the so-called \emph{insertion procedure} \emph{InsertPoint} is defined. 
It allows to generate all TLTs unambiguously from the unique TLT of size $1$ $\begin{tikzpicture}
\begin{scope}[scale=0.25]
\draw (0,0) -- (0,1) -- (1,1) -- (1,0) -- (0,0);
\draw (0.5,0.5) [fill] circle (.25);
\end{scope}
\end{tikzpicture}$ 
by insertion of points (together with a row or column, and possibly a ribbon, of empty cells) 
at the \emph{boundary edges} of TLTs, that is to say at edges of their Southeast border. 
We refer to~\cite{TLT} for details about this insertion procedure, 
and for proofs of statements about it in the remainder of this subsection. 

The reader familiar with \emph{generating trees} (as defined in~\cite{West_GT}) may note that 
this insertion procedure can also be interpreted as representing a generating tree for TLTs 
(although we won't use this fact in the present article). 
Indeed, the main result (Theorem 2.3) of~\cite{TLT} can be interpreted as follows: 
the infinite tree with root $\begin{tikzpicture}
\begin{scope}[scale=0.25]
\draw (0,0) -- (0,1) -- (1,1) -- (1,0) -- (0,0);
\draw (0.5,0.5) [fill] circle (.25);
\end{scope}
\end{tikzpicture}$, 
where all children of a given TLT $T$ are the TLTs obtained 
applying \emph{InsertPoint} on $T$ at each of the boundary edges of $T$, 
is a generating tree for TLTs. 

The insertion procedure on TLTs also induces a canonical labeling of the $n$ points of a TLT $T$ of size $n$
by the integers in $\{1,\dots,n\}$. 
It indicates the (unique) order in which the points of $T$ have been inserted to obtain $T$ from the empty TLT. 
This labeling is essential for the bijections that we define in Sections~\ref{sec:floorplans} and~\ref{sec:bax}, and we review it now. 
Actually, this labeling may alternatively be described as the order in which the points of $T$ should be {\em removed} with 
the procedure \emph{RemovePoint} of~\cite{TLT} to go from $T$ to $
\begin{tikzpicture}
\begin{scope}[scale=0.25]
\draw (0,0) -- (0,1) -- (1,1) -- (1,0) -- (0,0);
\draw (0.5,0.5) [fill] circle (.25);
\end{scope}
\end{tikzpicture}$ (the unique TLT of size $1$), 
and this is how we define it here. 
This is illustrated on Figure~\ref{fig:insertionhistory}. 

\begin{figure}[ht]
\begin{center}
$$
T = \ \raisebox{-.7cm}{\includegraphics[scale=0.8]{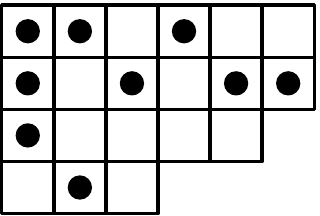}} \ = \ \raisebox{-.7cm}{\includegraphics[scale=0.8]{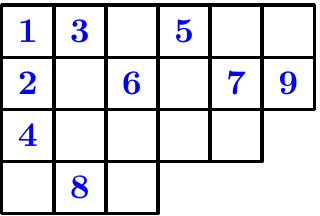}}
$$
$T = T_9 = \ \raisebox{-.7cm}{\includegraphics[scale=0.8]{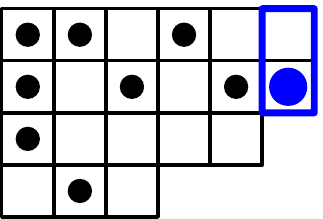}}$,
$T_8 = \ \raisebox{-.7cm}{\includegraphics[scale=0.8]{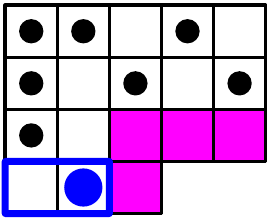}}$,
$T_7 = \ \raisebox{-.4cm}{\includegraphics[scale=0.8]{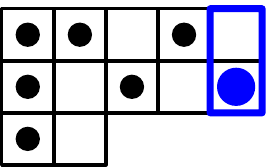}}$, \\ ~ \\~ \\
$T_6 = \ \raisebox{-.7cm}{\includegraphics[scale=0.8]{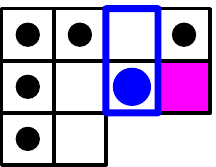}}$,
$T_5 = \ \raisebox{-.7cm}{\includegraphics[scale=0.8]{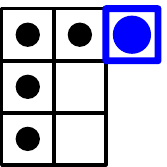}}$,
$T_4 = \ \raisebox{-.7cm}{\includegraphics[scale=0.8]{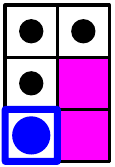}}$,
$T_3 = \ \raisebox{-.4cm}{\includegraphics[scale=0.8]{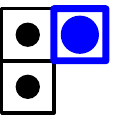}}$,
$T_2 = \ \raisebox{-.4cm}{\includegraphics[scale=0.8]{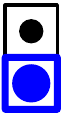}}$,
$T_1 = \ \raisebox{-.2cm}{\includegraphics[scale=0.8]{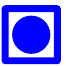}}$.
\end{center}
\caption{The labeling of the points of a TLT using the procedure \emph{RemovePoint} recursively.
The special points together with the associated columns (resp. rows) are denoted in boldface (of blue color),
whereas the ribbons are indicated by shaded cells (of magenta color).
}\label{fig:insertionhistory}
\end{figure}

Consider a TLT $T$ of size $n$. 
We define the \emph{special point} $s$ of $T$ as: the point at the bottom of its column, which is Northeastmost among such points. 
(Notice that $s$ always exists, since the bottom row of $T$ contains at least one point, by definition.) 
This special point $s$ gets the label $n$. 
To label the remaining $n-1$ points of $T$, we compute from $T$ and $s$ another TLT $T'$ of size $n-1$, by removing $s$ and some empty cells in $T$. 
The points of $T'$ are in immediate correspondence with those of $T$ except $s$, so that we may label in $T$ the special point of $T'$ with $n-1$, 
and proceed recursively. 
We will denote by $(T_n=T, T_{n-1}=T', T_{n-2}, \ldots, T_1 = 
\begin{tikzpicture}
\begin{scope}[scale=0.25]
\draw (0,0) -- (0,1) -- (1,1) -- (1,0) -- (0,0);
\draw (0.5,0.5) [fill] circle (.25);
\end{scope}
\end{tikzpicture}
)$ 
the corresponding sequence of TLTs (each $T_i$ having size $i$).

We now explain how to build $T'$ from $T$ and $s$. 
Unless $n=1$, $s$ is not the root point of $T$, and this implies that exactly one of the followings holds: 
either there is no point of $T$ above $s$ in the same column, or there is no point of $T$ to its left in the same row. 
In the former (resp. latter) case, we define the \emph{column} (resp. \emph{row}) \emph{of $s$} to be 
the cells above (resp. to the left of) $s$ in the same column (resp. row). 
If there is a cell adjacent to $s$ on its right, then this cell is empty (by definition of $s$). 
In this case, we claim that there is a ribbon in $T$ to the right of $s$. 
Indeed, this is derived from the two following facts: 
starting from the empty cell to the right of $s$, and
following the Southeast border of $T$, we eventually meet a pointed cell $p$, since the last column of $T$ contains a point; 
and $p$ has been reached from below, since otherwise $s$ would not be the special point. 
We call this set of empty cells the \emph{ribbon of $s$}. 
Now, $T'$ is obtained from $T$ by removing $s$, together with its column (resp. row) and its ribbon (when it exists). 

From now on, when we speak of the ribbons of $T$, 
we mean the ribbons removed 
when applying iteratively the procedure \emph{RemovePoint} from $T$ until $
\begin{tikzpicture}
\begin{scope}[scale=0.25]
\draw (0,0) -- (0,1) -- (1,1) -- (1,0) -- (0,0);
\draw (0.5,0.5) [fill] circle (.25);
\end{scope}
\end{tikzpicture}$ is reached. 

\begin{observation}\label{obs:ribbon}
For a TLT $T$ and two pointed cells $c$ and $c'$ with respective labels $i$ and $j$.
The following assertions are equivalent:
\begin{enumerate}
\item $c$ is (strictly) to the left and below $c'$ and $i=j+1$;
\item there is a ribbon of $T$ between $c$ and $c'$.
\end{enumerate}
\end{observation}

To conclude the general properties of TLTs, we observe a property of the cells of its ribbons.

\begin{definition}
A \emph{crossing} in a TLT $T$ is an empty (\emph{i.e.}, non-pointed) cell such that 
there are pointed cells both above it in the same column and to its left in the same row. 
\end{definition}

This terminology has already been introduced in~\cite{TLT}. 
The choice of the word \emph{crossing} is explained because 
such cells are those where two edges of the underlying binary tree\footnote{The binary trees considered 
in~\cite{TLT} are a slight modification of the ones considered in this paper. 
Specifically, in the present paper, there is no vertical (resp. horizontal) edge leaving a point of a TLT which has no point below it (resp. to its right) -- see Figure~\ref{fig:tlt} (left). 
However, in~\cite{TLT}, there are edges leaving such points, and which extend until the boundary of the TLT. With these additional edges, 
we really ``see'' the crossings, wherever two edges intersect.} of $T$ cross each other. 
It will be used mostly in Sections~\ref{sec:floorplans} and~\ref{sec:bax}, but also on a few occasions before. 

\begin{definition}
\label{dfn:ins}
Let $T$ be a TLT of size $n$, and denote by $(T_n=T, T_{n-1}, T_{n-2}, \ldots, T_1 = 
\begin{tikzpicture}
\begin{scope}[scale=0.25]
\draw (0,0) -- (0,1) -- (1,1) -- (1,0) -- (0,0);
\draw (0.5,0.5) [fill] circle (.25);
\end{scope}
\end{tikzpicture}
)$ 
the sequence of TLTs (each $T_i$ having size $i$) obtained iterating the procedure \it{RemovePoint} starting from $T$.
For any cell $c$ of the ribbon removed from $T_i$ to obtain $T_{i-1}$, we define the label $\ins(c)=i$. 
\end{definition}

This labeling is illustrated on Figure~\ref{fig:ins}. It will be used in Lemma~\ref{lem:crossings_and_2+12}. 

Notice that there are cells with no $\ins$-label. Indeed, we have the following characterization of cells having a $\ins$-label: 

\begin{observation}
A cell has a $\ins$-label if and only if it is a crossing. 
\label{obs:ins}
\end{observation}

\begin{proof}
Note that TLTs have no empty rows nor columns. 
Therefore the definition of ribbons ensures that 
if a cell has a $\ins$-label then it is a crossing. 
Conversely, considering a crossing $c$ and the smallest $i$ such that the cell $c$ belongs to $T_i$, 
we obtain that $c$ belongs to the ribbon removed from $T_i$ to obtain $T_{i-1}$, hence has a $\ins$-label (equal to $i$). 
\end{proof}

\begin{figure}[ht]
\begin{center}
$\begin{array}{c}
	\includegraphics[scale=0.7]{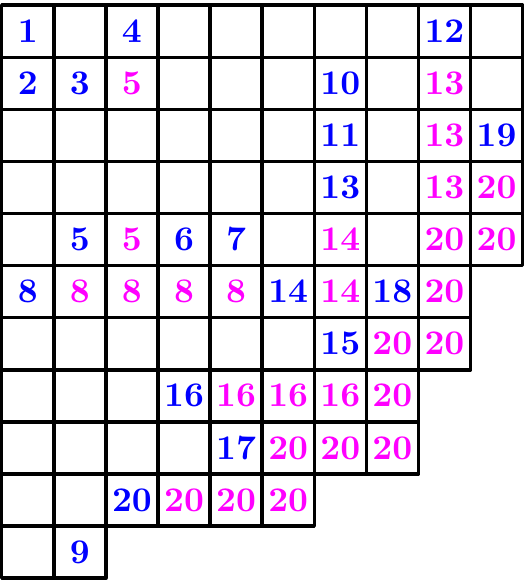}
\end{array}$
\hspace*{2cm}
$\begin{array}{c}
	\includegraphics[scale=0.8]{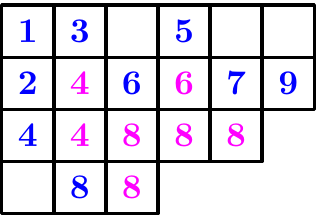}
\end{array}$
\end{center}
\caption{The labeling (of blue color) and the $\ins$-label (of magenta color) of the TLTs of Figures~\ref{fig:tlt} (left) and~\ref{fig:insertionhistory} (right).}
\label{fig:ins}
\end{figure}

\subsection{A family of TLTs enumerated by Baxter numbers}

In this work, we are interested in a family of TLTs restricted by pattern avoidance constraints. 
A TLT $T$ is said to contain the pattern $\Th$ if there exist $2$ rows and $3$ columns
in $T$ such that the restriction of $T$ to the $2\times3=6$ cells at their intersection is equal to $\SYoung{
\SCarre{}&\SCarre{\bullet}&\SCarre{}\cr
\SCarre{\bullet}&\SCarre{}&\SCarre{\bullet}\cr}
$ or $\SYoung{
\SCarre{\bullet}&\SCarre{\bullet}&\SCarre{}\cr
\SCarre{\bullet}&\SCarre{}&\SCarre{\bullet}\cr}
$. 
We define in the same way the pattern $\Tv$. 
With this kind of notation, the condition that every pointed cell in a TLT does not have pointed cells both 
above in the same column and to the left in the same row is expressed by the avoidance of the pattern $\Tc$.

\begin{definition}[Baxter tree-like tableau]
\label{defi:baxTLT}
A \emph{Baxter tree-like tableau} is a TLT which avoids (\ie does not contain any of) the patterns $\Th$ and $\Tv$.
We shall denote by $\T_{(k,\ell)}$ the set of Baxter TLTs with $k$ rows and $\ell$ columns and set: $\T_n=\sqcup_{k+\ell-1=n}\T_{(k,\ell)}$
where $\sqcup$ denotes the disjoint union.
\end{definition}

We may remark that  at each step of the procedure \emph{RemovePoint}, one point is removed and either a row or a column is removed. This implies that the size of a TLT is given by its semi-perimeter$-1$.
As a consequence, the size of any $T\in\T_n$ is $n$. 
Figure~\ref{fig:BTLT_size_4} shows all Baxter TLTs $(T^i)_{1 \leq i \leq 22}$ of size $4$.

We may also note (although we won't use it in this article) that the generating tree for TLTs induced by the procedure \emph{InsertPoint}, 
can be restricted to Baxter TLTs, yielding a generating tree for Baxter TLTs. 
Indeed, the procedure \emph{RemovePoint} applied to any Baxter TLT produces a Baxter TLT again 
(since applying \emph{RemovePoint} cannot create any occurrence of $\Th$ or $\Tv$). 

\begin{figure}[ht]
\begin{align*}
 T^{1} &=& & \hspace*{-1em} \begin{array}{l} \includegraphics[scale=\scaleBTLT]{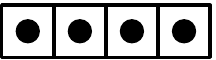} \end{array},  
&  T^{2} &=& & \hspace*{-1em} \begin{array}{l} \includegraphics[scale=\scaleBTLT]{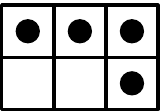} \end{array},  
&  T^{3} &=& & \hspace*{-1em} \begin{array}{l} \includegraphics[scale=\scaleBTLT]{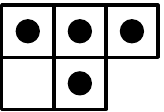} \end{array}, 
&  T^{4} &=& & \hspace*{-1em} \begin{array}{l} \includegraphics[scale=\scaleBTLT]{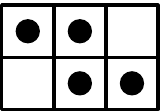} \end{array}, \\
 T^{5} &=& & \hspace*{-1em} \begin{array}{l} \includegraphics[scale=\scaleBTLT]{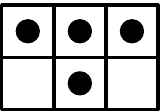} \end{array},  
&  T^{6} &=& & \hspace*{-1em} \begin{array}{l} \includegraphics[scale=\scaleBTLT]{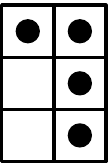} \end{array}, 
&  T^{7} &=& & \hspace*{-1em} \begin{array}{l} \includegraphics[scale=\scaleBTLT]{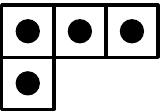} \end{array}, 
&  T^{8} &=& & \hspace*{-1em} \begin{array}{l} \includegraphics[scale=\scaleBTLT]{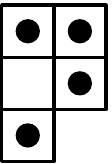} \end{array}, \\
 T^{9} &=& & \hspace*{-1em} \begin{array}{l} \includegraphics[scale=\scaleBTLT]{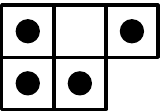} \end{array}, 
&  T^{10} &=& & \hspace*{-1em} \begin{array}{l} \includegraphics[scale=\scaleBTLT]{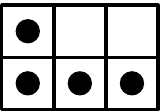} \end{array}, 
&  T^{11} &=& & \hspace*{-1em} \begin{array}{l} \includegraphics[scale=\scaleBTLT]{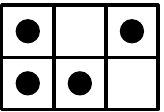} \end{array}, 
&  T^{12} &=& & \hspace*{-1em} \begin{array}{l} \includegraphics[scale=\scaleBTLT]{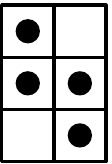} \end{array}, \\
 T^{13} &=& & \hspace*{-1em} \begin{array}{l} \includegraphics[scale=\scaleBTLT]{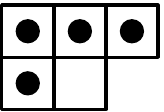} \end{array}, 
&  T^{14} &=& & \hspace*{-1em} \begin{array}{l} \includegraphics[scale=\scaleBTLT]{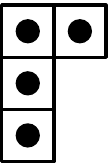} \end{array}, 
&  T^{15} &=& & \hspace*{-1em} \begin{array}{l} \includegraphics[scale=\scaleBTLT]{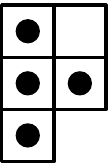} \end{array}, 
&  T^{16} &=& & \hspace*{-1em} \begin{array}{l} \includegraphics[scale=\scaleBTLT]{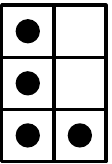} \end{array}, \\
 T^{17} &=& & \hspace*{-1em} \begin{array}{l} \includegraphics[scale=\scaleBTLT]{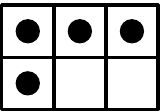} \end{array}, 
&  T^{18} &=& & \hspace*{-1em} \begin{array}{l} \includegraphics[scale=\scaleBTLT]{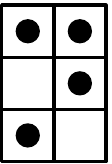} \end{array}, 
&  T^{19} &=& & \hspace*{-1em} \begin{array}{l} \includegraphics[scale=\scaleBTLT]{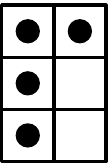} \end{array}, 
&  T^{20} &=& & \hspace*{-1em} \begin{array}{l} \includegraphics[scale=\scaleBTLT]{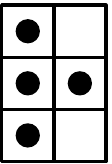} \end{array}, \\
 T^{21} &=& & \hspace*{-1em} \begin{array}{l} \includegraphics[scale=\scaleBTLT]{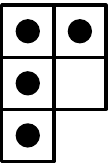} \end{array}, 
&  T^{22} &=& & \hspace*{-1em} \begin{array}{l} \includegraphics[scale=\scaleBTLT]{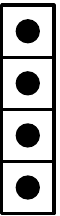} \end{array}.
\end{align*} 
\caption{The 22 Baxter TLTs of size $4$.} \label{fig:BTLT_size_4}
\end{figure}

Sections~\ref{sec:floorplans} to~\ref{sec:nilp} describe size-preserving bijections between Baxter TLTs and 
families of objects that are known to be enumerated by \emph{Baxter numbers}, hence their name. 
Specifically, Section~\ref{sec:floorplans} (resp.~\ref{sec:bax}, resp.~\ref{sec:nilp}) 
describes a bijection denoted $\Phi_{\F}$ (resp. $\Phi_\B$, resp. $\Phi_{\P}$) between Baxter TLTs and 
(packed) floorplans (resp. inverses of twisted Baxter permutations, resp. triples of non-intersecting lattice paths). 
These bijections are illustrated in size $4$ by Figures~\ref{fig:nexamples_pfp_size_4}, \ref{fig:twisted_size_4} and~\ref{fig:NILP_size_4} 
(p.~\pageref{fig:nexamples_pfp_size_4}, \pageref{fig:twisted_size_4} and~\pageref{fig:NILP_size_4}), 
where each TLT $T^i$ of Figure~\ref{fig:BTLT_size_4} is sent to $F^i$, $\sigma^i$ and $\pi^i$ 
by the bijections $\Phi_{\F}$, $\Phi_\B$ and $\Phi_{\P}$, respectively. 

\medskip 

Before moving on to the announced bijections between Baxter TLTs and other Baxter objects,
we make an observation that relates the tree structure of Baxter TLTs and the relative placement of their points. 

\begin{proposition}
Let $T$ be a Baxter TLT. 
Consider the bi-partition $(L,R)$ of the non-root points of $T$, where 
$L$ (resp. $R$) contains all points of $T$ that are in the left (resp. right) subtree pending from the root of the underlying tree of $T$. 
Then all points of $L$ are to the left and below all points of $R$.
\label{prop:decomposition_Baxter_TLT}
\end{proposition}

\begin{proof} 
The proof is by contradiction. 
Assume that there is a point $\ell \in L$ that lies to the right of a point $r$ of $R$. 
Among all ancestors of $r$ (including $r$) in the underlying tree of $T$, there is one which lies to the left of $\ell$ and above $\ell$. 
Indeed, all ancestors of $r$ are to the left of $r$, 
$\ell$ does not lie in the first row of $T$ (since in belongs to $L$), 
and $r$ has at least one non-root ancestor in the first row of $T$ (since it belongs to $R$). 
Denote $r'$ such an ancestor of $r$. We have then that $r'$ is above and to the left of $\ell$. 
Among all ancestors of $\ell$ (including $\ell$), denote by $\ell'$ the one closest to the root of $T$ such that $r'$ is above and to the left of $\ell'$. 
This ensures that $\ell'$ cannot be the root of $T$, and that together with the parent of $\ell'$, $\ell'$ and $r'$ form a pattern $\Th$ or $\Tv$, a contradiction.

Assuming instead that there is a point $\ell \in L$ that lies above of a point $r$ of $R$, 
we derive a contradiction in a symmetric fashion. 
\end{proof}

Several consequences of Proposition~\ref{prop:decomposition_Baxter_TLT} will be useful in proving properties of our bijections. 

\begin{corollary}
Any binary tree is the underlying tree of a unique {\em rectangular Baxter TLT}, 
that is to say a TLT with rectangular shape and which avoids the patterns $\Th$ and $\Tv$.
\label{cor:uniqueNAT}
\end{corollary}

We point out that a detailed study of TLTs with rectangular shapes 
(or alternatively of TLTs where we forget the underlying Ferrers diagram)
is provided in~\cite{ana}, 
where these objects are referred to as \emph{non-ambiguous trees}.

\begin{proof}
Consider a binary tree $B$, and denote by $B_{\ell}$ (resp. $B_r$) the left (resp. right) subtree pending from the root of $B$: 
$B = $ \tikz[level distance=15pt,baseline=-8pt]{ \node{$\bullet$} child {node {$B_\ell$}} child { node {$B_r$}};}. 
By induction (the base case of the induction, which corresponds to a tree with just one vertex, being clear), 
there are unique Baxter TLTs of rectangular shapes, denoted $T_{\ell}$ and $T_r$, whose underlying trees are respectively $B_{\ell}$ and $B_r$. 
We are looking for a Baxter TLT $T$ of rectangular shape whose underlying tree is $B$. 
Proposition~\ref{prop:decomposition_Baxter_TLT} leaves us no choice but to place all points of $T_{\ell}$ below and to the left of all points of $T_r$. 
That the resulting TLT (shown on the left of Figure~\ref{fig:uniqueNAT})
has a rectangular shape is ensured by the construction,
and the avoidance of the two patterns is immediate to prove. 
\end{proof}

Figure~\ref{fig:uniqueNAT} (right) shows an example of rectangular Baxter TLT associated with a binary tree by Corollary~\ref{cor:uniqueNAT}. 

\begin{figure}[ht]
\begin{center}
$\begin{array}{c}\includegraphics[scale=1.0]{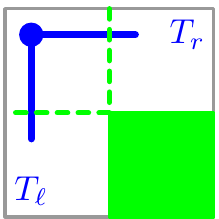}\end{array}$
\qquad \vline \qquad 
$B=\begin{array}{c}
	\includegraphics[scale=0.5]{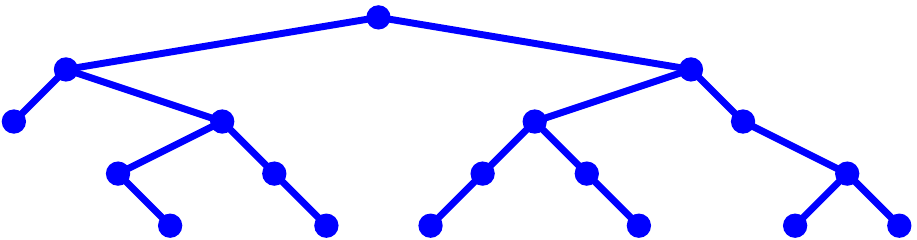}
\end{array}$ 
\quad and \quad 
$T=\begin{array}{c}
	\includegraphics[scale=0.5]{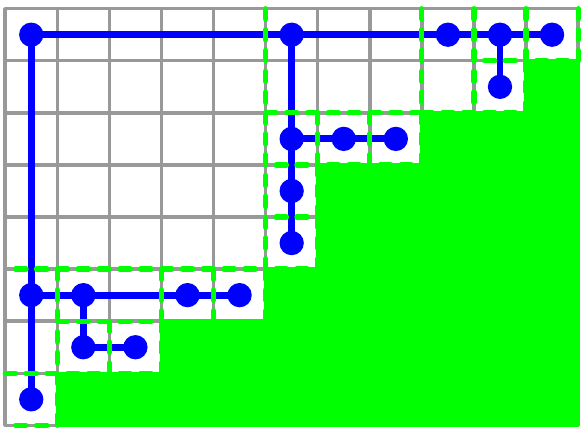}
\end{array}$
\caption{Left: Recursive construction of a TLT with rectangular shape. Right: A binary tree $B$, and the unique Baxter TLT $T$ with rectangular shape 
whose underlying tree is $B$.}
\label{fig:uniqueNAT}
\end{center}
\end{figure}

\begin{corollary}
\label{cor:decomposition_Baxter_TLT}
Let $T$ be a Baxter TLT, and $B$ be its underlying binary tree. 
Denote by $B_\ell$ (resp. $B_r$) the left (resp. right) subtree pending from the root of $B$. 
Consider the bi-partition $(L,R)$ of the non-root points of $T$, where 
$L$ (resp. $R$) contains all points of $T$ that are in $B_\ell$ (resp. $B_r$). 

We can split $T$ with two lines $\mathcal{V}$ and $\mathcal{H}$, uniquely defined by the following conditions: 
$\mathcal{V}$ is a vertical line leaving all points of $L$ to the left and all those of $R$ to the right, 
and $\mathcal{H}$ is a horizontal line leaving all points of $L$ below and all those of $R$ above. 

Provided that both $L$ and $R$ are non-empty, $\mathcal{V}$ and $\mathcal{H}$ split $T$ into four blocks, having the following properties. 
\begin{itemize}
 \item The Northwest block is a rectangle of empty cells, except for the North-westernmost cell, which contains the root of $T$.
 \item The Southwest block is a Baxter TLT, denoted $T_\ell$, whose underlying tree is $B_\ell$. 
 \item The Northeast block is a Baxter TLT, denoted $T_r$, whose underlying tree is $B_r$. 
 \item The Southeast block is a Ferrers diagram, possibly empty, and contains only crossings of $T$. 
\end{itemize}
\end{corollary}

\begin{proof}
The existence of $\mathcal{V}$ and $\mathcal{H}$ is guaranteed by Proposition~\ref{prop:decomposition_Baxter_TLT}. 
Their uniqueness in ensured by the fact that there are no empty rows nor columns in TLTs. 
The properties of the four blocks identified by $\mathcal{V}$ and $\mathcal{H}$ are immediate. 
We just note that the cells in the Southeast block are indeed crossings because they are all empty, 
have a point of $R$ above them, and a point of $L$ to their left (again, because every row and column of $T$ contains at least one point). 
\end{proof}

%%%%%%%%%%%%%%%%%%%%
\section{Bijection with packed floorplans}
\label{sec:floorplans}

\subsection{Packed floorplans: Definition and basic properties}
\label{ssec:def_PFP}

\begin{definition}[Packed floorplans]
\label{defi:PFP}
A {\em packed floorplan} (PFP) of size $(k,\ell)$ is a partition of a 
$k\times\ell$ rectangle (\emph{i.e.} a rectangle of height $k$ and width $\ell$)
into $k+\ell-1$ rectangular tiles 
whose sides have integer lengths 
such that the pattern $\Fp$ is avoided, meaning that:
for every pair of tiles $(t_1,t_2)$, 
denoting $(x_1,y_1)$ the coordinates of the bottom rightmost corner of $t_1$ 
and $(x_2,y_2)$ those of the top leftmost corner of $t_2$, 
it is not possible to have both $x_1 \leq x_2$ and $y_1 \geq y_2$. 

The set of packed floorplans of size $(k,\ell)$ will be denoted by 
$\F_{(k,\ell)}$, and we set: $\F_n=\sqcup_{k+\ell-1=n}\F_{(k,\ell)}$.
\end{definition}

\begin{figure}[htb]
 \centering
 \subfigure[Five packed floorplans.]
   {$\begin{array}{c}
	\includegraphics[scale=\scaleExplPFP]{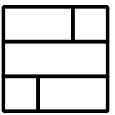}
	\end{array}
	\begin{array}{c}
	\includegraphics[scale=\scaleExplPFP]{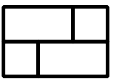}
	\end{array}
	\begin{array}{c}
	\includegraphics[scale=\scaleExplPFP]{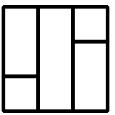}
	\end{array}
	\begin{array}{c}
	\includegraphics[scale=\scaleExplPFP]{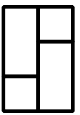}
	\end{array}
	\begin{array}{c}
	\includegraphics[scale=\scaleExplPFP]{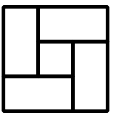}
	\end{array}$
} \qquad 
 \subfigure[These seven are not packed floorplans.]
   {$
   	\begin{array}{c}
		\includegraphics[scale=\scaleExplPFP]{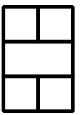}
	\end{array}
	\begin{array}{c}
		\includegraphics[scale=\scaleExplPFP]{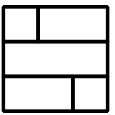}
	\end{array}
	\begin{array}{c}
		\includegraphics[scale=\scaleExplPFP]{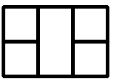}
	\end{array}
	\begin{array}{c}
		\includegraphics[scale=\scaleExplPFP]{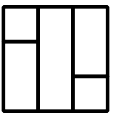}
	\end{array}
	\begin{array}{c}
		\includegraphics[scale=\scaleExplPFP]{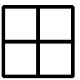}
	\end{array}
	\begin{array}{c}
		\includegraphics[scale=\scaleExplPFP]{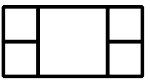}
	\end{array}
	\begin{array}{c}
		\includegraphics[scale=\scaleExplPFP]{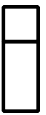}
	\end{array}
	$}
 \caption{Some examples and counterexamples of PFPs.} \label{PFPlist}
 \end{figure}

Some examples and counter-examples of PFPs are provided by Figure~\ref{PFPlist}, 
and Figure~\ref{fig:nexamples_pfp_size_4} shows all the packed floorplans of size $4$. 
These PFPs are new combinatorial objects, but they are in size-preserving bijection with {\em mosaic floorplans}~\cite{AcBaPi,saka,yao}. 
Indeed, as shown in the Appendix (see Proposition~\ref{prop:PFP_and_mosaicFP}), mosaic floorplans are equivalence classes of objects, 
and PFPs are canonical representatives of mosaic floorplans. 

\begin{figure}[htb]
\begin{align*}
 F^{1} &=& & \hspace*{-1em} \begin{array}{l} \includegraphics[scale=\scalePFP]{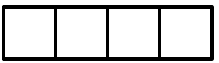} \end{array}, 
& F^{2} &=& & \hspace*{-1em} \begin{array}{l} \includegraphics[scale=\scalePFP]{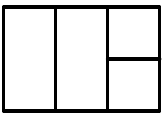} \end{array}, 
& F^{3} &=& & \hspace*{-1em} \begin{array}{l} \includegraphics[scale=\scalePFP]{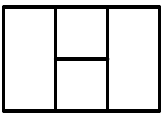} \end{array}, 
& F^{4} &=& & \hspace*{-1em} \begin{array}{l} \includegraphics[scale=\scalePFP]{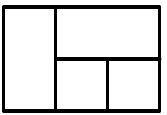} \end{array}, \\
 F^{5} &=& & \hspace*{-1em} \begin{array}{l} \includegraphics[scale=\scalePFP]{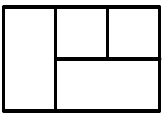} \end{array}, 
& F^{6} &=& & \hspace*{-1em} \begin{array}{l} \includegraphics[scale=\scalePFP]{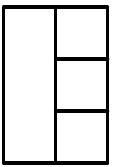} \end{array}, 
& F^{7} &=& & \hspace*{-1em} \begin{array}{l} \includegraphics[scale=\scalePFP]{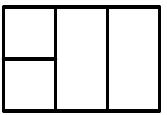} \end{array}, 
& F^{8} &=& & \hspace*{-1em} \begin{array}{l} \includegraphics[scale=\scalePFP]{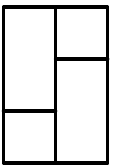} \end{array}, \\
 F^{9} &=& & \hspace*{-1em} \begin{array}{l} \includegraphics[scale=\scalePFP]{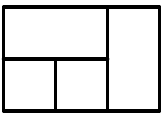} \end{array}, 
& F^{10} &=& & \hspace*{-1em} \begin{array}{l} \includegraphics[scale=\scalePFP]{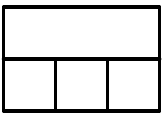} \end{array}, 
& F^{11} &=& & \hspace*{-1em} \begin{array}{l} \includegraphics[scale=\scalePFP]{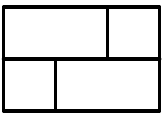} \end{array}, 
& F^{12} &=& & \hspace*{-1em} \begin{array}{l} \includegraphics[scale=\scalePFP]{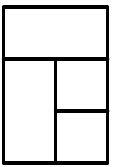} \end{array}, \\
 F^{13} &=& & \hspace*{-1em} \begin{array}{l} \includegraphics[scale=\scalePFP]{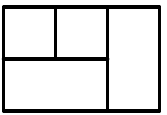} \end{array}, 
& F^{14} &=& & \hspace*{-1em} \begin{array}{l} \includegraphics[scale=\scalePFP]{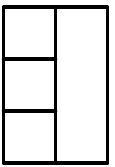} \end{array}, 
& F^{15} &=& & \hspace*{-1em} \begin{array}{l} \includegraphics[scale=\scalePFP]{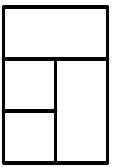} \end{array}, 
& F^{16} &=& & \hspace*{-1em} \begin{array}{l} \includegraphics[scale=\scalePFP]{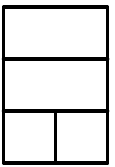} \end{array}, \\
 F^{17} &=& & \hspace*{-1em} \begin{array}{l} \includegraphics[scale=\scalePFP]{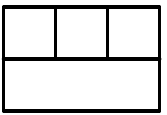} \end{array}, 
& F^{18} &=& & \hspace*{-1em} \begin{array}{l} \includegraphics[scale=\scalePFP]{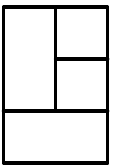} \end{array}, 
& F^{19} &=& & \hspace*{-1em} \begin{array}{l} \includegraphics[scale=\scalePFP]{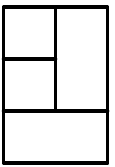} \end{array}, 
& F^{20} &=& & \hspace*{-1em} \begin{array}{l} \includegraphics[scale=\scalePFP]{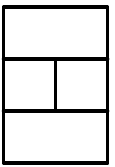} \end{array}, \\
 F^{21} &=& & \hspace*{-1em} \begin{array}{l} \includegraphics[scale=\scalePFP]{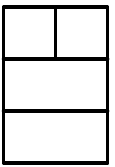} \end{array}, 
& F^{22} &=& & \hspace*{-1em} \begin{array}{l} \includegraphics[scale=\scalePFP]{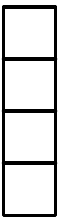} \end{array}. 
\end{align*}
\caption{The 22 packed floorplans of size 4. 
\label{fig:nexamples_pfp_size_4}}
\end{figure}

Some properties of PFPs follow easily from Definition~\ref{defi:PFP}. 
We first introduce notation. 
A \emph{T-junction} in a PFP $F$ is a point where the sides of the tiles of $F$ intersect in one of the following configurations: \label{pagedef:T-junction}
$
\begin{array}{c}
	\includegraphics[scale=\scalePFP]{images/tjunction_1}
\end{array}, 
\begin{array}{c}
	\includegraphics[scale=\scalePFP]{images/tjunction_2}
\end{array}, 
\begin{array}{c}
	\includegraphics[scale=\scalePFP]{images/tjunction_3}
\end{array}$ and $
\begin{array}{c}
	\includegraphics[scale=\scalePFP]{images/tjunction_4}
\end{array}
$. 
A \emph{segment} of a PFP $F$ is a union of sides of tiles of $F$ which forms a segment and which is maximal for this condition.  
Figure~\ref{fig:PFP_segments} shows a PFP and its segments. 
A segment which is not a side of the bounding rectangle is called {\em internal}.

\begin{figure}[ht]
\begin{center}
$\begin{array}{c}
	\includegraphics[scale=1.0]{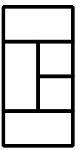}
\end{array}$
\hfill
\begin{minipage}{13cm}
All the horizontal segments of this PFP are (from bottom to top)
\PFPsegment{0}{0}{2}{0}, \PFPsegment{0}{1}{2}{1}, \PFPsegment{1}{2}{2}{2}, \PFPsegment{0}{3}{2}{3} and \PFPsegment{0}{4}{2}{4}. 
All its vertical segments are (from left to right) \PFPsegment{0}{0}{0}{4}, \PFPsegment{1}{1}{1}{3} and \PFPsegment{2}{0}{2}{4}.
\end{minipage}
\end{center}
\caption{A PFP of size $(4,2)$ and its segments.}
\label{fig:PFP_segments}
\end{figure}

\begin{observation}
\label{obs:PFP_obvious}
Let $F$ be a PFP. 
\begin{enumerate}[(i)]
 \item\label{item:PFP_obvious_Tjunction} Every corner of any of the tiles of $F$
either is a corner of the bounding rectangle of $F$ 
or forms a T-junction.
 \item\label{item:PFP_obvious_oneSegment} Every horizontal (resp. vertical) line of integer coordinate 
included in the bounding rectangle of $F$ contains exactly one segment of $F$.
 \item\label{item:PFP_obvious_oneCorner} Every horizontal (resp. vertical) line of integer coordinate 
included in the bounding rectangle of $F$ 
(except the bottom (resp. right) boundary of the bounding rectangle of $F$)
contains the top left corner of at least one tile of $F$. 
\end{enumerate}
\end{observation}

\begin{proof}
For the first item, assume there exists a corner of a tile which neither is a corner of the bounding rectangle of $F$ 
nor forms a T-junction. So, at this corner, either two or four tiles meet. 
In the first (resp. second) case, there exists then two tiles $t$ and $t'$ placed as 
$\begin{array}{c}\includegraphics[scale=0.9]{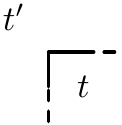} \end{array}$, up to rotation 
(resp. $\begin{array}{c}\includegraphics[scale=0.9]{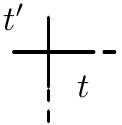}\end{array}$). 
In both cases, we derive a contradiction. 
For the first case, $t$ creates an inner corner in $t'$, hence $t'$ is not of rectangular shape.  
In the second case, the pair of tiles $(t',t)$ forms an occurrence of the pattern $\Fp$, which should be avoided. 

We prove the second item in two steps. 
First, we show that each line contains at most one segment. 
This is clearly true for the boundaries of the bounding rectangle, which are obviously segments themselves. 
Consider an internal line, that is to say a line which is not a boundary of the bounding rectangle, 
and assume it contains at least $2$ segments $s$ and $s'$. 
If $s$ and $s'$ lie on the same horizontal (resp. vertical) line, 
with $s$ to the left of (resp. above) $s'$, then because of $(\ref{item:PFP_obvious_Tjunction})$
the right (resp. bottom) end of $s$ is a T-junction $\TDroit$ (resp. $\TBas$)
and the left (resp. top) end of $s'$ is a T-junction $\TGauche$ (resp. $\THaut$). 
Consider the tile $t$ whose bottom right corner is the right (resp. bottom) end of $s$, 
and the tile $t'$ whose top left corner is the left (resp. top) end of $s'$. 
The pair $(t,t')$ forms a pattern $\Fp$, a contradiction. 

Next, we show that a PFP $F$ of size $(k,\ell)$ contains exactly $k+\ell+2$ segments. 
Since this is also the number of lines considered in $(\ref{item:PFP_obvious_oneSegment})$, 
it follows from our first step that each of them contains \emph{exactly} one segment. 
Denote by $n_c$ (resp. $n_t$, $n_j$, $n_s$) the number of corners of tiles (resp. of tiles, of T-junctions, of segments) of $F$. 
Each tile having $4$ corners, we have $4 n_t = n_c$. 
Also, from item~$(\ref{item:PFP_obvious_Tjunction})$, 
and since there are exactly two corners of tiles at any T-junction, 
$n_c = 2 n_j +4$. It follows that $n_j+2 = 2 n_t$. 
On the other end, every horizontal (resp. vertical) internal segment connects T-junctions of the form $\TGauche$ and $\TDroit$ (resp. $\THaut$ and $\TBas$). 
And every T-junction is an end of a segment. 
It follows that $n_j$ is twice the number of internal segments, which is $n_s-4$ taking into account the boundaries of the bounding rectangle of $F$. 
So, $n_j = 2 (n_s-4)$. 
Combining this equality with $n_j+2 = 2 n_t$ obtained earlier gives $n_t = (n_s-4) +1 = n_s-3$. 
Since $F$ is a PFP of size $(k,\ell)$, it contains $n_t=k+\ell-1$ tiles, and it follows that $n_s=k+\ell+2$ as wanted. 

Finally, the third item follows easily from the second one. 
Indeed, every line as in $(\ref{item:PFP_obvious_oneCorner})$ contains one segment. 
This segment, if horizontal (resp. vertical) is the support of the top (resp. left) side of at least one tile $t$, 
so it contains the top left corner of $t$. 
\end{proof}

\medskip

Our goal in this section is to describe a simple size-preserving bijection between Baxter TLTs and PFPs. 
The definition of our map $\Phi_{\F}$ from $\T_n$ to $\F_n$ is given in Subsection~\ref{subsec:dfn_phiF}. 
All that is needed to define it is the numbering of the points of TLTs induced by the procedure \emph{RemovePoint}, 
and reviewed in Section~\ref{sec:tlt}. 
The proof that $\Phi_{\F}$ is a bijection requires a few more properties of TLTs and PFPs. 
These are presented in Subsection~\ref{subsec:BijPFP}, and allow to describe the inverse of $\Phi_{\F}$, 
proving that it is indeed a bijection.  

\subsection{Definition of $\Phi_{\F}: \T_n \to \F_n$}
\label{subsec:dfn_phiF}

More precisely, we define $\Phi_{\F}:\T_{(k,\ell)} \to \F_{(k,\ell)}$, for all $k,\ell$. 

Let us consider $T\in\T_{(k,\ell)}$. 
As noticed earlier, $T$ contains $n=k+\ell -1$ points. 
These may be labeled by the integers in 
$\{1,\dots,n\}$ according to the insertion procedure of~\cite{TLT}, as explained in Section~\ref{sec:tlt}. 
We shall construct $\Phi_{\F}(T)$ as follows.

We start from a rectangular $k\times\ell$ box. We identify the unit cells of 
this box with the cells of $T$. 
For each label $j=n,\dots,1$, we iteratively add a tile, the largest possible, 
whose top leftmost cell $c$ is the one containing the point labeled by $j$ in $T$ 
(this tile is also said to have label $j$). 
To build this largest possible tile, we only have to draw two segments:
one vertical and one horizontal, each starting from the cell $c$ and going respectively to 
the South and to the East.
We denote the result by $\Phi_{\F}(T)$. See Figure \ref{fig:Phi} for an example.

\newlength{\espPhi}
\setlength{\espPhi}{-0.7cm} 
\begin{figure}[!ht]
$$
\begin{array}{lcr}
T = \raisebox{\espPhi}{\includegraphics[scale=1.0]{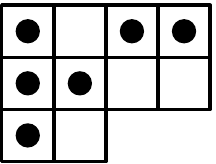}}
= \raisebox{\espPhi}{\includegraphics[scale=1.0]{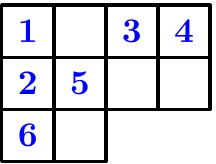}}
\mbox{ \quad and \quad } 
\Phi_{\F}(T) = \raisebox{\espPhi}{\includegraphics[scale=1.0]{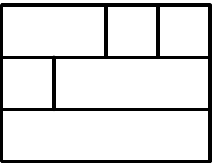}}
\end{array}
$$
\caption{The bijection $\Phi_{\F}$.}
 \label{fig:Phi}
\end{figure}

\begin{proposition}\label{prop:phiF_well_defined}
The mapping $\Phi_{\F}:\T_{(k,\ell)} \to \F_{(k,\ell)}$ is well-defined. 
\end{proposition}

\begin{proof}
Let $T \in \T_{(k,\ell)}$. 
Applying the above described construction, 
we obtain a tiling of a $k \times \ell$ rectangle by $n=k+\ell -1$ tiles. 
So, to check that the mapping $\Phi_{\F}$ is well-defined,
we are just left with checking 
that at each step the tile we construct is actually of
rectangular shape, and that the pattern $\Fp$ is avoided. 

First assume that some tile is not of rectangular shape, \ie has an inside corner (note that 
it has to be an inside SE corner because the tiles are added the largest possible). 
Denote by $q$ the point of $T$ in the top leftmost corner of this tile, 
$c$ the point of $T$ that creates the inside corner, 
and $p$ the parent of $c$ in $T$ (see Figure~\ref{fig:subfig_pattern_proof}$(a)$). 
Denoting $(X(z),Y(z))$ the Cartesian coordinates of any point $z$, 
this means that either $X(c) = X(p)$ and $Y(c) < Y(p)$, 
or $Y(c) = Y(p)$ and $X(c) > X(p)$. 
In the former (resp. latter) case, we claim that $Y(q) < Y(p)$ (resp. $X(p) < X(q)$). 
Indeed, assuming the contrary, 
the insertion procedure of~\cite{TLT} would insert the points $q$, $p$, and $c$ in this order, 
contradicting that $c$ is an inside corner of a tile. 
From these inequalities, we deduce that the points $q$, $p$, and $c$ form 
a pattern $\Tv$ (resp. $\Th$) contradicting that $T \in \T_{(k,\ell)}$.

Suppose now that there are two tiles $t_1$ and $t_2$ that form a pattern $\Fp$. 
Let us choose this pair such that the distance between the bottom rightmost corner of $t_1$ 
and the top leftmost corner of $t_2$ is minimal. 
By construction, there is a point $c$ of $T$ in the top leftmost corner of $t_2$. 
Also, because $t_1$ is constructed as large as possible, there is a point 
$q$ (resp. $u$) of $T$ immediately outside $t_1$ along its bottom (resp. right) border 
(see Figure~\ref{fig:subfig_pattern_proof}$(b)$).

\begin{figure}[ht]
 \centering
 \subfigure[$Y(q) < Y(p)$, or $c$ would not be an inside corner.]
   { 
$\begin{array}{c}
\begin{tikzpicture}
\begin{scope}[scale=0.6]
\draw[dotted] (1.2,0.8) -- (1.2,3);
\draw (0,0) -- (1,0) -- (1,1) -- (3,1) -- (3,2) -- (0,2) -- (0,0);
\draw (1.2,0.8) [fill] circle (.1);
\draw (1.2,3) [fill] circle (.1);
\draw (0.2,1.8) [fill] circle (.1);
\node at (1.5,0.5) {$c$};
\node at (1.5,2.8) {$p$};
\node at (0.4,1.5) {$q$};
% pour avoir une légende longue
\node at (-2,0) {~};
\node at (5,0) {~};
\end{scope}
\end{tikzpicture}
 \end{array}$
   } \qquad 
 \subfigure[$Y(q) < Y(p)$, or $(t_1,t_2)$ would not be minimal.]
{
$\begin{array}{c}
\begin{tikzpicture}
\begin{scope}[scale=0.6]
\draw[dotted] (2.2,0.8) -- (2.2,2.5);
\draw (2,0) -- (2,1) -- (3,1);
\draw (0,2) -- (1,2) -- (1,3);
\draw (2.2,0.8) [fill] circle (.1);
\draw (2.2,2.5) [fill] circle (.1);
\draw (1.2,2.8) [fill] circle (.1);
\draw (0.2,1.8) [fill] circle (.1);
\node at (2.5,0.5) {$c$};
\node at (2.5,2.3) {$p$};
\node at (0.4,1.5) {$q$};
\node at (1.5,2.8) {$u$};
\node at (0,2.5) {$t_1$};
\node at (3.3,0.2) {$t_2$};
% pour avoir une légende longue
\node at (-2,0) {~};
\node at (5,0) {~};
\end{scope}
\end{tikzpicture}
\end{array}$
}
 \caption{Illustrating the proof by contradiction that $\Phi_{\F}(T) \in \F_{(k,\ell)}$. } \label{fig:subfig_pattern_proof}
 \end{figure}
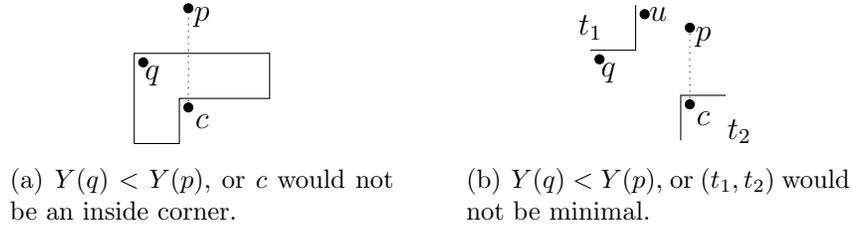

Denote by $p$ the parent of $c$ in $T$. We have $X(c) = X(p)$ and $Y(c) < Y(p)$, 
or $Y(c) = Y(p)$ and $X(c) > X(p)$. 
In the former (resp. latter) case, assume that $Y(p) \leq Y(q)$ (resp. $X(p) \geq X(u)$). 
Then $t_1$ and the tile whose top leftmost corner is $p$ would form a pattern $\Fp$, 
contradicting the minimality of $(t_1,t_2)$. 
Hence, we have $Y(p) > Y(q)$ (resp. $X(p) < X(u)$), 
so that $p,q,c$ (resp. $p,u,c$) form a pattern $\Tv$ (resp. $\Th$), 
contradicting that $T \in \T_{(k,\ell)}$.
\end{proof}

\subsection{Bijection between TLTs and packed floorplans}
\label{subsec:BijPFP}

The goal of the remainder of this section is to prove the following: 

\begin{theorem}\label{thm:FT}
For any $n$, $\Phi_{\F}$ is a bijection between $\T_n$ and $\F_n$.
\end{theorem}

We start with a property on the (tree) structure of packed floorplans.
From now on we will use the term {\em corner} for ``top leftmost corner'' (\ie NW-corner).

\begin{lemma}\label{lemma:tree-PFP}
Let F be a PFP.
The set of (top leftmost) corners of the tiles of $F$ has a tree structure:
for any corner $c$ (different from the top leftmost corner of the bounding box),
there exists another corner $c'$ either above $c$ or to its left, but not both.
\end{lemma}
\begin{proof}
Let us consider the corner $c$ of a tile $t$, different from the top leftmost tile of $F$.
By Observation \ref{obs:PFP_obvious} $(\ref{item:PFP_obvious_Tjunction})$, the corner $c$ has to be either a $\TGauche$ or a $\THaut$. 
In the first case, the tile above $t$ with common left side has its corner above $c$.
If there were a corner $c_1$ to the left of $c$, 
then there would exist another segment supported by the same line as the top edge of $c$, 
in contradiction with Observation \ref{obs:PFP_obvious} $(\ref{item:PFP_obvious_oneSegment})$. 
This proves our statement in the first case, and 
a symmetric argument applies to the second case.
\end{proof}

Next, we define an order on the tiles of PFPs. 
A key step in proving Theorem~\ref{thm:FT} will be to show that this order coincides with the labeling of the points of the corresponding TLT. 

\begin{definition}
Let $F\in\F_n$. 
We define the {\em tile-order} of $F$ as the labeling of the tiles obtained in the following way.
We label the tiles from $n$ to $1$.
After assigning the labels $n,\dots,k+1$, we label with $k$ the tile
which is the rightmost among the unlabeled tiles whose bottom border does not touch any unlabeled tile  
(equivalently, its bottom border touches only labeled tiles or the bottom border of the bounding rectangle). 
\end{definition}

The notion of tile-order is illustrated at Figure \ref{fig:tile-order}.
\begin{figure}[ht]
 \includegraphics[scale=0.15]{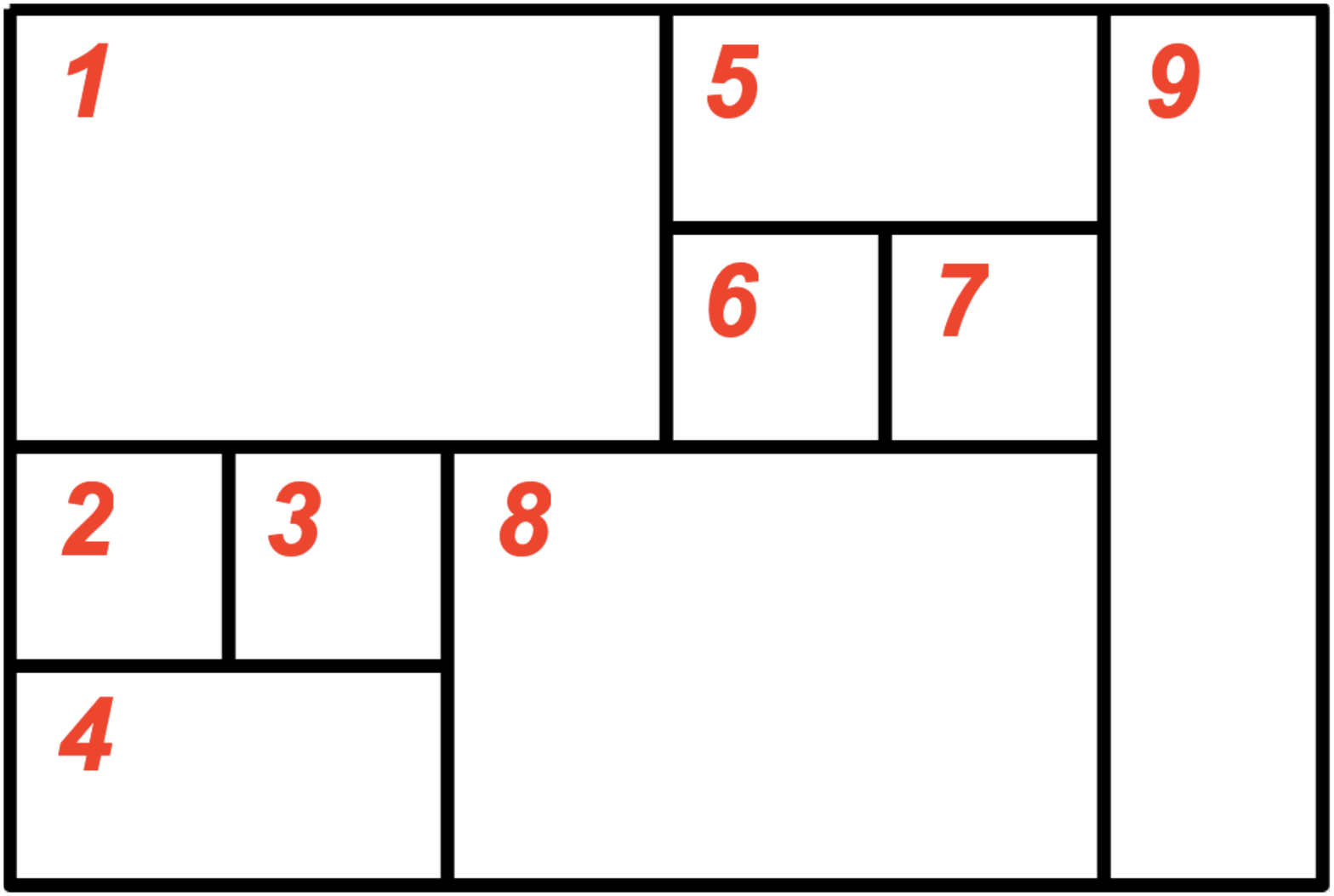} 
\caption{The tile-order on PFPs.}
\label{fig:tile-order}
\end{figure}

We shall now give an alternative description of the labeling of the pointed cells in a (Baxter) TLT, $T$, 
which is easily seen to be equivalent to the one given in Section~\ref{sec:tlt}. 
Indeed, this equivalent definition of the labels of $T$ makes it easier to prove that it coincides with the tile-order of $\Phi_{\F}(T)$.

\begin{definition}
Let $T\in\T_n$. 
We define the {\em point-order} of $T$ as follows.
We label the pointed cells from $n$ to $1$.
After assigning the labels $n,\dots,k+1$, we label with $k$ the point which is the rightmost 
among those with no cells below it.
If there is a cell to its right, we consider the ribbon of cells up to encountering a pointed cell, 
and we declare all cells of this ribbon ``dead''.
And in any case, we also declare ``dead'' the newly labeled pointed cell, and its empty row or column. 
Dead cells should be ignored (\ie treated as if they did not exist) in later iterations of this labeling procedure. 
\end{definition}

\begin{observation}\label{obs:point-order-region}
For a given pointed cell $c$ with label $j$, any pointed cell $c'$ which is weakly to the right and below $c$
has a label $j'>j$.
\end{observation}

We now define an application $\Psi_\F$, which we will prove to be the inverse of $\Phi_\F$.

Let $F\in\F_n$.
We associate to each tile a label $k\in\{1,\dots,n\}$ using the tile-order, 
and we will construct a TLT $T=\Psi_\F(F)$ with the pointed cells at the same positions
as the corners of the tiles of $F$.
Let us denote by $U$ the object under construction.
We start with $U=\emptyset$, 
and for $k=1,\dots,n$:
\begin{enumerate}
\item we add a pointed cell to $U$ at the same position as the corner of the tile labeled by $k$ in $F$;
\item we complete in such a way that the shape of $U$ is still a Ferrers diagram 
(that is we add empty cells to the NW of the added pointed cell);
\item if the pointed cell labeled $k$ is to the left of the pointed cell labeled $k-1$, we place a ribbon
from $k$ to $k-1$.
\end{enumerate}
After dealing with $n$, we let $T:=U$.

Given the insertion of ribbons in the last item above, 
a straightforward induction allows to prove the following observation. 

\begin{observation}\label{obs:last-added-Psi}
In the computation of $\Psi_\F(F)$, after dealing with the tiles labeled $1$ to $j$, 
the latest pointed cell added 
(\emph{i.e.}, the pointed cell with label $j$ for the tile-order) is the rightmost among the pointed cell
without any cell below it.
\end{observation}

\begin{proposition}\label{prop:PsiHasCodomainT}
For any $F\in\F_n$, $\Psi_\F(F)$ is in $\T_n$.
\end{proposition}

\begin{proof}
Let $F\in\F_n$, and $T:=\Psi_\F(F)$. 
By construction, the shape of $T$ is a Ferrers diagram.
Because the position of the pointed cells in $T$ corresponds to the position of the corners in $F$,
Lemma \ref{lemma:tree-PFP} implies that any pointed cell (with the exception of the root) has either a pointed cell above it 
or to its left but not both.
Observation \ref{obs:PFP_obvious}$(\ref{item:PFP_obvious_oneCorner} )$ 
implies that any row or column contains at least one pointed cell. 
Since the sizes clearly match, $T$ is a TLT of size $n$. 

To conclude the proof, it is enough to show that if $T$ contains $\Th$ or $\Tv$, 
then $F$ contains $\Fp$. 
So, assume $T$ contains $\Tv$ (the other case being similar). 
Consider an occurrence of $\Tv$ in $T$ where the distance between the two vertically aligned points of this occurrence is minimal. 
This implies that the topmost of these two points (denoted $p$) is the parent of the bottommost one (denoted $c$). 
By assumption, we know that there exists a point $q$ to the left of $p$ and $c$, 
which lies vertically between these two points. 
Consider the topmost $q$ among all such points. 
Denote by $t_1$ the tile of $F$ whose bottom side is supported by the same segment as the top side of the tile containing $q$. 
And denote by $t_2$ the tile containing $c$. 
This situation is displayed in Figure~\ref{fig:t1t2forbiddenpattern}. 
Then, $t_1$ and $t_2$ form a $\Fp$ pattern (since the right side of $t_1$ needs to be to the left of $p$). 
\end{proof}

\begin{figure}[ht]
\begin{tikzpicture}
\begin{scope}[scale=0.6]
\draw[dotted] (2.2,0.8) -- (2.2,2.5);
\draw (2,0) -- (2,1) -- (3,1);
\draw (2,1.8) -- (2,2.8) -- (3,2.8);
\draw (-0.5,2) -- (1.5,2) -- (1.5,3);
\draw (0,2) -- (0,1.2);
\draw (2.2,0.8) [fill] circle (.1);
\draw (2.2,2.5) [fill] circle (.1);
\draw (0.2,1.8) [fill] circle (.1);
\node at (2.5,0.5) {$c$};
\node at (2.5,2.3) {$p$};
\node at (0.4,1.5) {$q$};
\node at (0,2.5) {$t_1$};
\node at (3.3,0.2) {$t_2$};
\end{scope}
\end{tikzpicture}
 \caption{Proof of Proposition~\ref{prop:PsiHasCodomainT}: if $T$ contains a forbidden pattern then $F$ also does.} \label{fig:t1t2forbiddenpattern}
 \end{figure}
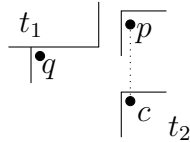

We can now turn to the proof (with the following two lemmas) that point-order and tile-order coincide. 

\begin{lemma}\label{lemma:order}
Let $T\in\T_n$ and $F=\Phi_{\F}(T)$.
When identifying the pointed cells of $T$ and the corners of the tiles of $F$,
the point-order of $T$ coincides with the tile-order of $F$.
\end{lemma}

\begin{proof}
We shall prove the following statement by induction on $k$ from $n$ to $1$: 
\begin{quote}
In the computation of $F=\Phi_{\F}(T)$, after dealing with the points of $T$ having point-labels $n, \dots k+1$, 
the tile associated with the point of $T$ with point-label $k$ is the rightmost among the unlabeled tiles whose bottom border does not touch any unlabeled tile. 
\end{quote}
Indeed, this ensures that at each step $k$, the point of $T$ with point-label $k$ corresponds to the tile of $F$ with tile-label $k$, and therefore proves the lemma. 

\medskip

We first prove our claim in the case $k=n$. 
The tile associated with the point of $T$ having point-label $n$ being added maximal, it reaches the bottom-right corner of the bounding rectangle of $F$, 
which obviously implies our claim. 

\medskip

We next examine the case $k=n-1$ (as preparation for the generic case). 
First observe that the point of $T$ of point-label $n-1$ cannot lie above and to the left of the point of label $n$ (see hatched area in Figure~\ref{fig:order1}, left), 
or the corresponding tile would not be of rectangular shape, contradicting Proposition~\ref{prop:phiF_well_defined}. 
We distinguish two subcases depending on whether this point is above and to the right of the point of label $n$ (case $\textcolor{blue}{(a)}$), or to the left and below it (case $\textcolor{blue}{(b)}$, 
both displayed in Figure~\ref{fig:order1}, left). 

\begin{figure}[ht]
\begin{tikzpicture}
\begin{scope}[scale=0.6]
\draw (0,0) rectangle (6,4);
\draw[fill=lightgray] (3,0) rectangle (6,2); 
\draw (3.2,1.8) [fill] circle (.1);
\node at (3.6,1.7) {\footnotesize $n$};
\draw[pattern=north west lines, pattern color=blue] (0,4) rectangle (3,2);
\node at (5.5,3.5) {$\textcolor{blue} {(a)}$};
\node at (2.5,0.5) {$\textcolor{blue} {(b)}$};
\end{scope}
\end{tikzpicture} 
\qquad 
\begin{tikzpicture}
\begin{scope}[scale=0.6]
\draw (0,0) rectangle (6,4);
\draw (1,0) rectangle (3,1.2); 
\draw[fill=lightgray] (3,0) rectangle (6,2); 
\draw[fill=lightgray] (4.5,2) rectangle (6,3.2); 
\draw (3.2,1.8) [fill] circle (.1);
\node at (3.6,1.7) {\footnotesize $n$};
\draw (1.2,1) [fill] circle (.1);
\node at (1.8,0.7) {\footnotesize $n-1$};
\draw (4.7,3) [fill] circle (.1);
\node at (5.1,2.9) {\footnotesize $x$};
\end{scope}
\end{tikzpicture}
\qquad 
\begin{tikzpicture}
\begin{scope}[scale=0.6]
\draw[draw=none,fill=lightgray] (0.5,0.4) rectangle (7,0);
\draw[draw=none,fill=lightgray] (1.2,1) rectangle (7,0);
\draw[draw=none,fill=lightgray] (3,2) rectangle (7,0);
\draw[draw=none,fill=lightgray] (4,2.6) rectangle (7,0);
\draw[draw=none,fill=lightgray] (4.8,3.3) rectangle (7,0);
\draw[draw=none,fill=lightgray] (5.7,3.7) rectangle (7,0);
\draw[draw=none,pattern=north west lines, pattern color=blue] (0,0.4) rectangle (0.5,4);
\draw[draw=none,pattern=north west lines, pattern color=blue] (0.5,1) rectangle (1.2,4);
\draw[draw=none,pattern=north west lines, pattern color=blue] (1.2,2) rectangle (3,4);
\draw[draw=none,pattern=north west lines, pattern color=blue] (3,2.6) rectangle (4,4);
\draw[draw=none,pattern=north west lines, pattern color=blue] (4,3.3) rectangle (4.8,4);
\draw[draw=none,pattern=north west lines, pattern color=blue] (4.8,3.7) rectangle (5.7,4);
\draw (0,0) rectangle (7,4);
\draw (0.5,0) -- (0.5,0.4) -- (1.2,0.4) -- (1.2,1) -- (3,1) -- (3,2) -- (4,2) -- (4,2.6) -- (4.8,2.6) -- (4.8,3.3) -- (5.7,3.3) -- (5.7,3.7) -- (7,3.7);
\draw[blue] (0,0) -- (0,0.4) -- (0.5,0.4) -- (0.5,1) -- (1.2,1) -- (1.2,2) -- (3,2) -- (3,2.6) -- (4,2.6) -- (4,3.3) -- (4.8,3.3) -- (4.8,3.7) -- (5.7,3.7) -- (5.7,4);
\draw (0.7,0.2) [fill] circle (.1);
\draw (1.4,0.8) [fill] circle (.1);
\draw (3.2,1.8) [fill] circle (.1);
\draw (4.2,2.4) [fill] circle (.1);
\draw (5,3.1) [fill] circle (.1);
\draw (5.9,3.5) [fill] circle (.1);
\draw (1.8,1) -- (1.8,1.7) -- (3,1.7);
\draw (1.95,1.55) [fill] circle (.1);
\node at (2.2,1.4) {\footnotesize $k$};
\end{scope}
\end{tikzpicture}
 \caption{Proof of Lemma~\ref{lemma:order}.} \label{fig:order1}
\end{figure}
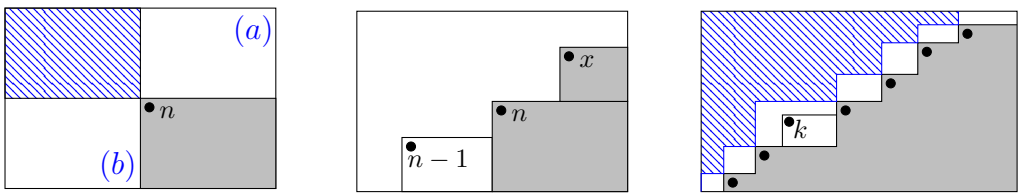

In case $\textcolor{blue}{(a)}$, the tile associated with the point of point-label $n-1$ being maximal, it reaches the topright corner of the tile labeled $n$, 
and therefore is the rightmost among the unlabeled tiles whose bottom border does not touch any unlabeled tile. 

In case $\textcolor{blue}{(b)}$ (and this is illustrated by Figure~\ref{fig:order1}, middle), 
we want to make sure that  there does not exist any point of $T$, located above and to the right of the point of label $n$, 
and whose tile reaches the topright corner of the tile labeled $n$. 
Indeed, only such a point could give rise a different tile of $F$ which would receive the tile-label $n-1$. 
So let us assume that such a point, $x$, of $T$ exists. 

We know that $x$ does not receive the point-label $n-1$. 
Since it is to the right of the point of $T$ receiving this point-label, 
it is necessary that the cell below $x$ is an empty cell of $T$. 
In its row, this empty cell has no point of $T$ to its right (by assumption on $x$), 
and therefore must have a point of $T$ to its left (its row being non-empty). 
This shows that this empty cell is a crossing of $T$. 
By Observation~\ref{obs:ins}, it follows that this empty cell belongs to a ribbon of $T$. 
Our next step is to show that the point of $T$ at the other extremity of this ribbon (denoted $y$) is the point of label $n$. 

Assume first that $y$ is located in a column to the left of the column of the point labeled $n$. 
It means that a cell of the ribbon we are considering belongs to the column of the point labeled $n$. 
This cell cannot be below the point labeled $n$ (otherwise it would not have received the point-label $n$). 
And it cannot be above the point labeled $n$, since a crossing above this point implies the existence of a forbidden pattern $\Tv$. 
So in this case we reach a contradiction. 

Assume next that $y$ lies in the column of the point labeled $n$ or further to the right, but is different from the point labeled $n$. 
Necessarily, $y$ lies above the point labeled $n$ (which has no point to its right and below, see e.g. Observation~\ref{obs:point-order-region}). 
Also, by Observation~\ref{obs:ribbon}, $y$ has point-label one larger than the point-label of $x$. 
Therefore, in the construction of $F=\Phi_{\F}(T)$, the tile associated with $y$ should be placed before the tile associated with $x$. 
This contradicts that the tile associated with $x$ reaches the topright corner of the tile labeled $n$ 
(instead, in this situation, the tile associated with $y$ should extend until this corner). 

Therefore, the point $y$ at the left extremity of the ribbon whose right extremity is $x$ is the point of label $n$. 
Observation~\ref{obs:ribbon} then implies that $x$ has point-label $n-1$, which  brings a contradiction. 
This concludes our analysis in case $\textcolor{blue}{(b)}$, proving that there exists no point $x$ as above, 
and therefore that the tile associated with the point labeled by $n-1$ also has tile-label $n-1$. 

\medskip

The case of generic $k \leq n-1$ is similar to the case $k=n-1$. 
Figure~\ref{fig:order1} (right) represents the construction of $F=\Phi_{\F}(T)$ after the placement of the tiles 
associated with points labeled $n$ to $k+1$. The gray area represents the tiles already in place. 
The Northwest boundary of this area is determined by a sequence of points $c_1, \dots, c_h$ ($h=6$ on the picture), 
which are those having point-label at least $k+1$ and no point with such labels above or to the left of them. 

First, we note that the point of $T$ labeled by $k$ cannot lie above and to the left of any $c_i$. 
Indeed, the tile associated with the point labeled $k$ would not be of rectangular shape in this case. 
In Figure~\ref{fig:order1} (right), the hatched area indicates this area where the point labeled by $k$ cannot lie. 

So, the point labeled by $k$ must lie in some ``white cell'' of Figure~\ref{fig:order1} (right). 
To ensure that the tile associated with the point labeled $k$ receives the tile-label $k$, 
similarly to the previous case ($k=n-1$), it is enough to show that 
in each ``white cell'' to the right of that containing the point labeled $k$ 
there is no point of $T$ whose tile reaches the bottom-right corner of this ``white cell''. 

To prove this claim, we proceed as in case $\textcolor{blue}{(b)}$ above. 
More precisely, assuming that a point $x$ as above exists, 
we consider the first point $c_i$ to the left of $x$ 
(this $c_i$ plays the role of the point labeled by $n$ in case $\textcolor{blue}{(b)}$ above), 
and we show that there exists a ribbon from $x$ to $c_i$. 
This part of the proof is identical to what was done in case $\textcolor{blue}{(b)}$. 

The label of $c_i$ being at least $k+1$, Observation~\ref{obs:ribbon} shows that $x$ has point-label at least $k$. 
On the other hand, $x$ has point-label less than $k$ (since it is a point not yet used for the construction of $F$ 
when placing the tile associated with the point labeled by $k$). 
This gives the contradiction showing that no such $x$ exists. 
And therefore, it guarantees that the point labeled by $k$ also has tile-label $k$, which concludes the proof. 
\end{proof}

\begin{lemma}\label{lemma:order2}
Let $F\in\F_n$ and $T=\Psi_{\F}(F)$.
When identifying the pointed cells of $T$ and the corners of the tiles of $F$,
the point-order of $T$ coincides with the tile-order of $F$.
\end{lemma}

\begin{proof}
This is a consequence of the definition of $\Psi_{\F}$ and can be proved by induction. 
Note first that Observation~\ref{obs:last-added-Psi} (applied for $j=n$) ensures that the tile of $F$ with label $n$ and the point of $T$ with label $n$ coincide.

Now, let us suppose that the point-order in $T$ coincides with the tile-order of $F$
for labels $n,\dots,k$. We shall prove that they also coincide for label $k-1$. 

Let us consider the TLT, denoted $\hat{T}$, obtained by removing the cells of labels $n,\dots,k$ from $T$, with the procedure \emph{RemovePoint}.
On the other hand, consider the object $\tilde{F}$ obtained from $F$ by keeping only the tiles of tile-label strictly less than $k$. 
Observe that $\tilde{F}$ is not a PFP, since it is not of rectangular shape. 
We can nevertheless apply to $\tilde{F}$ the same construction as in the definition of $\Psi_{\F}$, 
obtaining a TLT $\tilde{T}$. 
By definition of $\Psi_{\F}$, we have that $\tilde{T} = \hat{T}$.
Consider the tile labeled by $k-1$ in $\tilde{F}$ and the corresponding pointed cell in $\tilde{T}$, denoted $c$. 
(So, $c$ has label $k-1$ for the tile-order.) 
Because of Observation \ref{obs:last-added-Psi}, $c$ is the rightmost of the pointed cell
without any cell below it. 
In other words, it is the special point of $\tilde{T}$. 
Thus it gets the label $k-1$ for the point-order, as required.
\end{proof}

We shall now conclude the proof of Theorem \ref{thm:FT}.

\begin{proof}[Proof of Theorem \ref{thm:FT}]
We can now conclude that the two applications $\Psi_{\F}$ and $\Phi_{\F}$ are inverse.

Indeed, let us consider $T\in\T_n$. We have proved that $\Phi_{\F}(T)$ is in $\F_n$,
thus we may define $T'=\Psi_{\F}(\Phi_{\F}(T))$.
By definition of $\Psi_{\F}$ and $\Phi_{\F}$, the pointed cell are the same in $T$ and $T'$.
Moreover, the point-order of these pointed cells coincide 
(Lemmas \ref{lemma:order} and \ref{lemma:order2}).
Hence the ribbon configuration is the same in $T$ and $T'$, which implies that $T=T'$.
This proves that $\Psi_{\F}\circ\Phi_{\F}=Id_{\T_n}$.

We prove in the same way that $\Phi_{\F}\circ\Psi_{\F}=Id_{\F_n}$.
\end{proof}

%%%%%%%%%%%%%%%%%%%%
\section{Bijection with twisted Baxter permutations}
\label{sec:bax}

\subsection{A bijection between TLTs and permutations}
\label{subsec:bijectionTLTsPermutations}

Recall that TLTs are in size-preserving bijection with permutations. 
Indeed,~\cite{TLT} provides several bijections between them. 
Here, we define yet another bijection between TLTs and permutations, 
which is however related to the so-called \emph{code bijection} of~\cite{TLT} -- see Proposition~\ref{prop:code_is_isolabel}. 

Consider a TLT $T$ of size $n$. As described in Section~\ref{sec:tlt}, 
its pointed cells may be labeled by the integers $1, \ldots, n$, by the insertion procedure. 
We now describe a way to extend this labeling to the empty cells of $T$. 
First, every empty cell $c$ of the first column (resp. row) of $T$ 
takes the label of the closest pointed cell above (resp. to the left of) $c$ in the same column (resp. row). 
Notice that such a pointed cell always exists, because of the root of $T$. 
Second, we propagate this labeling to all empty cells of $T$, going from Northwest to Southeast, as follows. 
Consider an empty cell $c$ that has not yet been labeled. 
Proceeding iteratively, we can assume that $c$ has North, West, and Northwest neighboring cells in $T$, and that these have already been labeled. 
Denote by $y$, $z$ and $x$ their respective labels (see Figure~\ref{fig:code_bijection}, left). 
We then distinguish four cases to determine the label of $c$:
\begin{itemize}
 \item if there is a point above $c$ in the same column, and a point to the left of $c$ in the same row (recall that such cells are called \emph{crossings}, 
 a terminology that we will use again in Lemma~\ref{lem:crossings_and_2+12}),
 then $c$ receives the label $x$; 
 \item if there is a point above $c$ in the same column, but no point to the left of $c$ in the same row,
 then $c$ receives the label $y$; 
 \item if there is a point to the left of $c$ in the same row, but no point above $c$ in the same column, 
 then $c$ receives the label $z$; 
 \item if there is neither a point to the left of $c$ in the same row, nor a point above $c$ in the same column, 
 then an easy induction ensures that $x=y=z$, and $c$ receives this label.
\end{itemize}
Figure~\ref{fig:code_bijection} (right) shows an example. We shall denote by $\iso(c)$ 
the label associated with the cell $c$. 

Recall that we have defined another (partial) labeling, denoted $\ins$, in Definition~\ref{dfn:ins} (p.~\pageref{dfn:ins}); 
note that it is defined for crossings $c$ only, and that in general we have $\iso(c) \neq \ins(c)$.

\begin{figure}[ht]
$$
\begin{array}{lcr}
\raisebox{-1.5cm}{\includegraphics[scale=1]{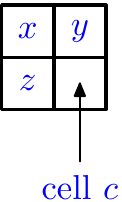}} & \hspace*{0.5cm}
\mbox{ For }T = \raisebox{-1.5cm}{\includegraphics[scale=0.8]{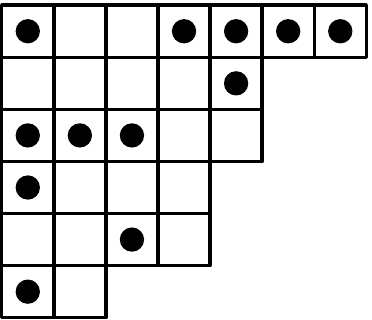}}
\mbox{ the labeling of $T$ is } \hspace*{0.1cm}
\raisebox{-1.5cm}{\includegraphics[scale=0.8]{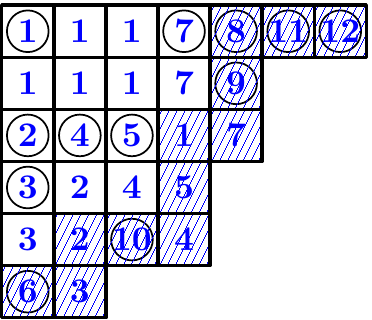}}
\end{array}
$$
\caption{The $\iso$-labeling of a TLT, which defines the bijection $\code$. 
Our notational convention in the rightmost part of this figure is that pointed cells of the TLT are indicated by circled entries. }
\label{fig:code_bijection}
\end{figure}

When all cells of $T$ are labeled, we define $\code(T)$ as follows. 
Starting from the bottommost cell of the first column of $T$, 
we go along its Southeast border until the rightmost cell of the first row of $T$ is reached 
(that is to say, at every step, we go one cell to the right if this is a cell of the TLT, one cell up otherwise). 
Then $\code(T)$ is just the sequence of labels that are met along this Southeast border. 
For example, for the TLT $T$ of Figure~\ref{fig:code_bijection}, we have $\code(T) = 6\, 3\, 2\, 10\, 4\, 5\, 1\, 7\, 9\, 8\, 11\, 12$. 

\begin{proposition}
$\code$ is a size-preserving bijection between TLTs and permutations. 
\label{prop:code_bijection}
\end{proposition}

Proposition~\ref{prop:code_bijection} is an immediate corollary of Proposition~\ref{prop:code_is_isolabel} below. 
Its proof uses the following lemma: 

\begin{lemma}
Let $T$ be a TLT of size $n+1$ and $T'$ be the TLT of size $n$ (defined in Section~\ref{sec:tlt}) 
obtained from $T$ by deletion of the special point $s$ of $T$ with its row (resp. column) and ribbon (if it exists). 
Let $\sigma = \code(T)$ and $\sigma'= \code(T')$. 
Then the deletion of $n+1$ in $\sigma$ gives $\sigma'$. 
Moreover, $n+1$ is located at the $j$-th position in $\sigma$ 
where $j$ is the number of cells (including $s$ itself) along the border of $T$ that are located at the Southwest of $s$. 
\label{lem:isolabel_and_insertion}
\end{lemma}

\begin{proof}
That $n+1$ is located at the $j$-th position in $\sigma$ is clear by definition of $\code$. 
To prove that $\sigma$ and $\sigma'$ coincide up to the deletion of $n+1$ in $\sigma$, 
let us examine how the labeling of all the cells of $T$ (except $s$) is related to the labeling of all the cells of $T'$. 

Consider first the row (resp. column) of empty cells to the left of $s$ (resp. above $s$). 
The rules defining the labeling of $T$ ensure that 
each such cell has the same label as the one immediately above it (resp. immediately to its left). 
As a consequence, all cells of $T$ (except $s$) that are neither in the row (resp. column) 
nor in the ribbon of $s$ have the same label in $T$ as in $T'$. 
The only labels yet to determine are those of the cells of the ribbon of $s$, when it exists. 
Because there are no empty rows nor columns, such a cell is always a crossing, 
hence it has the same label as its Northwest neighboring cell. 

Whether or not the ribbon of $s$ exists, we can now compare the labelings of $T$ and $T'$. 
And it follows immediately that, up to $n+1$ which corresponds to $s$, 
the sequence of integers defining $\sigma$ that is read along the Southeast border of $T$ 
is the same as the sequence read along Southeast border of $T'$, \ie is $\sigma'$. 
\end{proof}

\begin{proposition}
Denote by $\Phi_1$ the \emph{code bijection} of~\cite[Theorem 3.4]{TLT} between TLTs and permutations. 
For any TLT $T$, denoting by $\tau$ the permutation such that $\Phi_1(T) = \tau^{-1}$, 
we have $\tau = \code(T)$. 
\label{prop:code_is_isolabel}
\end{proposition}

\begin{proof}
The proof is by induction on the size of $T$. 
The base case of the induction is clear. 
Let $T$ be a TLT of size $n+1$ and $T'$ be the TLT of size $n$ 
obtained from $T$ by deletion of the special point $s$ of $T$ with its row (resp. column) and ribbon (if it exists). 
Consider the permutations $\sigma = \code(T)$, $\sigma' = \code(T')$, $\tau = \Phi_1(T)^{-1}$ and $\tau' = \Phi_1(T')^{-1}$. 
By induction hypothesis, $\sigma'=\tau'$. 
Moreover, by Lemma~\ref{lem:isolabel_and_insertion}, 
$\sigma$ is obtained from $\sigma'$ by insertion of $n+1$ at position $j$, 
where $j$ is the number of cells along the border of $T$ that are located at the Southwest of $s$. 
From the definition of $\Phi_1$ in~\cite{TLT}, we have similarly that 
$\tau$ is obtained from $\tau'$ by insertion of $n+1$ at the same position $j$. 
We deduce that  $\sigma=\tau$, concluding the proof.
\end{proof}

\subsection{Labels of crossings in the bijection~$\code$}

In the labeling of the cells of a TLT $T$ that defines the bijection~$\code$, 
the integers labeling the crossings of $T$ have a nice interpretation in the permutation $\code(T)$ -- see Lemma~\ref{lem:crossings_and_2+12}. 
This interpretation is essential for the study of the specialization of $\code$ to Baxter objects, 
in Subsection~\ref{subsec:SpecializationBijectionBaxterTLTs}. 
To explain it, we need to review the classical notions of vincular and bivincular patterns in permutations \cite{BCDK}.

\begin{itemize}
 \item A \emph{(classical) pattern} is simply a permutation; but for notational convenience, we insert a dash between any two adjacent entries. 
An \emph{occurrence} of a classical pattern $\tau$ in a permutation $\sigma$ is a subsequence of $\sigma$ which is order-isomorphic to $\tau$. 
 \item A \emph{vincular pattern} (or \emph{dashed pattern}) is a permutation in which every pair of adjacent entries may be linked by a dash. 
Occurrences of a vincular pattern $\tau$ in a permutation $\sigma$ are defined like in the case of classical patterns, 
with the additional restriction that two adjacent entries of $\tau$ that are not separated by a dash must correspond to adjacent entries in $\sigma$.
 \item \emph{Bivincular} patterns are a generalization of vincular patterns, where adjacency constraints are allowed not only on positions but also on values. 
\end{itemize}

Here, we will be interested in very simple bivincular patterns, with only one constraint on values. 
Such patterns of size $n$ can be represented as a vincular pattern whose entries are $\{1,2, \ldots, n-1\} \cup \{i^+\}$, 
for some $i \in \{1,2, \ldots, n-1\}$. 
In an occurrence of such a pattern in a permutation $\sigma$, 
we require that the entries of $\sigma$ corresponding to $i$ and $i^+$ have consecutive values 
(namely, that $i^+$ corresponds to $k+1$ when $i$ corresponds to $k$). 
The (bivincular) pattern $2^+-1-2$ will be of particular interest to us, so let us rephrase: 
An occurrence of a $2^+-1-2$ pattern in a permutation $\sigma$ is a subsequence 
$\sigma(i) \sigma(j) \sigma(k)$ of $\sigma$, with $i<j<k$ such that 
$\sigma(j) < \sigma(k)$ and $\sigma(i) = \sigma(k)+1$.

For example, consider the classical pattern $3-1-2$, the vincular pattern $3-12$, 
and the bivincular patterns $2^+-1-2$ and $2^+-12$. 
Their occurrences in $\sigma = 6\ 3\ 2\ 10\ 4\ 5\ 1\ 7\ 9\ 8\ 11\ 12$ 
are summarized in Figure~\ref{fig:occ}. 
Notice that this permutation $\sigma$ satisfies $\sigma = \code(T),$ where $T$ is the TLT of Figure~\ref{fig:code_bijection}.

\begin{figure}[ht]
\begin{center}
\begin{tabular}{|l|l|}
\hline
$\tau$ & Occurrences of $\tau$ in $\sigma$ \\
\hline
$3-1-2$ & $6\ 3\ 4$, $6\ 3\ 5$, $6\ 2\ 4$, $6\ 2\ 5$, $6\ 4\ 5$, $10\ 4\ 5$, $10\ 4\ 7$, $10\ 4\ 9$, $10\ 4\ 8$,  \\
 & $10\ 5\ 7$, $10\ 5\ 9$, $10\ 5\ 8$, $10\ 1\ 7$, $10\ 1\ 9$, $10\ 1\ 8$, $10\ 7\ 9$, $10\ 7\ 8$ \\
\hline
$3-12$ & $6\ 4\ 5$, $10\ 4\ 5$, $10\ 1\ 7$, $10\ 7\ 9$ \\
\hline
$2^+-1-2$ & $6\ 3\ 5$, $6\ 2\ 5$, $6\ 4\ 5$, $10\ 4\ 9$, $10\ 5\ 9$, $10\ 1\ 9$, $10\ 7\ 9$ \\
\hline
$2^+-12$ & $6\ 4\ 5$, $10\ 7\ 9$\\
\hline
\end{tabular}
\caption{Occurrences of several patterns in $\sigma = 6\ 3\ 2\ 10\ 4\ 5\ 1\ 7\ 9\ 8\ 11\ 12$. \label{fig:occ}}
\end{center}
\end{figure}

\begin{lemma}
The crossings of any TLT $T$ are in one-to-one correspondence with the occurrences of the pattern 
$2^+-1-2$ in $\sigma = \code(T)$. 
\\ \noindent
Under this correspondence, for any crossing $c$,
the value of $\sigma$ to which $1$ is mapped is $\iso(c)$ defined in Subsection~\ref{subsec:bijectionTLTsPermutations} 
and the value of  $\sigma$ to which $2^+$ is mapped is $\ins(c)$ defined in Definition~\ref{dfn:ins} (p.~\pageref{dfn:ins}). 
\label{lem:crossings_and_2+12}
\end{lemma}

\begin{proof} 
The proof is by induction on the size of $T$, the base case of the induction being clear. 
Let $T$ be a TLT of size $n+1$ and $T'$ be the TLT of size $n$ 
obtained from $T$ by deletion of the special point $s$ of $T$ with its row (resp. column) and ribbon (if it exists). 

The crossings of $T$ are partitioned in two categories: the crossings of $T'$ and the cells of the ribbon of $s$. 
As explained in Observation~\ref{obs:ins}, note that all the cells of the ribbon of $s$ are crossings. 

Consider $\sigma = \code(T)$ and  $\sigma' = \code(T')$. 
From Lemma~\ref{lem:isolabel_and_insertion}, $\sigma'$ may be described as $\sigma$ from which $n+1$ has been deleted. 
Hence, the occurrences of $2^+-1-2$ in $\sigma$ are also partitioned in two categories: 
the occurrences of $2^+-1-2$ in $\sigma'$ and those where $2^+$ is mapped to $n+1$. 

By induction, it follows that the crossings of $T'$ are mapped to the occurrences of $2^+-1-2$ in $\sigma'$. 
The assertion about the values is readily checked, 
since for any such crossing $c$ of $T'$, $\iso(c)$ (resp. $\ins(c)$) 
is the same in $T$ and in $T'$.

Consider now a crossing of $T$ which is a cell of the ribbon of $s$. 
Recall that the label of $s$ is $n+1$. 
Observe now that the pointed cell located at the right extremity of the ribbon of $s$ is the special point $s'$ of $T'$, whose label is therefore $n$. 
From these observations, it is now clear that the occurrences of $2^+-1-2$ where $2^+$ is mapped to $n+1$ 
are in one-to-one correspondence with the cells of the ribbon of $s$. 
In this case, the assertion about the values follows immediately by definition of $\iso(\cdot)$ and $\ins(\cdot)$.  
\end{proof}

\subsection{Specialization of this bijection on Baxter TLTs}
\label{subsec:SpecializationBijectionBaxterTLTs}

\begin{definition}[Twisted Baxter permutations]
\label{defi:Tbax}
A {\em twisted Baxter permutation} is a permutation $\sigma$ which avoids the two vincular patterns $2-41-3$ and $3-41-2$ 
(\ie such that none of these patterns has any occurrence in $\sigma$).
We denote by $\B_n$ the set of inverses of twisted Baxter permutations of size $n$. 
\end{definition}

\begin{observation}
The permutations of $\B_n$ may alternatively be characterized as 
the permutations of size $n$ avoiding the patterns $2^+-1-3-2$ and $2^+-3-1-2$, 
or equivalently the patterns $3-14-2$ and $3-41-2$, 
\ie 
\[\B_n = Av_n(2^+-1-3-2, 2^+-3-1-2) = Av_n(3-14-2, 3-41-2).\]
\label{obs:twistedBaxter_patterns}
\end{observation}
\begin{proof}
When taking the inverse, the adjacency constraints in a vincular pattern are turned into constraints that two elements should have \emph{consecutive values}, 
(which can be represented by a bivincular pattern). 
This proves the first statement of Observation~\ref{obs:twistedBaxter_patterns}. 

The second statement of Observation~\ref{obs:twistedBaxter_patterns} is proved using classical arguments of permutation patterns analysis. 
We prove that a permutation $\sigma$ contains a pattern $3-14-2$ if and only if it contains a pattern $2^+-1-3-2$, 
the case of $3-41-2$ and $2^+-3-1-2$ being similar.

Suppose that  $\sigma$ contains a pattern $3-14-2$ 
where the $2$ is mapped to the entry $i$ and the $3$ is mapped to the entry $j$ of $\sigma$.
We consider the integers in the interval $\{i,\dots,j\}$.
They stand in $\sigma$ either to the left or to the right of the subpattern $14$, with $j$ to the left and $i$ to the right.
Thus, considering these integers in decreasing order, there exist two consecutive of them, $k$ and $(k+1)$ with $i\leq{k}<j$, which stand as: $(k+1)\dots14\dots k$.
This gives a pattern $2^+-1-3-2$.

Conversely, if $\sigma$ contains a pattern $2^+-1-3-2$, we consider the entries in the subword $1-3$. 
They are either strictly greater than $2^+$ or smaller than $2$, 
thus there are two of them which are at adjacent positions and form a pattern $2^+-13-2$ and hence $3-14-2$.
\end{proof}

Figure~\ref{fig:twisted_size_4} lists all permutations of $\B_4$.

\begin{figure}[ht]
\begin{center}
$\begin{array}{l@{\hspace{1pt}}c@{\hspace{1pt}}rl@{\hspace{1pt}}c@{\hspace{1pt}}rl@{\hspace{1pt}}c@{\hspace{1pt}}rl@{\hspace{1pt}}c@{\hspace{1pt}}rl@{\hspace{1pt}}c@{\hspace{1pt}}rl@{\hspace{1pt}}c@{\hspace{1pt}}r}
\sigma^{1} &=& 1234, & 
\sigma^{2} &=& 1243, & 
\sigma^{3} &=& 1324, & 
\sigma^{4} &=& 1342, & 
\sigma^{5} &=& 1423, & 
\sigma^{6} &=& 1432, \\ 
\sigma^{7} &=& 2134, & 
\sigma^{8} &=& 2143, &
\sigma^{9} &=& 2314, & 
\sigma^{10} &=& 2341, & 
\sigma^{11} &=& 2413, & 
\sigma^{12} &=& 2431, \\ 
\sigma^{13} &=& 3124, & 
\sigma^{14} &=& 3214, & 
\sigma^{15} &=& 3241, & 
\sigma^{16} &=& 3421, & 
\sigma^{17} &=& 4123, & 
\sigma^{18} &=& 4132, \\ 
\sigma^{19} &=& 4213, & 
\sigma^{20} &=& 4231, & 
\sigma^{21} &=& 4312, &  
\sigma^{22} &=& 4321. & &
\end{array}$
\end{center}
\caption{The 22 permutations of $\B_4$. }
\label{fig:twisted_size_4}
\end{figure}

There are several bijective proofs in the literature that $|\B_n| = Bax_n$ for all $n$, or more precisely 
that (some symmetry of) twisted Baxter permutations are enumerated by Baxter numbers. 
See~\cite{GuibertThese} for a recursive bijection between Baxter permutations 
and permutations avoiding $2-14-3$ and $3-14-2$ (whose reverses are twisted Baxter permutations), or 
\cite{West,Gir} for more recent bijections between Baxter permutations and twisted Baxter permutations. 
In these articles, 
the pattern-avoiding families of permutations are defined by the avoidance \emph{barred} patterns 
rather than by excluded \emph{dashed} patterns; 
but in each case the equivalence between both descriptions is easily proved, 
with simple arguments similar to those in the proof of Observation~\ref{obs:twistedBaxter_patterns}. 

\medskip

Let us denote by $\Phi_\B$ the restriction of the bijection 
$\code$ of Subsection~\ref{subsec:bijectionTLTsPermutations} 
to the set of Baxter TLTs. 

Our goal is to prove the following: 

\begin{theorem}\label{thm:TTW}
For any $n$, $\Phi_\B$ is a bijection between $\T_n$ and $\B_n$.
\end{theorem}

This is an immediate consequence of the following proposition. 

\begin{proposition}
Let $\sigma$ be a permutation, of size $n$, and $T$ be the TLT defined by $T=\code^{-1}(\sigma)$. 
Then $\sigma \in \B_n$ if and only if $T \in \T_n$. 
\end{proposition}

\begin{proof}
The key point is Lemma \ref{lem:crossings_and_2+12}.
Let $\sigma$ and $T$ be as in the statement of the proposition, and 
consider the sequence of TLTs $(T_n=T, T_{n-1}=T', T_{n-2}, \ldots, $ $ T_1 = 
\begin{tikzpicture}
\begin{scope}[scale=0.25]
\draw (0,0) -- (0,1) -- (1,1) -- (1,0) -- (0,0);
\draw (0.5,0.5) [fill] circle (.25);
\end{scope}
\end{tikzpicture}
)$ defined in Section~\ref{sec:tlt}.

We first prove that if $\sigma$ contains a pattern $2^+-3-1-2$ or $2^+-1-3-2$, 
then $T$ contains a pattern $\Tv$ or $\Th$. 
For this purpose, we consider the occurrence of one of the patterns $2^+-3-1-2$ and $2^+-1-3-2$ in $\sigma$ 
such that the value $j$ of the ``$2^+$'' is maximal among all possibilities, 
and such that the value $k$ of the ``$3$'' is minimal among these occurrences. 
From~\cite{TLT}, this guarantees that the ribbon $R$ from $j$ (the ``$2^+$'') to $j-1$ (the ``$2$'') is exactly the same in $T_j$ and $T_{k-1}$. 
In $T_k$, the pointed cell labeled by $k$ is either to the South or to the East of a crossing which belongs to $R$,
thus giving a pattern $\Tv$ or $\Th$ pattern in $T$.

Conversely, we prove that if $T$ contains a pattern $\Tv$ or $\Th$, then 
$\sigma$ contains a pattern $2^+-3-1-2$ or $2^+-1-3-2$. 
We consider the smallest $i$ such that $T_i$ contains a pattern $\Tv$ or $\Th$. 
The South-easternmost point of this pattern in $T_i$ is the special point $s_i$ of $T_i$ 
(because $T_{i-1}$ avoids the two patterns). 
Moreover, it is to the South or to the East of a crossing of $T_i$, which we denote $c$. 
Let us denote $\sigma_i = \code(T_i)$. We claim that $\sigma_i$ contains a $2^+-3-1-2$ or $2^+-1-3-2$ pattern. 
Indeed, such a pattern can be identified using Lemma~\ref{lem:crossings_and_2+12}. 
More precisely, it is enough consider the entries of $\sigma_i$ corresponding to the following cells of $T_i$:
the pointed cell labeled $iso(c)$ (which corresponds to the ``$1$'' in the pattern), 
$s_i$  (which is the``$3$'') 
and the two extremities of the ribbon of $c$ (which are the ``$2^+$'' and ``$2$''). 
To conclude, Lemma~\ref{lem:isolabel_and_insertion} implies that $\sigma_i$ is a pattern of $\sigma$, 
so that $\sigma$ also contains a $2^+-3-1-2$ or $2^+-1-3-2$ pattern. 
\end{proof}

Figure~\ref{fig:code_bijection-Baxter} shows an example of a Baxter TLT $T$ with the corresponding permutation $\Phi_\B(T)$.

\begin{figure}[ht]
\begin{center}
\includegraphics[scale=1]{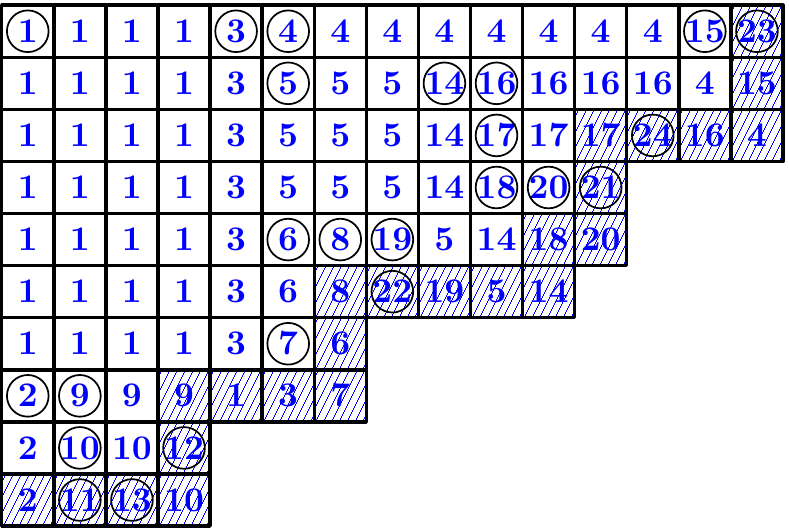}

\includegraphics[scale=0.9]{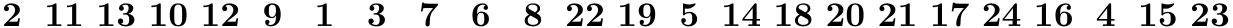}
\end{center}
\caption{A Baxter TLT $T$ (circled entries represent pointed cells), its labeling by $\iso(\cdot)$, 
and the permutation $\Phi_\B(T)$.}
\label{fig:code_bijection-Baxter}
\end{figure}

\subsection{$\Phi_\B$ and classical permutation statistics}

It it interesting to note that the $\iso$-labeling of a Baxter TLT 
allows to interpret some classical statistics on the permutation $\Phi_\B(T)$ directly on the Baxter TLT $T$. 
We describe here these interpretations for the descents and the left-to-right minima. 

\subsubsection{Descents}

Our notational convention in Figures~\ref{fig:propagation} to~\ref{fig:proof_row-ancestor} 
is that 
\includegraphics[scale=\scaleBaxterTLTAscentDescent]{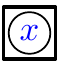} or \includegraphics[scale=\scaleBaxterTLTAscentDescent]{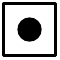} represents a pointed cell, 
\includegraphics[scale=\scaleBaxterTLTAscentDescent]{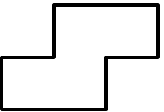} is a region of empty cells, 
\includegraphics[scale=\scaleBaxterTLTAscentDescent]{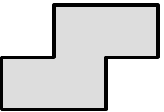} is a region of pointed and/or empty cells, and 
\includegraphics[scale=\scaleBaxterTLTAscentDescent]{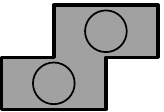} is a region of cells containing at least a point. 
Note that some regions contain only empty cells because of the excluded patterns that define Baxter TLTs.

\begin{lemma}
\label{lem:TTW-ascent+descent}
Let $T$ be a Baxter TLT. 
Let $c$ be a (pointed or empty) cell of $T$ that does not belong to the first column (resp. row) 
and without any pointed cell below it (resp. to its right). 
Denote by $a$ the $\iso$-label of $c$ and by $b$ the $\iso$-label of the cell\footnote{Note that this cell exists, 
because of our assumption that $c$ does not belong to the first column (resp. row).}
immediately to the left of (resp. above) $c$. 
Then $ba$ (resp. $ab$) is a factor\footnote{As in words, a \emph{factor} in a permutation is a sequence of symbols which appear consecutively.} of $\sigma = \Phi_\B(T)$. 
\end{lemma}

\begin{proof}
If $c$ is the bottom-most (resp. rightmost) cell in its column (resp. row), then the claim immediately holds by definition of $\Phi_\B$. 
Otherwise, we prove that the cell immediately below (resp. to the right) of $c$ 
satisfies the same conditions as $c$, induction giving then the conclusion. 

Suppose that $c$ does not belong to the first column and has no pointed cell below it.
Consider the cell immediately below $c$, denoted $c'$. Note that our assumptions ensure that $c'$ is empty. 
Two cases may occur, as shown on Figure~\ref{fig:propagation} (left): 
either there is no pointed cell to the left of $c'$, or there is at least one. 
In the first case, $c'$ has no pointed cell below it, is labeled by $a$ and the cell to its left by $b$, 
so that we can proceed inductively, and obtain that $ba$ is a factor of $\sigma$. 
In the second case, let us first notice that there is at least a point above $c'$ (since the column of $c'$ contains at least a point). 
This implies that there is no point to the right of $c'$ in its row
(or otherwise, we would obtain a pattern $\Th$, with a point above $c'$ 
and two points to the left and right of $c'$). 
This also implies that the label of $c'$ is $b$ 
(by the first case of the rule for the propagation of the $\iso$-label). 
The cell above $c'$ is nothing but our original cell $c$, so it is labeled by $a$. 
We can therefore apply our inductive statement on $c'$, obtaining that $ba$ is a factor of $\sigma$. 

In the case where $c$ does not belong to the first row and has no pointed cell to its right, 
we proceed similarly, distinguishing cases as shown on Figure~\ref{fig:propagation} (right).
\end{proof}

\begin{figure}[ht]
\begin{center}
\begin{tabular}{cccc|cccc}
 \includegraphics[scale=\scaleBaxterTLTAscentDescent]{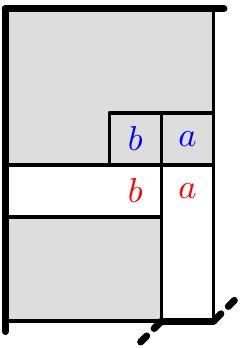} & \hfill & 
 \includegraphics[scale=\scaleBaxterTLTAscentDescent]{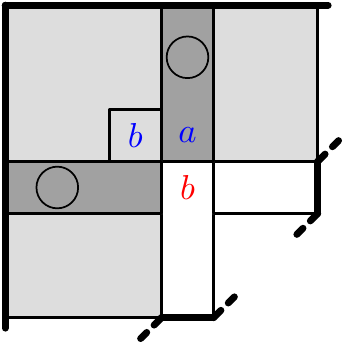} & \hfill & 
\hfill &  \includegraphics[scale=\scaleBaxterTLTAscentDescent]{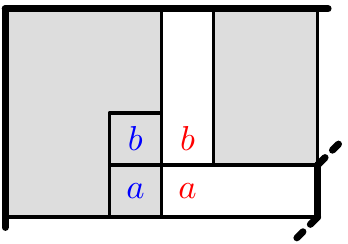} & \hfill & 
 \includegraphics[scale=\scaleBaxterTLTAscentDescent]{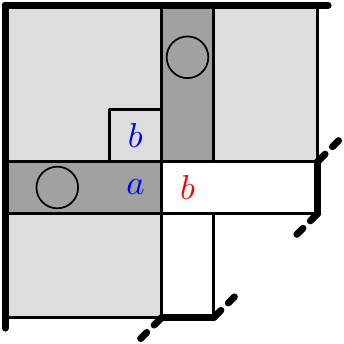} 
\end{tabular}
\end{center}
\caption{Proof of Lemma~\ref{lem:TTW-ascent+descent} when $c$ has no pointed cell below it (left), and when $c$ has no pointed cell to its right (right).}
\label{fig:propagation}
\end{figure}

\begin{definition}
\label{def:ancestors}
A point of a Baxter TLT is \emph{column-extremal} (resp. \emph{row-extremal}) if it does not belong to the first column (resp. row) 
and there is no pointed cell below it (resp. to its right). 

The \emph{column-ancestor} (resp. \emph{row-ancestor}) of a column-extremal (resp. row-extremal) point $x$ 
is the parent of the top-most (resp. left-most) point in the column (resp. row) of $x$. 
\end{definition}

Figure~\ref{fig:column-ancestor} illustrates this definition. 

In Definition~\ref{def:ancestors}, note that 
the top-most (resp. left-most) point in the column (resp. row) of $x$ is not the root of the TLT, 
since $x$ does not belong to the first column (resp. row). This ensures that it has a parent 
so that column- and row-ancestors are well defined. 
Note also that the column-ancestor (resp. row-ancestor) of a column-extremal (resp. row-extremal) point $x$ 
is located in a column to the left of $x$ (resp. in a row above $x$).

\begin{figure}[ht]
\begin{center}
\begin{tabular}{ccccc}
 \includegraphics[scale=\scaleBaxterTLTAscentDescent]{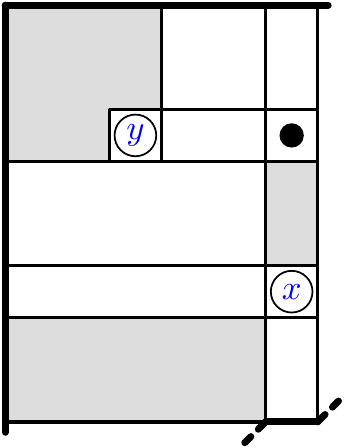} & 
 \includegraphics[scale=\scaleBaxterTLTAscentDescent]{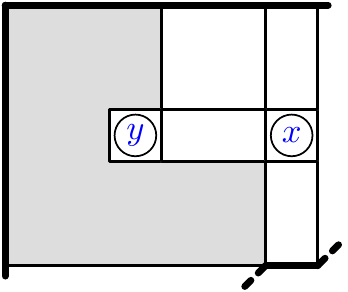} & \hfill & 
 \includegraphics[scale=\scaleBaxterTLTAscentDescent]{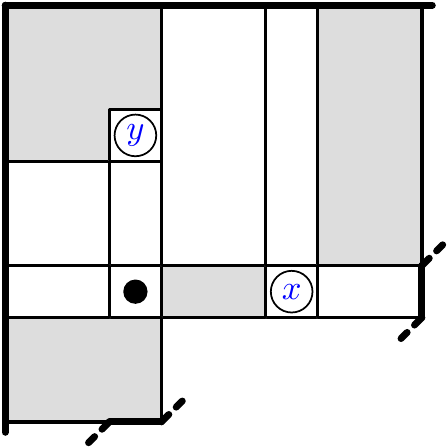} &
 \includegraphics[scale=\scaleBaxterTLTAscentDescent]{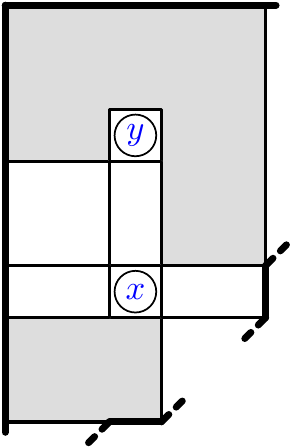} 
\end{tabular}
\end{center}
\caption{The column-ancestor (resp. row-ancestor) $y$ of a column-extremal (resp. row-extremal) point $x$ in a Baxter TLT,  
with the special case where $x$ is the only point in its column (resp. row) shown on the right.}
\label{fig:column-ancestor}
\end{figure}

\begin{proposition}
\label{prop:TTW-ascent+descent}
In a Baxter TLT $T$, consider a column-extremal (resp. row-extremal) point $x$ and its column-ancestor (resp. row-ancestor) $y$. 
Denote by $a$ the label of $x$ and by $b$ the label of $y$. 
Then $ba$ (resp. $ab$) is a factor of $\sigma = \Phi_\B(T)$ and forms an ascent (resp. descent) in $\sigma$. 
\end{proposition}

\begin{proof}
First, we will see that the cell immediately to the left of (resp. above) $x$ has $\iso$-label $b$. 
Lemma~\ref{lem:TTW-ascent+descent} will then ensure that $ba$ (resp. $ab$) is a factor of $\sigma$. 
The claim we actually prove is a bit more general: 
it states that all the cells in the rectangle extending from $y$ to the cell immediately to the left of (resp. above) $x$ have $\iso$-label $b$. 
(See Figures \ref{fig:proof_column-ancestor} and \ref{fig:proof_row-ancestor}.) 

Assume that $x$ is a column-extremal point and that $y$ is its column-ancestor, 
and consider the rectangle extending from $y$ to the cell immediately to the left of $x$. 
This case is illustrated on Figure~\ref{fig:proof_column-ancestor}. 
For the cells of the top row of this rectangle, our claim follows easily from the third case defining the propagation of the $\iso$-labels, 
since such cells have no point of $T$ above them (or a pattern $\Th$ would be created). 
It is just as easy for the cells of the leftmost column of this rectangle, for which the second case of the propagation rule applies, 
since they have no point of $T$ to their left (or a pattern $\Tv$ would be created). 
For the remaining cells of this rectangle, the last case of propagation rule applies, all together proving our claim. 

The case where $x$ is a row-extremal point and $y$ is its row-ancestor is handled similarly, 
as shown by Figure~\ref{fig:proof_row-ancestor}. 

Finally, the parent of any point $x$ in an TLT has a label smaller than the one of $x$, 
and so do all ancestors of $x$. This proves that $ba$ (resp. $ab$) is an ascent (resp. descent) of $\sigma$. 
\end{proof}

\begin{figure}[ht]
\begin{center}
\includegraphics[scale=\scaleBaxterTLTAscentDescentFormsSitua]{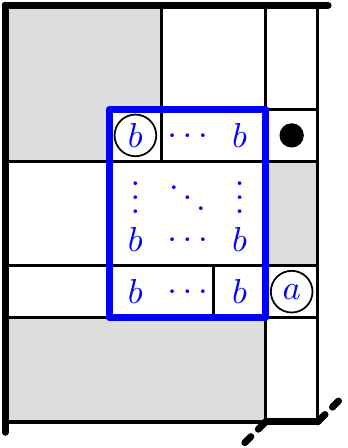} \qquad 
\includegraphics[scale=\scaleBaxterTLTAscentDescentFormsSitua]{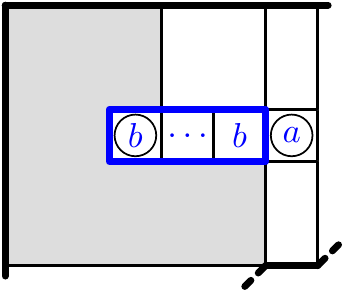} \qquad
\includegraphics[scale=\scaleBaxterTLTAscentDescentFormsSitua]{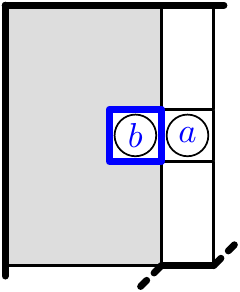}
\end{center}
\caption{Propagation of the $\iso$-label of the column-ancestor of a column-extremal point, in all possible configurations.}
\label{fig:proof_column-ancestor}
\end{figure}

\begin{figure}[ht]
\begin{center}
\includegraphics[scale=\scaleBaxterTLTAscentDescentFormsSitua]{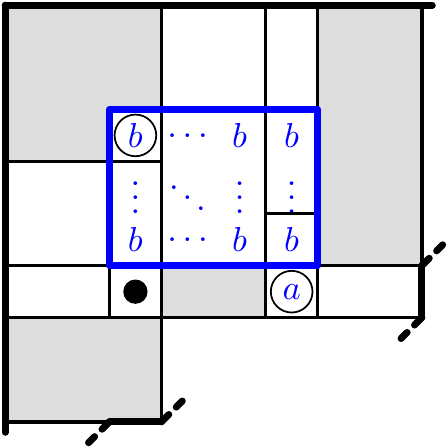} \qquad
\includegraphics[scale=\scaleBaxterTLTAscentDescentFormsSitua]{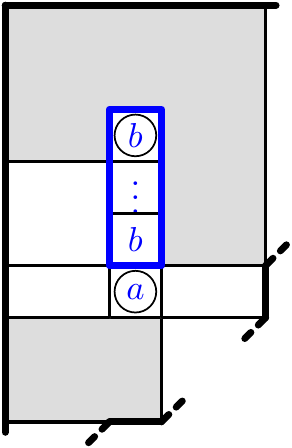} \qquad 
\includegraphics[scale=\scaleBaxterTLTAscentDescentFormsSitua]{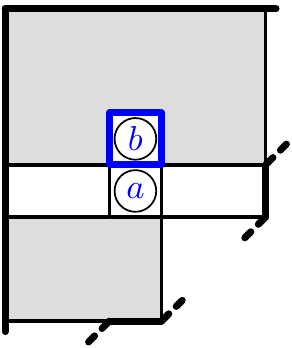}
\end{center}
\caption{Propagation of the $\iso$-label of the row-ancestor of a row-extremal point, in all possible configurations.}
\label{fig:proof_row-ancestor}
\end{figure}

\begin{corollary}
\label{cor:TTW-ascent+descent}
For a Baxter TLT $T$, all the ascents and descents of $\sigma = \Phi_\B(T)$ are those described in Proposition~\ref{prop:TTW-ascent+descent}.
\end{corollary}

\begin{proof}
Denote by $n$ the common size of $T$ and $\sigma$. 
From Definition~\ref{defi:baxTLT}, $T$ has in total $n+1$ rows and columns, so that there are $n-1$ distinct pairs $(x,y)$ where 
$x$ is a column-extremal (resp. row-extremal) point and $y$ is its column-ancestor (resp. row-ancestor). 
Proposition~\ref{prop:TTW-ascent+descent} then gives $n-1$ distinct factors of length $2$ in $\sigma$, 
which are either ascents or descents as described in Proposition~\ref{prop:TTW-ascent+descent}, 
so that the ascents and descents of $\sigma$ are completely described.
\end{proof}

\subsubsection{Left-to-right minima}

\begin{proposition}
Let $T$ be a Baxter TLT, and denote by $T_\ell$ and $T_r$ the Baxter TLTs defined in Corollary~\ref{cor:decomposition_Baxter_TLT} (p.\pageref{cor:decomposition_Baxter_TLT}). 
Let $\sigma = \Phi_\B(T)$, and decompose $\sigma$ around its minimum value as $\sigma = \sigma_\ell 1 \sigma_r$. 
The permutations $\Phi_\B(T_\ell)$ and $\Phi_\B(T_r)$ are order-isomorphic to the sequences $\sigma_\ell$ and $\sigma_r$ respectively. 
\label{prop:placement_of_1}
\end{proposition}

\begin{proof}
The claim obviously holds when $T_\ell$ or $T_r$ is empty, so assume they are not.  

Notice first that the points of $T_\ell$ (resp. $T_r$) may be equipped with two different labelings: 
one labeling is inherited from the labeling of the TLT $T$, 
and one is its own labeling as TLTs. 
These are not identical, but by construction, they are ``order-isomorphic'' 
(that is to say, the comparison between the labels of any two points is the same in both labelings). 

Now, both labelings may be propagated following the rules of propagation of the $\iso$-labeling, 
yielding two order-isomorphic labelings of the cells of $T_\ell$ (resp. $T_r$). 

To conclude, it is enough to prove that $\sigma_\ell$ (resp. $\sigma_r$) 
is the word that is read along the Southeast border of $T_\ell$ (resp. $T_r$), in the $\iso$-labeling propagated from the original labeling of $T$. 
And this claim holds because all cells in the Southeast block identified by Corollary~\ref{cor:decomposition_Baxter_TLT} 
are crossings, and because the $\iso$-labels of crossings are inherited following Northeast-Southwest diagonals. 
\end{proof}

\begin{corollary}
Let $T$ be a Baxter TLT, and let $\sigma = \Phi_\B(T)$. 
The left-to-right minima of $\sigma$ are the labels of the points of the left-most column of $T$. 
\label{cor:LtoR-min}
\end{corollary}

\begin{proof}
It follows easily from Proposition~\ref{prop:placement_of_1} by induction. 

Given $T$, denote by $T_\ell$ and $T_r$ the Baxter TLTs defined in Corollary~\ref{cor:decomposition_Baxter_TLT}. 
Decompose also $\sigma$ as $\sigma_\ell 1 \sigma_r$ like in the statement of Proposition~\ref{prop:placement_of_1}. 
The claim clearly holds when $T_\ell$ is empty: 
indeed, the only point of $T$ in the first column is the root of $T$, labeled by $1$, 
and $\sigma$ starts with $1$. 
If $T_\ell$ is not empty, the left-to-right minimal of $\sigma$ are those of $\sigma_\ell$ and $1$. 
Induction ensures that the left-to-right minima of $\sigma_\ell$ are the labels of the points in its first column. 
And $1$ is the label of the root of $T$, which is the only point in the first column of $T$ that is not in the first column of $T_\ell$. 
\end{proof}

%%%%%%%%%%%%%%%%%%%%
\section{Bijection with non-intersecting lattice paths}
\label{sec:nilp}

\begin{definition}[Triples of non-intersecting lattice paths]
\label{defi:TP}
A \emph{triple of non-intersecting lattice paths of size $n$} is a set of three lattice paths, 
with unitary $N=(0,1)$ and $E=(1,0)$ steps, 
that never meet, 
which respectively start at $(1,0)$, $(0,1)$ and $(-1,2)$ and end at $(n-i,i)$, $(n-i-1,i+1)$ and $(n-i-2,i+2)$ for some $i \in [0..(n-1)]$ 
(thus each of the three paths has $n-1$ steps). 

Let us denote by $\P_n$ the set of triples of non-intersecting lattice paths of size~$n$.
\end{definition}

Figure~\ref{fig:psi} (p.~\pageref{fig:psi}) (right) shows an example of triple of non-intersecting lattice paths of size $18$, 
and Figure~\ref{fig:NILP_size_4} shows all triples of non-intersecting lattice paths of size $4$. 
On these figures, the extremities of the paths are indicated by circles; an additional $E$ (resp. $N$) step 
has been represented at the beginning (resp. end) of the middle and lower paths. 
The reason for this choice will appear clearly later. 

\begin{figure}[h!]
\begin{align*}
 \pi^{1} &=& & \hspace*{-1em} \begin{array}{c} \includegraphics[scale=\scaleNILP]{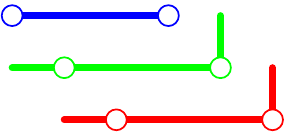} \end{array}, 
& \pi^{2} &=& & \hspace*{-1em} \begin{array}{c} \includegraphics[scale=\scaleNILP]{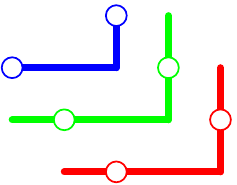} \end{array}, 
& \pi^{3} &=& & \hspace*{-1em} \begin{array}{c} \includegraphics[scale=\scaleNILP]{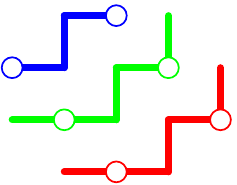} \end{array}, \\
 \pi^{4} &=& & \hspace*{-1em} \begin{array}{c} \includegraphics[scale=\scaleNILP]{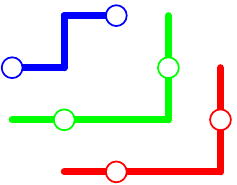} \end{array}, 
& \pi^{5} &=& & \hspace*{-1em} \begin{array}{c} \includegraphics[scale=\scaleNILP]{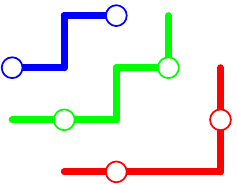} \end{array}, 
& \pi^{6} &=& & \hspace*{-1em} \begin{array}{c} \includegraphics[scale=\scaleNILP]{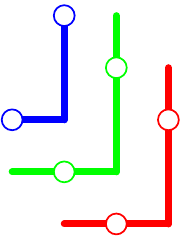} \end{array}, \\
 \pi^{7} &=& & \hspace*{-1em} \begin{array}{c} \includegraphics[scale=\scaleNILP]{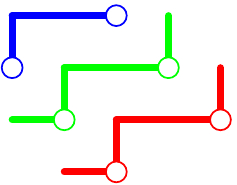} \end{array}, 
& \pi^{8} &=& & \hspace*{-1em} \begin{array}{c} \includegraphics[scale=\scaleNILP]{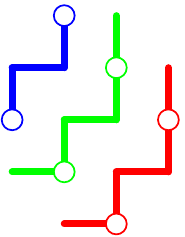} \end{array}, 
& \pi^{9} &=& & \hspace*{-1em} \begin{array}{c} \includegraphics[scale=\scaleNILP]{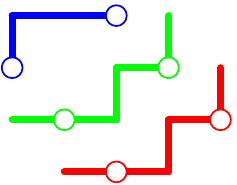} \end{array}, \\
 \pi^{10} &=& & \hspace*{-1em} \begin{array}{c} \includegraphics[scale=\scaleNILP]{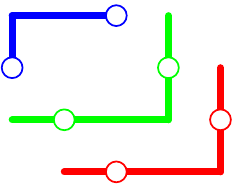} \end{array}, 
& \pi^{11} &=& & \hspace*{-1em} \begin{array}{c} \includegraphics[scale=\scaleNILP]{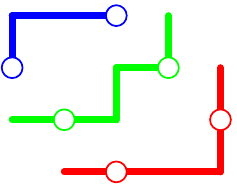} \end{array}, 
& \pi^{12} &=& & \hspace*{-1em} \begin{array}{c} \includegraphics[scale=\scaleNILP]{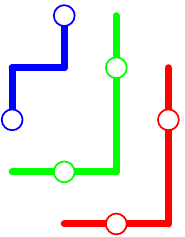} \end{array}, \\
 \pi^{13} &=& & \hspace*{-1em} \begin{array}{c} \includegraphics[scale=\scaleNILP]{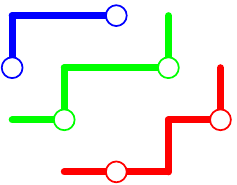} \end{array}, 
& \pi^{14} &=& & \hspace*{-1em} \begin{array}{c} \includegraphics[scale=\scaleNILP]{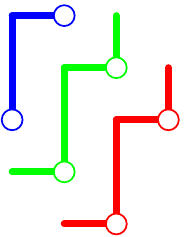} \end{array}, 
& \pi^{15} &=& & \hspace*{-1em} \begin{array}{c} \includegraphics[scale=\scaleNILP]{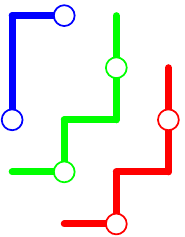} \end{array}, \\
 \pi^{16} &=& & \hspace*{-1em} \begin{array}{c} \includegraphics[scale=\scaleNILP]{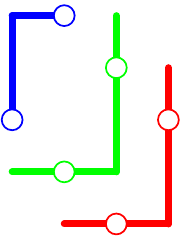} \end{array}, 
& \pi^{17} &=& & \hspace*{-1em} \begin{array}{c} \includegraphics[scale=\scaleNILP]{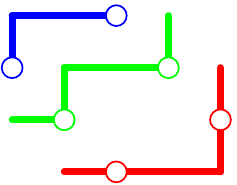} \end{array}, 
& \pi^{18} &=& & \hspace*{-1em} \begin{array}{c} \includegraphics[scale=\scaleNILP]{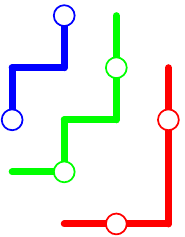} \end{array}, 
\end{align*}
\[\pi^{19} = \begin{array}{c} \includegraphics[scale=\scaleNILP]{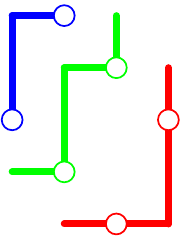} \end{array}, \quad 
\pi^{20} = \begin{array}{c} \includegraphics[scale=\scaleNILP]{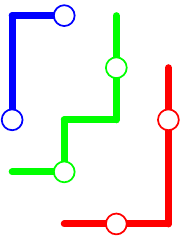} \end{array}, \quad
\pi^{21} = \begin{array}{c} \includegraphics[scale=\scaleNILP]{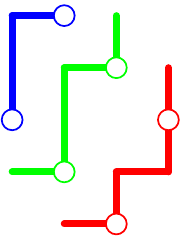} \end{array}, \quad 
\pi^{22} = \begin{array}{c} \includegraphics[scale=\scaleNILP]{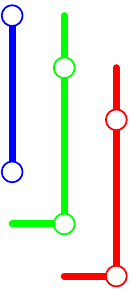} \end{array}. \]
\caption{The 22 triples of non-intersecting paths of $\P_4$. }
\label{fig:NILP_size_4}
\end{figure}

From the general techniques developed in~\cite{lind,GV} (and applied to the case of interest to us in~\cite{Viennot_SLC81}),  
we know that triples of non-intersecting lattice paths are enumerated by Baxter numbers. 
In the following, we exhibit a size-preserving bijection between these objects and Baxter TLTs. 

\subsection{A bijection between binary trees and pairs of non-intersecting lattice paths} 
The announced bijection actually extends a bijection between binary trees and pairs of non-intersecting lattice paths, which we describe below. 
In our context, a pair of non-intersecting lattice paths of size $n$ is a pair of lattice paths with unitary $N$ and $E$ steps, which never meet, 
starting at $(1,0)$ and $(0,1)$ and ending at $(n-i,i)$ and $(n-i-1,i+1)$ for some $i \in [0..(n-1)]$ (thus each of the two paths has $n-1$ steps). 

To any binary tree $B$, we associate a word $w(B)$ on the alphabet $\{L_\ell, L_r, E_\ell, E_r\}$ as follows. 
We complete $B$ by leaves, \ie we compute the complete binary tree $B_c$ whose internal nodes form the tree $B$. 
We perform the depth first traversal of $B_c$, starting on the left. We start with an empty word $w=\varepsilon$. 
Whenever a left leaf (resp. right leaf, left internal edge, right internal edge) is first encountered, 
we append $L_\ell$ (resp. $L_r$, $E_\ell$, $E_r$) to the end of $w$ (except for the first and the last leaves of $B_c$, in which case we do nothing). 
When the traversal of $B_c$ is over, we set $w(B) = w$.

We define $w_{1}(B)$ from $w(B)$ by deleting the letters $L_\ell$ and $L_r$ and by replacing letters $E_\ell$ (resp. $E_r$) by $N$ (resp. $E$). 
Similarly, we define $w_{2}(B)$ from $w(B)$ by deleting the letters $E_\ell$ and $E_r$ and by replacing letters $L_\ell$ (resp. $L_r$) by $E$ (resp. $N$). 
Finally, we set $\varphi(B) = (w_{1}(B),w_{2}(B))$. 
Our notational convention is that the word $w_1(B)$ (resp. $w_2(B)$) designate the upper (resp. lower) path of the pair, starting at $(0,1)$ (resp. $(1,0)$).

Note that treating the first and last leaves of $B_c$ like all others does not change much in the pair of words produced: 
it only results in replacing $w_{2}(B)$ by $E\cdot w_{2}(B) \cdot N$. 
In figures, we usually indicate those additional steps, before and after the circles marking the extremities of the paths. 

\begin{figure}[ht]
\begin{center}
$B = \raisebox{-0.5cm}{\includegraphics[scale=0.4]{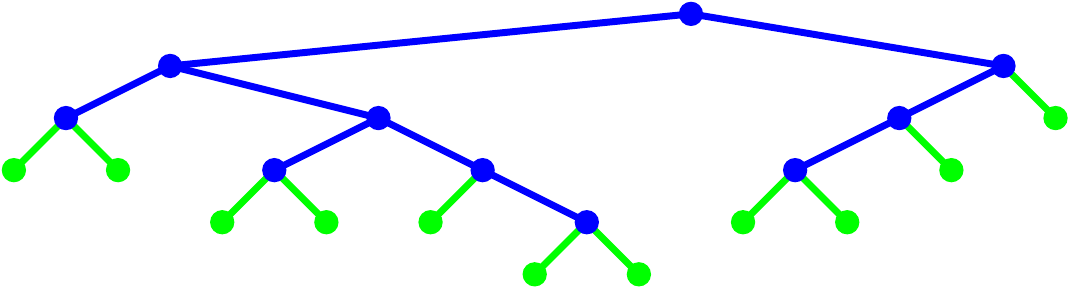}}$ \
and \
$\varphi(B) = \raisebox{-0.5cm}{\includegraphics[scale=0.5]{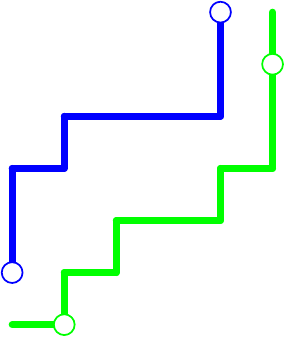}}$
\caption{The bijection $\varphi$. 
On the left, a binary tree $B$ is shown in blue, and the leaves added to obtain $B_c$ are drawn in green. 
On the right, the two paths $w_1(B)$ and $w_2(B)$ that define $\varphi(B)$ are shown respectively in blue and green. }
\label{fig:bij_tree_pair}
\end{center}
\end{figure} 

Figure~\ref{fig:bij_tree_pair} provides an example of this construction. 
For this particular tree $B$, we have $w(B) = E_\ell E_\ell L_r E_r E_\ell L_\ell L_r E_r L_\ell E_r L_\ell L_r E_r E_\ell E_\ell L_\ell L_r L_r$,
so that $w_{1}(B) = N N E N E E E N N$ and $w_{2}(B) = N E N E E N E N N$.

\begin{proposition}
$\varphi$ is a bijection between binary trees having $n$ nodes and pairs of non-intersecting lattice paths of size $n$. 
\label{prop:auxiliary_bijection}
\end{proposition}

Notice that $\varphi$ has already been defined in~\cite{DG}, where it is stated that it provides an alternative description 
of the bijection of~\cite{DV} between binary trees and parallelogram polyominoes. 
Proposition~\ref{prop:auxiliary_bijection} follows directly from this statement, which is however not proved in~\cite{DG}. 
For this reason, we prefer to give here a proof of Proposition~\ref{prop:auxiliary_bijection}.

\begin{proof}
Denoting $C_n = \frac{1}{n+1} { {2n}\choose n}$ the $n$-th Catalan number, we know that there are 
$C_n$ binary trees with $n$ nodes 
as well as $C_n$ pairs of non-intersecting lattice paths of size $n$~\cite{lind,GV}. 
Therefore, to prove that $\varphi$ is a bijection as claimed, 
it is enough to prove that the image of $\varphi$ is included in the set of pairs of non-intersecting lattice paths, 
and that $\varphi$ is injective. 

Let $B$ be a binary tree with $n$ nodes. 
We want to prove that $\varphi(B)$ is a pair of non-intersecting lattice paths of size $n$. 
For this purpose, let us define the correspondence $(\star)$ 
between the left (resp. right) internal edges and the right (resp. left) leaves of $B_c$ (except the first and the last leaves) as follows: 
to any left (resp. right) internal edge whose lower node is $x$, 
associate the right (resp. left) leaf that is reached when following (in $B_c$) only right (resp. left) edges from $x$ until a right (resp. left) leaf is reached. 
A simple observation, which is however a key fact, is that this correspondence is a bijection. 
This is illustrated on Figure~\ref{fig:bij_tree_pair_corresNodesLeaves}. 
\begin{figure}[ht]
\begin{center}
$\includegraphics[scale=0.6]{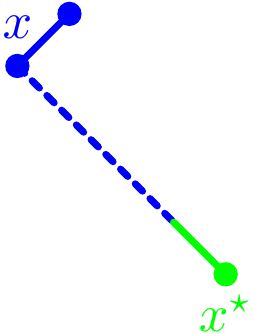}$ \
and \
$\includegraphics[scale=0.6]{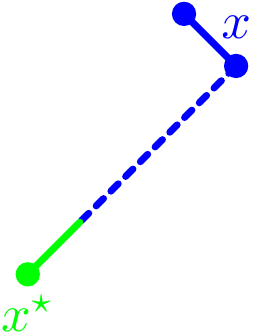}$
\qquad \vline \qquad 
$\includegraphics[scale=0.5]{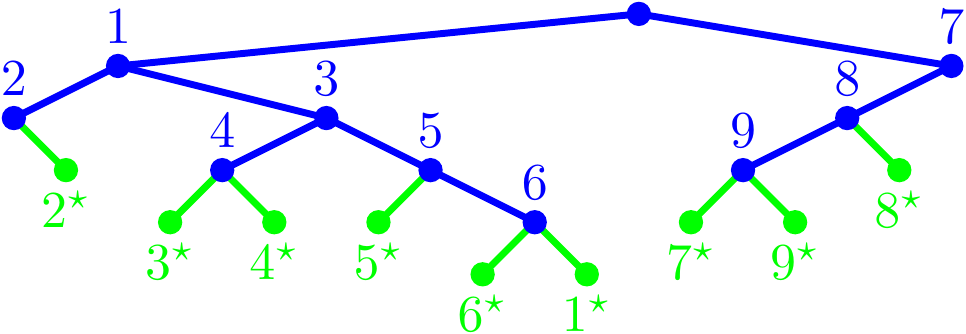}$
\caption{The correspondence $(\star)$: its definition (left) and an illustration with the example of Figure~\ref{fig:bij_tree_pair} (right). 
On the right, the correspondence is indicated by a labeling of the nodes by integers: 
the leaf labeled by $i$ is in correspondence with the internal edge whose lower node is labeled by $i$.}
\label{fig:bij_tree_pair_corresNodesLeaves}
\end{center}
\end{figure} 

Now let us denote by $i \in [0..(n-1)]$ the number of $N$ steps of $w_1(B)$. 
Since $w_1(B)$ starts at $(0,1)$, it ends at $(n-i-1,i+1)$. Because $(\star)$ is a bijection, 
$w_1(B)$ and $w_2(B)$ have the same number of $N$ (resp. $E$) steps, so that $w_2(B)$, starting at $(1,0)$, ends at $(n-i,i)$. 

Next, we claim that $w_1(B)$ is always strictly above $w_2(B)$. 
To prove this claim, let us consider the coordinates of the points visited by $w_1(B)$. 
They are of the form 
$($number of right edges visited$,1+$ number of left edges visited$)$ 
when we consider all instants of the depth first traversal of $B_c$. 
Similarly, the coordinates of the points of $w_2(B)$ are of the form $(1+$number of left leaves visited$,$ number of right leaves visited$)$ 
(always leaving aside the first and last leaves). 
So the claim will follow if we prove that at any instant of the traversal of $B_c$, 
if the number of left leaves visited (including the first one) is equal to the number of right edges visited, 
then the number of left edges visited is larger than or equal to the number of right leaves visited. 
This is easily proved using the correspondence $(\star)$ between leaves and edges of $B$. 
Indeed, in the traversal of $B$, any left (resp. right) edge is visited before the corresponding right (resp. left) leaf. 

\smallskip

It remains to prove that $\varphi$ is injective. 
Consider $B$ and $B'$ two different binary trees with $n$ nodes. 
Of course, the words $w(B)$ and $w(B')$ encoding the depth-first traversals of $B_c$ and $B'_c$ are different. 
We prove that $w(B) \neq w(B')$ implies that $\varphi(B) \neq \varphi(B')$. 

Consider the first time $w(B)$ and $w(B')$ differ: $w(B)_j \neq w(B')_j$ while  $w(B)_k = w(B')_k$ for all $k<j$. 
The possible values for the pair of letters $(w(B)_j, w(B')_j)$ (up to the order) are described in the following table, 
together with the subsequent difference between $\varphi(B)$ and $\varphi(B')$. 
In this table, we denote by $i_1$ (resp. $i_2$) the number of letters $E_r$ or $E_\ell$ (resp. $L_r$ or $L_\ell$) among the first $j$ letters of $w(B)$. 
Therefore, the letter $w(B)_j$ corresponds either to $w_1(B)_{i_1}$ or to $w_2(B)_{i_2}$ (depending on whether $w(B)_j \in \{ E_r ,E_\ell\}$ or $w(B)_j \in \{ L_r ,L_\ell\}$).

\begin{center}
\begin{tabular}{|c|c|l|l|}
\hline 
$w(B)_j$ & $w(B')_j$ & Difference & Fact(s) used\\
\hline
$E_\ell$ & $E_r$ & $w_1(B)_{i_1} =N$ and $w_1(B')_{i_1} =E$ & (obvious)\\
\hline
$E_\ell$ & $L_\ell$ & $w_1(B)_{i_1} =N$ and $w_1(B')_{i_1} =E$ & $(a)$\\
\hline
$E_\ell$ & $L_r$ & $w_1(B)_{i_1} =N$ and $w_1(B')_{i_1} =E$ & $(a)$\\
\hline
$E_r$ & $L_r$ & $w_2(B)_{i_2} =E$ and $w_2(B')_{i_2} =N$ & $(b)$\\
\hline
$E_r$ & $L_\ell$ & In this case, $w(B)_{j-1} \neq w(B')_{j-1}$, & $(c)$ and $(d)$\\
& & contradicting the minimality of $j$. & \\
\hline
$L_r$ & $L_\ell$ & $w_2(B)_{i_2} =N$ and $w_2(B')_{i_2} =E$ & (obvious)\\
\hline
\end{tabular}
\end{center}

All these cases follow from the following facts: 
\begin{itemize}
 \item[$(a)$] When the traversal reaches a leaf, then the next edge to be discovered is a right edge. 
 \item[$(b)$] When the traversal reaches an edge, then the next leaf to be discovered is a left leaf. 
 \item[$(c)$] In any depth-first traversal word $w(B)$, any letter $E_r$ follows a letter $L_r$ or $L_\ell$.
 \item[$(d)$] In any depth-first traversal word $w(B)$, any letter $L_\ell$ follows a letter $E_r$ or $E_\ell$. \qedhere
\end{itemize}
\end{proof}

%%%%%%%%%%%%%%
\subsection{Extension to a bijection between $\T_n$ and $\P_n$}

To any TLT $T \in \T_n$, we may associate a binary tree $B(T)$ with $n$ nodes, as explained in Section~\ref{sec:tlt} (see also an example in Figure~\ref{fig:psi}, left). 
We define $\Phi_{\P}(T) = (w_{top}(T),w_{middle}(T), $ $w_{bottom}(T))$ as follows:
\begin{itemize}
 \item $w_{top}(T)$, from $(-1,2)$ to $(n-i-2,i+2)$, is $w_{1}(B(T))$.
 \item $w_{middle}(T)$, from $(0,1)$ to $(n-i-1,i+1)$, is $w_{2}(B(T))$.
 \item $w_{bottom}(T)$, from $(1,0)$ to $(n-i,i)$, is the Southeast border of $T$, except the first and the last edge.
\end{itemize}
Figure~\ref{fig:psi} illustrates this construction. 
Note that the TLTs of Figures~\ref{fig:psi} and~\ref{fig:uniqueNAT} (right) (p.\pageref{fig:uniqueNAT}) differ only by their underlying Ferrers diagrams.

\begin{figure}[ht]
\begin{center}
$T = \raisebox{-1.5cm}{\includegraphics[scale=0.7]{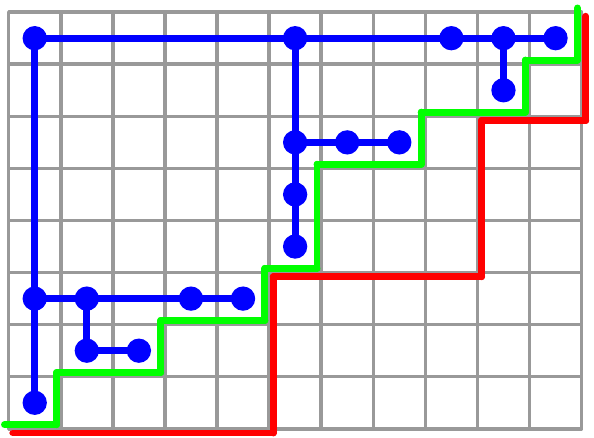}}$ \
and \
$\Phi_{\P}(T) = \raisebox{-1.5cm}{\includegraphics[scale=0.7]{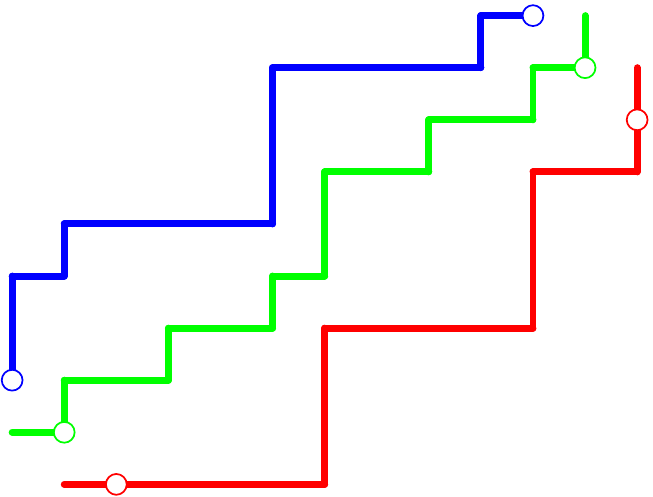}}$
\caption{The bijection $\Phi_{\P}$.}
 \label{fig:psi}
\end{center}
\end{figure} 

Recall from Corollary~\ref{cor:uniqueNAT} that any binary tree is the underlying tree of a unique rectangular Baxter TLT. 
This fact will be useful both for proving that $\Phi_{\P}$ is well-defined 
and that it is a bijection between $\T_n$ and $\P_n$.

\begin{lemma}
$\Phi_{\P}$ is well-defined, \ie for any $T \in \T_n$, $\Phi_{\P}(T)$ is a triple of non-intersecting lattice paths of size $n$. 
\label{lem:phiP-well-defined}
\end{lemma}

\begin{proof}
From Proposition~\ref{prop:auxiliary_bijection}, 
we know that $w_{top}(T)$ and $w_{middle}(T)$ are a pair of non-intersecting lattice paths of size $n$. 
So the conclusion will follow if we prove that $w_{middle}(T)$ and $w_{bottom}(T)$ also form a pair of non-intersecting lattice paths of size $n$. 
This is an immediate consequence of the following claim (which we then prove): $E \cdot w_{middle}(T) \cdot N$ can also be interpreted as 
the Southeast border of the thinnest Ferrers diagram containing all the points of $T$ (see Figure~\ref{fig:psi}, left). 

From Corollary~\ref{cor:uniqueNAT}, we know that 
there is a unique TLT $T'$ of rectangular shape avoiding the patterns $\Th$ and $\Tv$ 
whose underlying binary tree is $B(T)$. 
Moreover, the proof of Corollary~\ref{cor:uniqueNAT} provides a description of $T'$. 
By uniqueness, the points of $T$ and $T'$ are located in the same cells. 
Consequently, it is enough to show that the
Southeast border of the thinnest Ferrers diagram containing all the points of $T'$ 
is $E \cdot w_{middle}(T) \cdot N$.  

From the description of $T'$ in the proof of Corollary~\ref{cor:uniqueNAT}, 
we see that the Southeast border of the thinnest Ferrers diagram containing all the points of $T'$ 
is nothing but the reading of the leaves of $B(T')_c$ in the depth-first traversal of $B(T')_c$ 
(with the encoding $E$ for a left leaf and $N$ for a right leaf). 
By definition, $E \cdot w_{middle}(T) \cdot N$ describes the leaves of $B(T)_c$ that are visited in the depth-first traversal of $B(T)_c$. 
Because $B(T)=B(T')$, the conclusion follows. 
\end{proof}

\begin{theorem}
For any $n$, $\Phi_{\P}$ is a bijection between $\T_n$ and $\P_n$. 
\end{theorem}

\begin{proof}
Lemma~\ref{lem:phiP-well-defined} ensures that the image of $\T_n$ by $\Phi_{\P}$ is included in $\P_n$. 
In addition, from Theorem~\ref{thm:TTW} and~\cite[for instance]{Gir}, the cardinality of $\T_n$ is $Bax_n$, 
and the same holds for $\P_n$~\cite{Viennot_SLC81}. 
So it is enough to prove that $\Phi_{\P}$ is a bijection. 
By Corollary~\ref{cor:uniqueNAT}, a Baxter TLT $T \in \T_n$ is uniquely characterized by the pair $(B(T),w_{bottom}(T))$. 
Moreover, by Proposition~\ref{prop:auxiliary_bijection}, we may associate with $B(T)$ the pair $(w_{top}(T),w_{middle}(T))$, 
and this correspondence is bijective. 
Therefore, the correspondence between $T \in \T_n$ and $(w_{top}(T),w_{middle}(T),w_{bottom}(T)) \in \P_n$ is a bijection. 
\end{proof}

\subsection{Refined enumeration using the Lindstr\"om-Gessel-Viennot lemma}

We derive easily from the Lindstr\"om-Gessel-Viennot lemma~\cite{lind,GV} 
a refined enumeration of \triples, according to the parameters shown on Figure~\ref{fig:refinement_tlt_triple}. 
This yields a refined enumeration of Baxter TLTs via $\Phi_{\P}$, 
and subsequently of PFPs and permutations avoiding $3-14-2$ and $3-41-2$ via $\Phi_{\F}$ and $\Phi_\B$.

\begin{figure}[ht]
\centering{\includegraphics[scale=0.75]{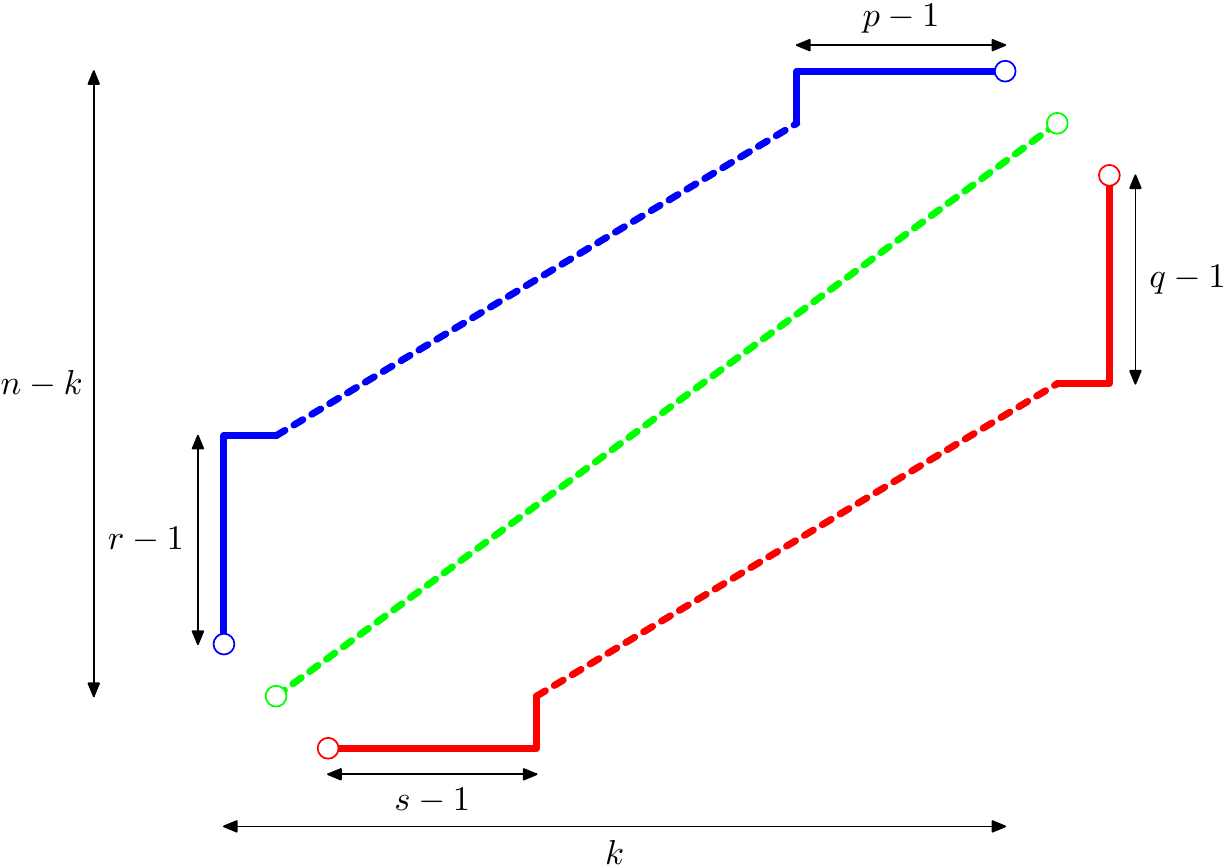}}
\caption{Parameters for the refined enumeration of \triples in Lemma~\ref{lem:LGV}.}
\label{fig:refinement_tlt_triple}
\end{figure} 

\begin{lemma}
The number of \triples of size $n$ such that 
each path has $k$ $E$ steps and $n-1-k$ $N$ steps, 
the upper path starts with $r-1$ $N$ steps followed by an $E$ step, and ends with $p-1$ $E$ steps preceded by a $N$ step, 
and the lower path starts with $s-1$ $E$ steps followed by a $N$ step, and ends with $q-1$ $N$ steps preceded by an $E$ step 
(see Figure~\ref{fig:refinement_tlt_triple}) 
is given by the determinant: 
\[ LGV(n,k,r,p,s,q)= 
\left| \begin{array}{ccc}
{n-1-r-p   \choose k-p} & {n-1-p   \choose k-p} & {n-1-s-p   \choose k-s-p} \\
{n-1-r     \choose k  } & {n-1     \choose k  } & {n-1-s     \choose k-s  } \\
{n-1-r  -q \choose k  } & {n-1  -q \choose k  } & {n-1-s  -q \choose k-s  }
\end{array} \right|.
\]
\label{lem:LGV}
\end{lemma}

\begin{proof}
This is a simple application of the results of~\cite{lind,GV}. 
\end{proof}

\begin{corollary}\label{coro:refinement_tlt_triple}
The determinant $LGV(n,k,r,p,s,q)$ also counts 
the number of Baxter TLTs of size $n$, 
with $k+1$ columns, 
$r$ points in the first column, 
and $p-1$ columns to the right of the rightmost column containing at least two 
points\footnote{Necessarily, each of these $p-1$ columns contains a single point.},
and such that the Southeast border of the Ferrers diagram of the TLT 
starts with $s$ horizontal steps and ends with $q$ vertical steps.
\end{corollary}

For example, the TLT $T$ of Figure~\ref{fig:psi} has parameters 
$n=18$, $k=10$, $r=3$, $p=2$, $s=5$, and $q=2$. 

\begin{proof}
It can be easily checked that all parameters in Lemma~\ref{lem:LGV} 
are translated on TLTs through $\Phi_{\P}$ 
as stated in Corollary~\ref{coro:refinement_tlt_triple}. 
For instance, the upper path of $\Phi_{\P}(T)$ ends with exactly $p-1$ $E$ steps 
if and only if exactly the last $p-1$ internal edges of $B(T)$ in the depth-first search are right edges, 
which translates into exactly the $p-1$ rightmost columns of $T$ containing a single point. 
\end{proof}

\begin{corollary}
\label{coro:refinement_tlt_PFP+perm}
The number of PFPs of size $n$, in a bounding rectangle of height $n-k$ and width $k+1$, 
with $r$ tiles whose left edge is supported by the left edge of the bounding rectangle, 
and such that the rightmost vertical line that supports the left edge of at least two tiles is at distance $p$ from the right edge of the bounding rectangle
is $\sum_{q,s} LGV(n,k,r,p,s,q)$. 

The number of permutations of size $n$, avoiding $3-14-2$ and $3-41-2$, with $k$ ascents and $r$ left-to-right minima 
is  $\sum_{p,q,s} LGV(n,k,r,p,s,q)$.
\end{corollary}

\begin{proof}
It is enough to check that the parameters $n$, $k$, $r$ and $p$ in Corollary~\ref{coro:refinement_tlt_triple} 
are translated on PFPs and permutations avoiding $3-14-2$ and $3-41-2$ via $\Phi_{\F}$ and $\Phi_\B$ 
as expressed in the statement of Corollary~\ref{coro:refinement_tlt_PFP+perm}. 

Because $\Phi_{\F}$ maps $\T_{(n-k,k+1)}$ to $\F_{(n-k,k+1)}$, we are only left with the interpretation of the parameters $r$ and $p$ on PFPs. 
They follow easily from the description of $\Phi_{\F}$, 
since any point $x$ of $T$ is associated with a tile whose top-left corner is the cell containing $x$. 

Consider now $T \in \T_n$ and $\sigma = \Phi_\B(T) \in \B_n$. 
From Corollary~\ref{cor:TTW-ascent+descent}, the number of ascents in $\sigma$ is the number of column-extremal points in $T$. 
Since there is one such point in each column except the first, this solves the case of parameter $k$. 
From Corollary~\ref{cor:LtoR-min}, there is one left-to-right minimum for each point in the first column, proving the statement for the parameter $r$. 
\end{proof}

%%%%%%%%%%%%%%%%%%%%
\section{Specializations of the bijections}
\label{sec:special}

Using the underlying tree structure of TLTs, we define a subfamily (denoted $\TT$) of Baxter TLTs whose nice combinatorial properties are explained in this section. 

For each of the studied bijections, $\Phi_{\P}$, $\Phi_{\F}$ and $\Phi_{\B}$, 
we consider its restriction to the domain $\TT$. 
These restrictions, denoted $\phiPP$, $\phiFF$ and $\phiBB$, provide bijections between 
$\TT_n$ and subclasses of $\P_n$, $\F_n$ and $\B_n$, which we denote $\PP_n$, $\FF_n$ and $\BB_n$. 
As we shall see, the family $\PP_n$ is well-known, 
giving easy access to the enumeration of these restricted families, 
the definition of $\FF_n$ involves a new type of constraint in floorplans, 
and the permutations of $\BB_n$ are natural combinatorial objects. 
In particular, the enumeration of the permutations of $\BB_n$ triggers 
some intriguing enumerative problems, which we discuss at the end of this section. 

For now, we define the subfamily $\TT$ and state a lemma which is essential to analyze the restrictions of our bijections.

\begin{definition}\label{def:complete-BTLT}
A \emph{complete Baxter TLT} is a Baxter TLT whose underlying tree is a complete binary tree. 

An \emph{almost complete Baxter TLT} of size $n$ is a Baxter TLT of size $n$ whose underlying tree is \emph{almost complete}: 
namely, it is complete binary tree from which the following have been removed: 
 \begin{itemize}
 \item the leaf $\ell$ that is reached from the root when following only left edges, if $n$ is even;
 \item $\ell$ and the leaf $r$ that is reached from the root when following only right edges, if $n$ is odd;
\end{itemize}
We denote by $\TT_n$ the set of almost complete Baxter TLT of size $n$. 
\label{def:restricted_TLT}
\end{definition}

Figures~\ref{atn-small} and~\ref{atn} show some examples.

\begin{figure}[ht]
\begin{center}
\begin{tabular}{c|cc|c|c}
\tikz[scale=0.8]{\draw (0,0) [fill] circle (.1); 
\draw (0,-1) [fill] circle (.1);
\draw (1,0) [fill] circle (.1);
\draw (0,-1) -- (0,0);
\draw (1,0) -- (0,0);
} \qquad & \qquad 
\includegraphics[scale=1.0]{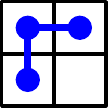} & 
\includegraphics[scale=1.0]{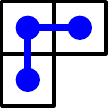} \qquad & \qquad  
\includegraphics[scale=1.0]{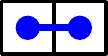} \qquad & \qquad
\includegraphics[scale=1.0]{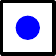} \\ 
\end{tabular}
\end{center}
The complete binary tree with $3$ vertices (denoted $t_3$), \\ with all the complete Baxter TLTs of size $3$, \\
and all the almost complete Baxter TLTs of size $2$ and $1$ (obtained from $t_3$).
\\
\begin{center}
\begin{tabular}{c|ccc|c|c}
\tikz[scale=0.8]{\draw (0,0) [fill] circle (.1); 
\draw (0,-1) [fill] circle (.1);
\draw (1,0) [fill] circle (.1);
\draw (-1,-1) [fill] circle (.1);
\draw (-1,0) [fill] circle (.1);
\draw (0,-1) -- (0,0);
\draw (1,0) -- (0,0);
\draw (-1,0) -- (-1,-1);
\draw(0,0) -- (-1,0);
} \qquad & \qquad 
\includegraphics[scale=1.0]{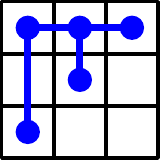} & 
\includegraphics[scale=1.0]{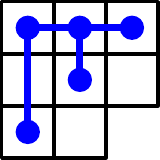} & 
\includegraphics[scale=1.0]{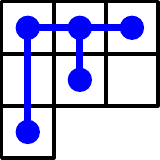} \qquad & \qquad
\includegraphics[scale=1.0]{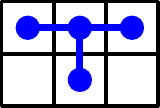} \qquad & \qquad
\includegraphics[scale=1.0]{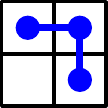} \\ 
& &
\includegraphics[scale=1.0]{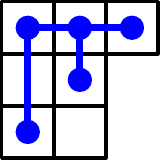} & 
\includegraphics[scale=1.0]{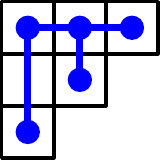} \qquad & \qquad
\includegraphics[scale=1.0]{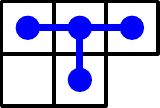}\qquad  &  
 \\ 
\end{tabular}

\bigskip

\begin{tabular}{c|ccc|c|c}
\tikz[scale=0.8]{
\phantom{\draw (-1,0) [fill] circle (.1);}
\draw (0,0) [fill] circle (.1); 
\draw (0,-1) [fill] circle (.1);
\draw (1,0) [fill] circle (.1);
\draw (0,1) [fill] circle (.1);
\draw (1,1) [fill] circle (.1);
\draw (0,-1) -- (0,0);
\draw (1,0) -- (0,0);
\draw (0,1) -- (1,1);
\draw (0,0) -- (0,1);
} \qquad & \qquad 
\includegraphics[scale=1.0]{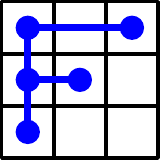} & 
\includegraphics[scale=1.0]{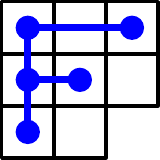} & 
\includegraphics[scale=1.0]{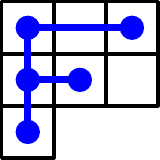} \qquad & \qquad 
\includegraphics[scale=1.0]{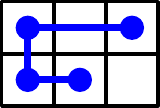} \qquad & \qquad 
\includegraphics[scale=1.0]{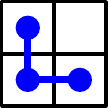} \\ 
& & 
\includegraphics[scale=1.0]{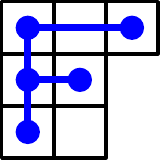} & 
\includegraphics[scale=1.0]{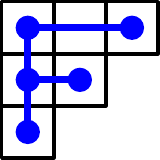} \qquad & \qquad
\includegraphics[scale=1.0]{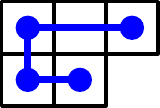} \qquad & 
 \\ 
\end{tabular}
\end{center}
The complete binary trees with $5$ vertices (denoted $t_5$ and $t'_5$), \\ with all the complete Baxter TLTs of size $5$, \\
and all the almost complete Baxter TLTs of size $4$ and $3$ (obtained from $t_5$ and $t'_5$).
\caption{Small complete and almost complete Baxter TLTs. (For consistency with TLTs, 
the roots of binary trees are drawn in the top left.)}\label{atn-small}
\end{figure}

\begin{figure}[ht]
$\begin{array}{c|cc}
 
\includegraphics[scale=1.0]{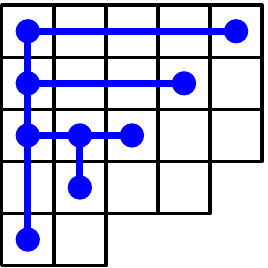}
\qquad & \qquad 
\includegraphics[scale=1.0]{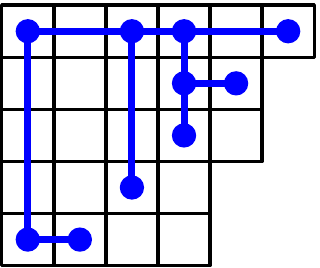}
& \quad
\includegraphics[scale=1.0]{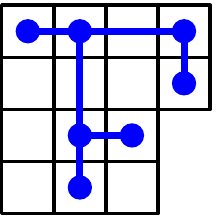} \\
\end{array}$
\caption{Left: A complete Baxter TLT of size $9$. Right: Two almost complete Baxter TLTs, of respective size $10$ and $7$.}\label{atn}
\end{figure}

\begin{lemma}
Let $T$ be a Baxter TLT, and define its \emph{leaves} as its points which correspond to leaves of the underlying tree of $T$. 
The followings are equivalent: 
\begin{itemize}
 \item $T$ is a complete Baxter TLT; 
 \item the leaves of $T$ form a staircase shape, 
\ie they are located on the Southwest-Northeast diagonal which starts at the bottommost point of the first column of $T$, 
and occupy every cell of this diagonal. 
\end{itemize}
\label{lem:escalier}
\end{lemma}

\begin{proof}
We note that a complete Baxter TLT is necessarily of odd size. 
We also observe that a TLT satisfying the second condition above is also necessarily of odd size: 
indeed, recalling that there are no empty rows nor columns in TLTs, it is forced to have $n$ rows and $n$ columns for some $n$, and therefore $2n-1$ points.

We prove the claimed statement for all Baxter TLTs of size $2n+1$ by induction on $n$. 

The  base case $n=0$ is obvious.

Let $T$ be a Baxter TLT of size $2n+1$ for $n>0$. 
Consider the bi-partition $(L,R)$ of the non-root points of $T$ where 
$L$ (resp. $R$) contains all points of $T$  
in the left (resp. right) subtree pending from the root of the underlying tree of $T$.
From Proposition~\ref{prop:decomposition_Baxter_TLT} and Corollary~\ref{cor:decomposition_Baxter_TLT}, 
$T$ can be decomposed into $4$ blocks as 
$\begin{array}{c} 
\tikz[scale=0.2]{\node at (-1,1) {$A$};\node at (-1,-1) {$B$};\node at (1,-1) {$C$};\node at (1,1) {$D$}; \draw (0,1.5) -- (0,-1.5); \draw (1.5,0) -- (-1.5,0);}
\end{array}$, 
with $A$ containing only the root of $T$, $B$ (resp. $D$) containing all points of $L$ (resp. $R$), and $C$ containing no points. 
We recall that a binary tree is complete if and only if the left and right subtrees pending from its root are also complete binary trees. 

It follows that, if $T$ is a complete Baxter TLT, then $B$ and $D$ are also complete Baxter TLTs. 
They are smaller than $T$, and by induction have their leaves which form a staircase shape. 
This consequently also holds for the leaves of $T$, which form a staircase shape obtained from the concatenation of those of $B$ and $D$. 
Conversely, if the leaves of $T$ form a staircase shape, then it also holds for those of $B$ and $D$, which are therefore complete Baxter TLTs, 
implying that $T$ also is a complete Baxter TLT. 
\end{proof}

\subsection{Restriction on lattice paths}

\begin{proposition}
Let $\PP_n$ be the set of pairs of Dyck paths with $n$ steps, if $n$ is even
(resp. with $n+1$ and $n-1$ steps respectively, if $n$ is odd). 
Define $\phiPbar$ as follows: 
\begin{itemize}
 \item if $n$ is even, then for all $T \in \TT_n$, writing $\phiPP(T) = (w_{top},w_{middle},w_{bottom})$, we set $\phiPbar(T) = (N \cdot w_{top},w_{bottom} \cdot N)$;
 \item if $n$ is odd, then for all $T \in \TT_n$, writing $\phiPP(T) = (w_{top},w_{middle},w_{bottom})$, we set $\phiPbar(T) = (N \cdot w_{top} \cdot E,w_{bottom})$.
\end{itemize}
Then, it holds that $\phiPbar$ is a bijection between $\TT_n$ and $\PP_n$. 
\label{prop:restriction_paths}
\end{proposition}

\begin{proof}
To prove this statement, we must keep in mind the interpretation of the three paths of $\phiPP(T)= (w_{top},w_{middle},w_{bottom})$ for $T \in \T_n$ 
explained in the proof of Lemma~\ref{lem:phiP-well-defined}. 
In particular, the following holds. 
\begin{itemize}
 \item $w_{top}$ encodes the underlying tree structure of $T$. In the present case where $T \in \TT_n$, this implies that $N \cdot w_{top}$ when $n$ is even and $N \cdot w_{top} \cdot E$ when $n$ is odd 
 encodes the complete binary tree from which $T$ was built (hence, in particular, is a generic Dyck path). 
 \item $w_{middle}$ has been obtained from the path which follows the leaves of $T$ by removing the first and the last steps. 
 For $T \in \TT_n$, Lemma~\ref{lem:escalier} implies that $w_{middle}$ is an alternation of $N$ and $E$ steps starting with an $E$. 
 \item $w_{bottom}$ is the Southeast border of $T$ from which the first and last steps have been removed. 
 Since $w_{bottom}$ is a path located to the Southeast of $w_{middle}$, and given the very specific form of $w_{middle}$ in our case, 
 this implies that $w_{bottom} \cdot N$ if $n$ is even (resp. $w_{bottom}$ if $n$ is odd) is the symmetric of a generic Dyck path w.r.t. the main diagonal. 
\end{itemize}

Summing up, this shows that pairs $(d_t,d_b)$ of Dyck paths in $\PP_n$ are (up to removing initial and/or final steps as described above) 
in bijection with \triples of size $n$ whose middle path is an alternation of single $N$ and $E$ steps starting with $E$, 
which are themselves in bijection with $\TT_n$ by Lemma~\ref{lem:escalier}, thus concluding the proof.
\end{proof}

Figure~\ref{atn-triple} shows the \triples\ and the corresponding pairs of Dyck paths of the two almost complete Baxter TLTs of even and odd size given in Figure~\ref{atn}. In this figure, dashed steps correspond
to the leaves removed (according to parity) from the complete binary tree in Definition~\ref{def:complete-BTLT}.

\begin{figure}[ht]
$\begin{array}{c|c}
\includegraphics[scale=1.0]{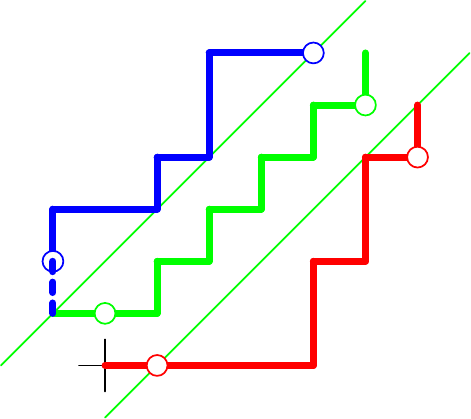} \qquad  & \qquad 
\includegraphics[scale=1.0]{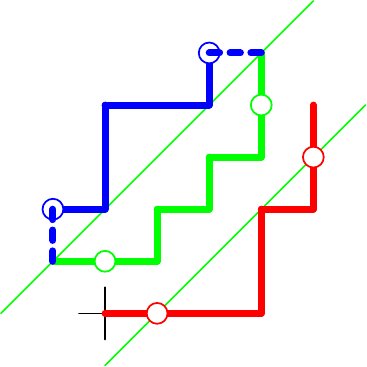} \\
\includegraphics[scale=1.0]{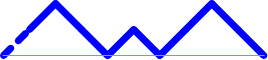} \text{ \ and \ }
\includegraphics[scale=1.0]{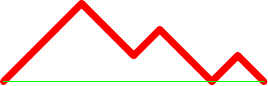} \qquad  & \qquad 
\includegraphics[scale=1.0]{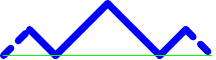} \text{ \ and \ }
\includegraphics[scale=1.0]{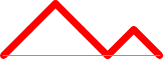}
\end{array}$
\caption{Two \triples\ and the corresponding pairs of Dyck paths.}\label{atn-triple}
\end{figure}

Proposition~\ref{prop:restriction_paths} has an immediate enumerative consequence: 
\begin{corollary}
For any $n$, the cardinality of $\TT_{2n}$ is $C_n^2$, 
and the one of $\TT_{2n+1}$ is $C_n\cdot C_{n+1}$, 
where $C_n = \frac{1}{n+1} {{2n} \choose n}$ is the $n$-th Catalan number. 
\label{cor:product_of_Catalan}
\end{corollary}

\subsection{Restriction on floorplans}

In the characterization of the image of $\TT_n$ under the bijection $\Phi_{\F}$ between Baxter TLTs and packed floorplans, 
we are led to defining a new class of floorplans with constraints along the Southwest-Northeast diagonal. 

Recall from Section~\ref{sec:floorplans} that floorplans are rectangular partitions of a rectangle such that 
every pair of segments with non-empty intersection forms a {\em T-junction}.

\begin{definition}
An \emph{alternating floorplan} of size $n$ is a 
partition of a rectangle $R$ 
of width $\lceil \frac{n+1}{2} \rceil$ and height $\lfloor \frac{n+1}{2} \rfloor$ 
into $n$ rectangular tiles whose sides have integer lengths 
such that the path from the Southwest corner to the Northeast corner of $R$ 
which moves alternately one unit step East and one unit step North (starting with East) is included in the boundaries of partitioning rectangles of $F$. 
We call this path the \emph{alternating path} of $F$. 

We denote by $\FF_n$ the set of alternating floorplans of size $n$. 
\label{def:alternating_floorplan}
\end{definition}

Figure~\ref{fig:alternating_floorplan} (left) shows an example of an alternating floorplan. 

\begin{figure}[ht]
$$
\begin{array}{ccc}
\includegraphics[scale=0.8]{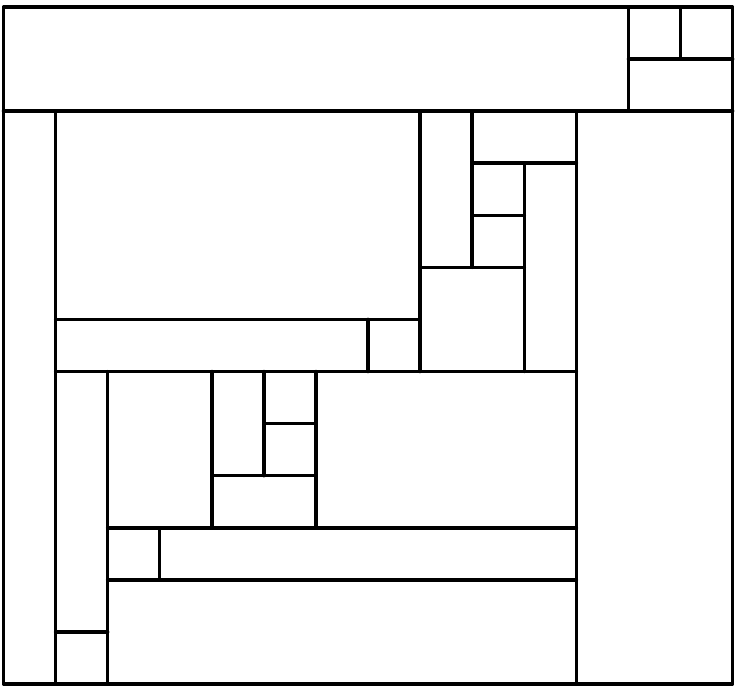} & 
\qquad & 
\includegraphics[scale=0.8]{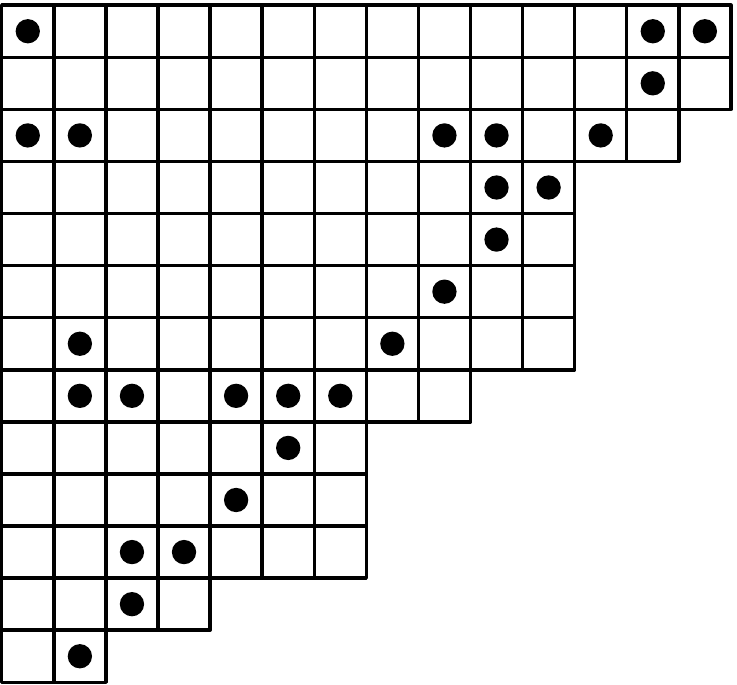} 
\end{array}
$$
\caption{Left: An alternating floorplan $F$ of size $26$. Right: The almost complete Baxter TLT $T$ of size $26$ such that $\phiFF(T)=F$.}
\label{fig:alternating_floorplan}
\end{figure}

We will show that $\FF_n$ is indeed included in the family $\F_n$ of packed floorplans. 
To this end, we first state an important property of the alternating paths of alternating floorplans. 

\begin{lemma}
In any alternating floorplan $F$, 
every tile has either its bottom-right corner or its topleft corner 
on the alternating path of $F$. 
\label{lem:all_tiles_on_alt_path}
\end{lemma}

\begin{proof}
Let us denote by $n$ the size of $F$. 
By definition, the alternating path of $F$ contains $\lceil n/2 \rceil$ factors $EN$ and $\lfloor n/2 \rfloor$ factors $NE$, 
each corresponding to a bottom-right corner or a topleft corner of a tile of $F$, respectively. 
This gives a total of $n$ tiles, necessarily all distinct, and therefore all tiles of $F$ have their bottom-right or topleft corner 
on the alternating path as claimed. 
\end{proof}

When a tile $t$ of an alternating floorplan $F$ has its topleft (resp. bottom-right) corner on the alternating path of $F$, 
we say that $t$ is \emph{below} (resp. \emph{above}) the alternating path of $F$. 

\begin{lemma}
Every alternating floorplan is a packed floorplan.
\label{lem:alternating_floorplan_are_packed}
\end{lemma}

\begin{proof}
Let $F$ be an alternating floorplan of size $n$. To prove that $F$ is a packed floorplan, 
comparing with Definition~\ref{defi:PFP}, 
we only need to check that $F$ avoids the pattern $\Fp$.

Assume that two tiles $t$ and $t'$ of $F$ form a $\Fp$ pattern. 
Using Lemma~\ref{lem:all_tiles_on_alt_path}, 
we distinguish three cases (w.l.o.g. up to exchanging $t$ and $t'$): either $t$ is above the alternating path 
and $t'$ is below it, 
or $t$ and $t$' are both above the alternating path, or they are both below. 
In the first case, because of the alternating path condition, 
the bottom-right corner of $t$ is forced to sit either above and to the right, or below and to the left, 
of the topleft corner of $t'$. 
But this is impossible when $t$ and $t'$ form a pattern $\Fp$. 
In the second case, the bottom-right corners of $t$ and $t'$ are forced to sit as $\begin{array}{c}\tikz[scale=0.2]{\node at (0,-1) {$\lrcorner$};\node at (1,0) {$\lrcorner$};}\end{array}$
by the alternating path condition, 
but as $\begin{array}{c}\tikz[scale=0.2]{\node at (1,-1) {$\lrcorner$};\node at (0,0) {$\lrcorner$};}\end{array}$
by the occurrence of $\Fp$, also yielding a contradiction. 
The third case is similar, considering the topleft corners of $t$ and $t'$. 
This concludes the proof that $F$ avoids $\Fp$, hence is a packed floorplan. 
\end{proof}

\begin{proposition}
$\phiFF$ is a bijection between $\TT_n$ and $\FF_n$ . 
\label{prop:restriction_PFP}
\end{proposition}

\begin{proof} 
Consider an alternating floorplan $F \in \FF_n$ and its preimage $T$ under $\Phi_{\F}$. 
We know that $T$ is a Baxter TLT. 
To prove that it is almost complete, 
we use Lemma~\ref{lem:all_tiles_on_alt_path} to ensure that $T$ as $\lfloor n/2 \rfloor$ points on the main diagonal 
(shifted by one unit to the right): 
namely, those corresponding to the topleft corners of the tiles below the alternating path of $F$. 
Then, Lemma~\ref{lem:escalier} ensures that $T$ is an almost complete Baxter TLT. 

Conversely, for an almost complete Baxter TLT $T$, 
we show that $\Phi_{\F}(T)$ is an alternating floorplan 
by ensuring that it contains a valid alternating path. 
Lemma~\ref{lem:escalier} forces the placement of the topleft corners of some tiles of $F$, 
namely, those whose topleft corner is a leaf of $T$. 
As a consequence, the path $E(NE)^{\lfloor n/2 \rfloor}N^{\delta}$ is supported by the sides of the tiles of $F$, 
for $\delta = 0$ if $n$ is even and $\delta=1$ is $n$ is odd. 
More precisely, the first $E$ step is supported by the bottom edge of the bounding rectangle of $F$ (hence by its bottom-leftmost tile), 
each $NE$ factor surrounds the topleft corner of a tile corresponding to a leaf of $T$, 
and, in case $n$ is odd, the final $N$ step is supported by the right edge of the bounding rectangle of $F$ (hence by its top-rightmost tile).
\end{proof}

Figure~\ref{fig:alternating_floorplan} (right) shows an example of an almost complete Baxter TLT which is in bijection by $\phiFF$ with the alternating floorplan on the left of this figure. 

\subsection{Restriction on permutations}

We recall that a permutation $\sigma$ is \emph{alternating} if the comparisons between consecutive elements alternate between ascents and descents, 
that is to say if $\sigma(1) > \sigma(2) <\sigma(3)>\sigma(4) < \dots$ or $\sigma(1) < \sigma(2) >\sigma(3)<\sigma(4) > \dots$. 
Alternating permutations arise naturally when studying the restriction of our bijection $\Phi_{\B}$ to $\TT_n$.

\begin{proposition}
Let $\BB_n$ be the set of permutations in $\B_n$ that are alternating and start with an ascent. 
$\phiBB$ is a bijection between $\TT_n$ and $\BB_n$ . 
\label{prop:restriction_twisted_Baxter}
\end{proposition}

\begin{proof}
We first prove that the image of $\TT_n$ by $\phiBB$ is included in $\BB_n$. 
So, consider $T \in \TT_n$ and its image $\sigma=\phiBB(T)$. We know that $\sigma$ is in $\B_n$, 
and want to prove that $\sigma$ is alternating starting with an ascent. 

By construction, $\sigma$ is the sequence of $\iso$-labels read along the Southeast border of $T$. 
Recalling the rule for propagation of $\iso$-labels (and in particular, the first item in Subsection~\ref{subsec:bijectionTLTsPermutations}), 
and the placement of the leaves of $T$ (see Lemma~\ref{lem:escalier}), 
it follows that $\sigma$ is also read on the path ``inside'' $T$ along the boundary determined by the leaves. 
More precisely, we mean that $\sigma$ is obtained by reading the $\iso$-labels of the following cells of $T$, in this order: 
the bottommost cell of the first column, then its right neighbor (which is the first leaf), then the cell above it, then its right neighbor (which is the second leaf), 
and all cells subsequently met by moving alternately one cell to the top and one cell to the right, until the rightmost cell of the top row is reached. 
See Figure~\ref{atn-perm} for an illustration of this fact. 

So, $\sigma$ alternates between reading $\iso$-labels of leaves and $\iso$-labels of pointed or empty cells corresponding to internal nodes.  
The two leaves surrounding a pointed or empty cell with $\iso$-label $x$ 
have larger $\iso$-labels, because the pointed cell of $T$ carrying the $\iso$-label $x$ is an ancestor of both leaves. 
So $\sigma$ is alternating. 
Moreover, $\sigma$ starts with an ascent because the first leaf is the second cell whose $\iso$-label is read when building $\sigma$ 
(the first cell read carrying as above the $\iso$-label of an ancestor of this first leaf). 

\begin{figure}[ht]
\begin{align*}
& \begin{array}{c}\includegraphics[scale=1.0]{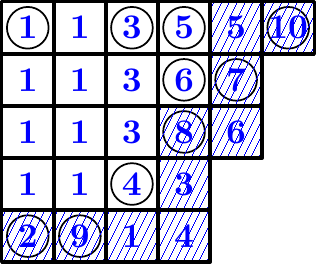}\end{array}
=
\begin{array}{c}\includegraphics[scale=1.0]{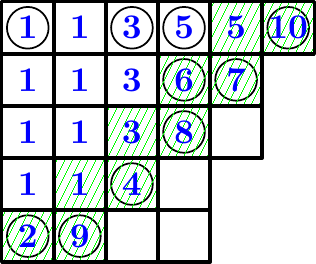}\end{array}
\mapsto 2 \ \mathbf{9} \ 1 \ \mathbf{4} \ 3 \ \mathbf{8} \ 6 \ \mathbf{7} \ 5 \ \mathbf{10} \\
& \begin{array}{c}\includegraphics[scale=1.0]{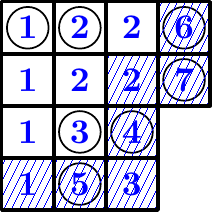}\end{array}
=
\begin{array}{c}\includegraphics[scale=1.0]{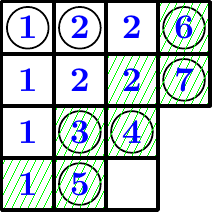}\end{array}
\mapsto
1 \ \mathbf{5} \ 3 \ \mathbf{4} \ 2 \ \mathbf{7} \ 6
\end{align*}
\caption{The two almost complete Baxter TLTs of Figure~\ref{atn} with their $\iso$-labeling, and the corresponding permutations. 
(The bold elements in permutations correspond to the leaves in the TLTs.)}\label{atn-perm}
\end{figure}

\medskip 

Conversely, let $\sigma \in \BB_n$, and let $T = \phiBB^{-1}(\sigma)$. 
$T$ is a Baxter TLT, and we want to prove that $T$ is almost-complete. 
To this end, using Lemma~\ref{lem:escalier}, it is enough to prove that the leafs of $T$ occupy all cells on the main diagonal starting at the second cell of the bottom row of $T$ 
(which we refer to as \emph{staircase shape}). 
To do so, our first step is to show that the points of $T$ with $\iso$-labels $\sigma(2i)$ (for $1 \leq i \leq \lfloor n/2 \rfloor$) are all leaves of $T$. 

For any $1 \leq i \leq \lfloor n/2 \rfloor$, it holds that $\sigma(2i-1) < \sigma(2i) > \sigma(2i+1)$ by the alternating condition 
(the second inequality being undefined in the case $n$ is even and $i=n/2$). 
Denote by $\ell$ the point of $T$ with $\iso$-label $\sigma(2i)$. 
By Proposition~\ref{prop:TTW-ascent+descent} and Corollary~\ref{cor:TTW-ascent+descent}, and according to Definition~\ref{def:ancestors}, 
$\sigma(2i-1) < \sigma(2i) > \sigma(2i+1)$ implies that $\ell$ is both column-extremal and row-extremal 
(except for the case $n$ is even and $i = n/2$, in which case $\ell$ is column-extremal and is the rightmost point in the first row). 
In particular, such points $\ell$ are leaves of $T$. 

Our next step is to show that these leaves form a staircase shape.
We start by noticing that for any point $c$ of $T$ which is both column-extremal and row-extremal, 
all cells of $T$ below and to the right of $c$ must be empty. 
Indeed, assuming that a cell $c'$ below and to the right of $c$ were pointed, 
and considering w.l.o.g. $c'$ topmost and leftmost among such cells, 
$c$, $c'$ and the parent of $c'$ would form a pattern $\Tv$ or $\Th$. 

Therefore, all pointed cells of $T$ with $\iso$-label $\sigma(2i)$ for $1 \leq i \leq \demiinf{n}$ must be in different columns and in different rows (by ``extremality''). 
Moreover, considering these points from left to right gives a sequence of points which have increasing $y$-coordinates (since they have no point below and to their right). 

As a consequence, the number of columns of $T$ is at least $1 + \demiinf{n}$, 
the $1$ accounting for the first column, and $\demiinf{n}$ accounting for the column-extremal points $\sigma(2i)$ for all $i \in [1;\demiinf{n}]$. 
Similarly, the number of rows must be at least $\demiinf{n}+1$ if $n$ is odd (resp. $\demiinf{n}$ if $n$ is even, the point with $\iso$-label $\sigma(n)$ not being row-extremal in this case), 
yielding a total of at least $n + 1$ rows and columns in total (both in the even and in the odd case). 
Since $T \in \TT_n$, we know that $T$ has \emph{exactly} $n+1$ rows and columns in total, 
implying that all the columns (except the leftmost) and all the rows (except the topmost when $n$ is odd) 
are occupied by leaves of $T$, so that the leaves of $T$ must form a staircase shape. 
Lemma~\ref{lem:escalier} then concludes the proof.
\end{proof}

From Proposition~\ref{prop:restriction_twisted_Baxter} and Corollary~\ref{cor:product_of_Catalan}, 
we immediately deduce the enumeration of $\BB_n$. 

\begin{corollary}
For any $n$, there are $C_n^2$ (resp. $C_n\cdot C_{n+1}$) permutations of size $2n$ (resp. $2n+1$)  
which avoid the patterns $3-14-2$ and $3-41-2$ 
and are alternating starting with an ascent. 
\label{cor:enum_alternating_twistedBaxter}
\end{corollary}

\subsection{Enumerative problems opened by the enumeration of alternating twisted Baxter permutations}

Corollary~\ref{cor:enum_alternating_twistedBaxter} provides an enumeration result which we have not been able to find in the literature. 
Our proof is bijective, and obtained as the result of composing two bijections:
one between $\BB_n$ and $\TT_n$ and the other one between $\TT_n$ and $\PP_n$. 
While this shows that TLT can be useful to prove meaningful results on other combinatorial objects, 
this also raises the question of whether $\BB_n$ could be enumerated directly, without appealing to TLTs. 
It is not hard to observe that, for any permutation $\sigma$ of $\BB_n$, its pattern $\sigma_{odd}$ (resp. $\sigma_{even}$) 
corresponding to the odd (resp. even) positions in $[1, n]$ 
is a permutation avoiding $312$ (resp. $231$) as classical patterns. 
And it is well-known that the families $Av(312)$ and $Av(231)$ are enumerated by the Catalan numbers. 
Therefore, considering also the enumeration obtained in Corollary~\ref{cor:enum_alternating_twistedBaxter}, 
it is tempting to conjecture that the map $\sigma \mapsto (\sigma_{odd},\sigma_{even})$ 
is a bijection between $\BB_n$ and $Av_{\lceil n/2\rceil}(312) \times Av_{\lfloor n/2 \rfloor}(231)$.
We leave this question open. 

In addition, we wish to point out that the permutations appearing in Corollary~\ref{cor:enum_alternating_twistedBaxter} 
are enumerated like the alternating Baxter permutations 
(\ie alternating permutations that avoid the patterns $2-41-3$ and $3-14-2$). 
Indeed, in~\cite{CDV}, the authors give a bijective proof that the number of alternating Baxter permutations of size $2n$ (resp. $2n+1$) 
is $C_n^2$ (resp. $C_n\cdot C_{n+1}$). 
Another combinatorial proof using \triples is given in \cite{DG}. 

Another important observation is that, unlike Baxter permutations, 
the permutations that avoid $3-14-2$ and $3-41-2$ are not stable under the reverse symmetry, 
nor under the complement symmetry. 
Therefore, Corollary~\ref{cor:enum_alternating_twistedBaxter} does not solve the enumeration of 
alternating permutations avoiding $3-14-2$ and $3-41-2$ and starting with a \emph{descent}, 
whereas the results of~\cite{CDV,DG} do solve the analogous problem for Baxter permutations. 

In view of these two very similar enumeration results, 
it is natural to look for a (hopefully simple) bijection 
between alternating Baxter permutations and alternating (inverses of) twisted Baxter permutations starting with an ascent. 
We leave this problem open, but point out one possible direction for finding such a bijection. 
In~\cite{ffno}, the authors describe several bijections between families of Baxter objects, 
and in particular a bijection $\Theta_1$ between Baxter permutations and pairs of twin binary trees, 
and a bijection $\Theta_2$ between pairs of twin binary trees and \emph{rectangulations}. 
These objects are exactly our packed floorplans, up to a rotation of $90^{\circ}$. 
The restriction of $\Theta_1$ to alternating Baxter permutations provides a bijection 
with pairs of twin binary trees with additional restrictions. 
A first task would be to examine how these restrictions are translated on the rectangulations (or packed floorplans) {\em via} $\Theta_2$, 
and then on the Baxter TLTs {\em via} $\Phi_\F^{-1}$. 
These restricted Baxter TLTs are equinumerous with the almost complete Baxter TLTs. 
It is actually possible that they are {\em exactly} the almost complete Baxter TLTs. 
If it is not the case, a second task would be to identify a bijection $\Lambda$ between these two families of TLTs. 
A bijection between alternating Baxter permutations and alternating (inverses of) twisted Baxter permutations starting with an ascent 
would then be the composition $\phiBB \circ \Lambda \circ \Phi_\F^{-1} \circ \Theta_2 \circ \Theta_1$. 
Describing this bijection directly would be the third task in this search of a \emph{simple} bijection.

\appendix
%%%%%%%%%%%%%%%%%%%%%%%%%%%%%%%%%%%%%%%%
\section*{Appendix: Size-preserving bijection between PFPs and mosaic floorplans}

Mosaic floorplans were defined as follows by Hong {\it et.al.}~\cite{hong}. 

In a rectangular partition of a rectangle, a {\em segment} is a straight line, 
not included in the boundary of the partitioned rectangle, 
that is the union of some rectangle sides, 
and is maximal for this property. 
Let us call {\em floorplans} the rectangular partitions of a rectangle such that 
every pair of segments with non-empty intersection 
forms a T-junction (defined on p.~\pageref{pagedef:T-junction}). 
Two floorplans are said $R$-equivalent if one can pass from one to the other by sliding the segments 
to adjust the sizes of the rectangles. 
A {\em mosaic floorplan} is defined as an equivalence class of floorplans under $R$. 

\smallskip

In \cite{AcBaPi}, the authors describe a bijection between mosaic floorplans and Baxter permutations, 
\ie permutations avoiding the patterns $2-41-3$ and $3-14-2$ 
(see the definition of dashed patterns in Section~\ref{sec:bax}). 
This implies that mosaic floorplans are enumerated by Baxter numbers. 
From Theorems~\ref{thm:FT} and~\ref{thm:TTW}, 
and since twisted Baxter permutations are also enumerated by Baxter numbers, 
it follows that PFPs are in size-preserving bijection with mosaic floorplans. 
This correspondence can be made more precise: 

\begin{proposition}
Every mosaic floorplans (\ie every equivalence class of floorplans under $R$) contains exactly one PFP.
\label{prop:PFP_and_mosaicFP}
\end{proposition}

\begin{proof}
Recall that there are as many mosaic floorplans of size $n$ as PFPs of size $n$ (namely, $Bax_n$). 
Thus, it is enough to prove that every mosaic floorplan contains \emph{at least} one PFP. 
To prove this statement, we show that every floorplan containing some patterns $\Fp$ 
is $R$-equivalent to a floorplan containing strictly fewer such patterns. 

Let $F$ be a floorplan containing $\Fp$. Consider two tiles $t_1$ and $t_2$ 
forming a pattern $\Fp$, 
with $t_1$ located Northwest from $t_2$, 
and such that the distance between the bottom rightmost corner of $t_1$ 
and the top leftmost corner of $t_2$ is minimal among all such patterns in $F$. 
Necessarily, the two segments meeting at the bottom rightmost corner of $t_1$ form a T-junction, of type $\TBas$ or $\TDroit$. 
We assume that it is of type $\TBas$. The case $\TDroit$ is easily deduced by symmetry 
(applying reflection along a Northwest-Southeast axis). 

To obtain a floorplan $F'$ $R$-equivalent to $F$ with fewer patterns $\Fp$, 
we slide a segment of $F$, denoted $E$, and defined as follows (see~Figure~\ref{fig:tout_PFP}$(a)$). 

Let $t_3$ be the tile located immediately to the right of $t_1$. 
By minimality of $(t_1,t_2)$, the $x$-coordinate of right side of $t_3$ 
is strictly larger than the $x$-coordinate (denoted $x_c$) of the top leftmost corner $c$ of $t_2$. 
Now, consider the stack of tiles containing $t_3$ and all the tiles stacked on 
$t_3$ such that 
\begin{itemize}
\item the $x$-coordinates of the left sides of the tiles are weakly increasing from bottom to top and are all smaller than or equal to $x_c$;
\item the $x$-coordinates of the right sides of the tiles are strictly larger than $x_c$. 
\end{itemize}
We consider the tiles of the stack whose left sides have maximal $x$-coordinate (denoted $x_E$),
and we define $E$ as the union of all these left sides. 
We claim that $E$ is a segment of $F$, 
\ie that the lower (resp. upper) extremity of $E$ is a T-junction of the form $\TBas$ (resp. $\THaut$). 
This claim is easily proved by contradiction, 
using the above definition of the stack of tiles, its maximality, 
and the fact that the T-junction at the bottom rightmost corner of $t_1$ 
(which is also the bottom leftmost corner of $t_3$) 
is of type $\TBas$. 
It follows that $E$ is also the union of the right sides of some (one or several) tiles, 
and for each of these tiles $t$, $(t,t_2)$ is a pattern $\Fp$. 
The segment $E$ may be slided to the right until $x_E > x_c$, 
to get a floorplan $F'$ which is $R$-equivalent to $F$ 
and contains strictly fewer patterns $\Fp$. 
\end{proof}

Figure~\ref{fig:tout_PFP}$(b)$ shows a floorplan and a PFP that are $R$-equivalent. 

\begin{figure}[ht]
 \centering
 \subfigure[How to pack a floorplan.] 
   { 
\centering{
\includegraphics[scale=0.85]{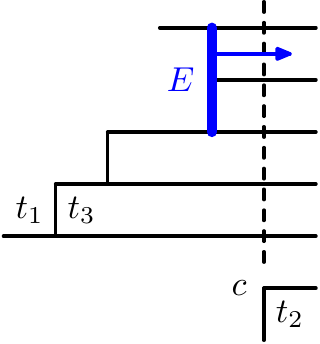}
}
   } \hspace*{1cm}
 \subfigure[A floorplan (left) and its $R$-equivalent PFP (right).] 
{
\centering{
\includegraphics[scale=0.65]{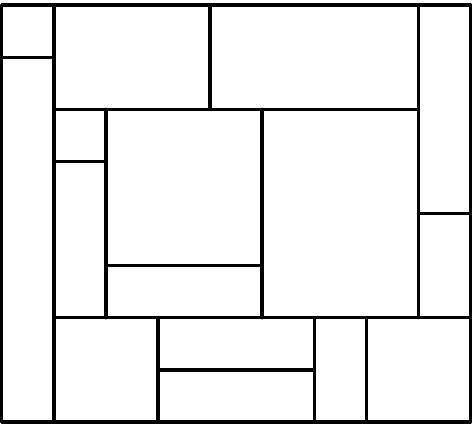}
\hspace*{0.5cm} 
\includegraphics[scale=0.65]{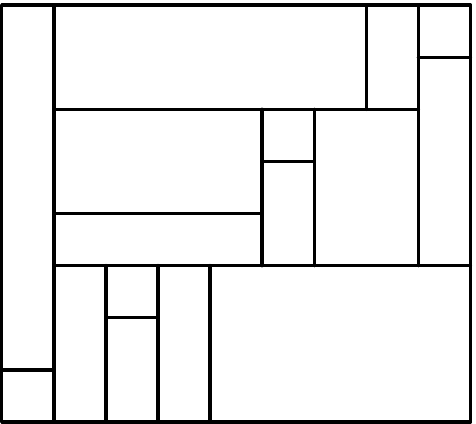}
}
}
 \caption{Every mosaic floorplan contains exactly one PFP.} \label{fig:tout_PFP}
 \end{figure}

\section*{Acknowledgments}
This research has received the support of the ANR, through the ANR -- PSYCO project (ANR-11-JS02-001).

%%%%%%%%%%%%%%%%%%%%%%%%%


\begin{thebibliography}{alpha}

\bibitem{AcBaPi}
 {\sc E. Ackerman, G. Barequet, R.~Y. Pinter},
 {\em A bijection between permutations and floorplans, and its applications},
Discrete Applied Mathematics, Vol. 154, Issue 12, 2006, 1674--1684.

\bibitem{TLT}
 {\sc J.-C. Aval,  A. Boussicault, P. Nadeau}, 
 {\em Tree-like tableaux}, 
DMTCS Proceedings FPSAC'2011, 63--74 [hal-00618274].
% Attention, ne pas remettre la référence à la version longue. C'est important pour la bijection du code. 

\bibitem{ana}
 {\sc J.-C. Aval,  A. Boussicault, M. Bouvel, M. Silimbani}, 
 {\em Combinatorics of non-ambiguous trees}, 
Advances in Applied Mathematics, Vol. 56, 2014, 78--108.

\bibitem{GBaxter}
{\sc G. Baxter},
{\em On fixed points of the composite of commuting functions},
Proceedings of the American Mathematical Society, Vol. 15, No. 6, 1964, 851--855.

\bibitem{RBaxter}
{\sc R. Baxter}, 
{\em Dichromatic polynomials and Potts models summed over rooted maps}, 
Annals of Combinatorics, Vol. 5, 2001, p. 17.

\bibitem{BBF}
{\sc N. Bonichon, M. Bousquet-M\'elou, E. Fusy},
{\em Baxter permutations and bipolar orientations},
S\'eminaire Lotharingien de Combinatoire 61A, 2010, Article B61Ah.

\bibitem{BCDK}
{\sc M. Bousquet-M\'elou, A. Claesson, M. Dukes, S. Kitaev},
{\em $(2+2)$-Free posets, ascent sequences and pattern avoiding permutations},
J. Combin. Theory Ser. A 117 no. 7 (2010) 884-909. 

\bibitem{CEF} 
{\sc T. Chow, H. Eriksson, and C. K. Fan}, 
{\em Chess tableaux}, 
Electronic J. Combin. 11(2) (2005), \#A3.
    
\bibitem{CGHK78}
{\sc F.R.K. Chung, R. Graham, V. Hoggatt, M. Kleiman},
{\em The number of {B}axter permutations},
 Journal of Combinatorial Theory, Series A, 24(3):382--394, 1978.

\bibitem{CDV} 
{\sc R. Cori, S. Dulucq, G. Viennot},
{\em Shuffle of Parenthesis Systems and Baxter Permutations}, 
J. Combin. Theory Ser. A 43 (1986), 1--22.

\bibitem{DV}
{\sc M.-P. Delest, G. Viennot},
\newblock Algebraic languages and polyominoes enumeration,
\newblock Theoretical Computer Science (vol. 34, 1984), 169--206.

\bibitem{DG96}
{\sc S. Dulucq, O. Guibert},
{\em Stack words, standard permutations, and Baxter permutations},
Discrete Mathematics, Vol. 157 (1996), 91--106.

\bibitem{DG}
{\sc S. Dulucq, O. Guibert},
{\em Baxter permutations}, 
Discrete Mathematics, 180 (1998), 143--156.

\bibitem{ffno}
{\sc S. Felsner, E. Fusy, M. Noy, D. Orden},
{\em Bijections for Baxter families and related objects}, 
Journal of Combinatorial Theory Series A (Vol. 118(3), 2011), 993--1020. 

\bibitem{GV}
{\sc I.M.~Gessel, G.~Viennot}, 
{\em Binomial determinants, paths, and hook length formulae}, 
Advances in Mathematics 58, 1985, 300--321.
% {\sc I. Gessel, X. Viennot},
% {\em Determinants, paths, and plane partitions}, 
% preprint (1989).

\bibitem{Gir}
{\sc S. Giraudo},
{\em Algebraic and combinatorial structures on Baxter permutations},
FPSAC 2011, Reykjav\'ik, Iceland, DMTCS proc., 2011, 387--398.

\bibitem{hong}
{\sc X. Hong, et al.},
{\em Conner block list: An effective and efficient topological representation of non-slicing floorplan},
 In Proceedings of the International Conference on Computer Aided Design (ICCAD '00) 8-12.

\bibitem{GuibertThese}
{\sc O. Guibert},
{\em Combinatoire des permutations \`a motifs exclus en liaison avec mots, cartes planaires et tableaux de Young},
PhD thesis (Univ. Bordeaux, 1995).

\bibitem{LR}
{\sc S. Law, N. Reading},
{\em The Hopf algebra of diagonal rectangulations},
Journal of Combinatorial Theory, Series A, 119, 2012, 788--824.

\bibitem{lind}
{\sc B. Lindstr\"om},
{\em On the vector representation of induced matroids},
 Bull. London Math. Soc., 5:85--90, 1973.

\bibitem{pos} 
{\sc A. Postnikov}, 
{\em Total positivity, Grassmannians, and networks},
 arXiv:math/0609764v1, 2006.

\bibitem{reading05}
{\sc N. Reading}, 
{\em Lattice congruences, fans and Hopf algebras}, 
J. Combin. Theory Ser. A 110 (2005) no. 2, 237--273.

\bibitem{saka}
{\sc K. Sakanushi, Y. Kajitani, D.P. Mehta}, 
{\em The quarter-state-sequence floorplan representation},
IEEE Trans. on Circuits and Systems I: Fundamental Theory and Applications, 50:3 (2003), 376--386.

\bibitem{oeis}
{\sc N.J.A. Sloane},
{\em The On-line Encyclopedia of Integer Sequences},
(2007) published electronically at \verb+www.research.att.com/~njas/sequences/+.

\bibitem{Viennot_SLC81}
{\sc G. Viennot},
{\em A bijective proof for the number of Baxter permutations}, 
Troisi\`eme S\'eminaire Lotharingien de Combinatoire, Le Klebach (1981), 28--29.

\bibitem{AT} 
{\sc X. Viennot}, 
{\em Alternative tableaux, permutations and partially asymmetric exclusion process},
Slides of a talk at the Isaac Newton Institute in Cambridge, 2008.

\bibitem{West_GT}
{\sc J. West}, 
{\em Generating trees and the Catalan and Schr\"oder numbers}, 
Discrete Mathematics, vol. 146 (1995), pp. 247--262. 

\bibitem{West}
{\sc J. West},
{\em Enumeration of Reading's twisted Baxter permutations}, 
talk presented at \emph{Permutation Patterns 2006}, 
preprint available at \texttt{http://www.cs.otago.ac.nz/staffpriv/mike/PP2006/abs/West.pdf}.

\bibitem{yao}
{\sc B. Yao, H. Chen, C.K. Cheng, R.L. Graham}, 
{\em Floorplan representations: Complexity and connections}, 
ACM Transactions on Design Automation of Electronic Systems, 8:1 (2003), 55--80.

\end{thebibliography}
\end{document}